\documentclass[12pt, reqno, a4paper]{amsart}

\usepackage[
  a4paper,
  textwidth=16cm,
  textheight=23cm,
  centering
]{geometry}
\usepackage{mathtools}
\usepackage{amssymb}
\usepackage{mathrsfs}
\usepackage{bm}
\usepackage[T1]{fontenc}
\usepackage{lmodern}
\usepackage{textcomp}
\usepackage[utf8]{inputenc}
\usepackage{xcolor}
\usepackage[shortlabels]{enumitem}
\usepackage{indentfirst}
\usepackage{graphicx}
\usepackage{url}
\usepackage{autobreak}
\usepackage[justification=justified]{caption}
\usepackage[labelformat=empty,subrefformat=parens]{subcaption}
\usepackage{float}
\usepackage{tikz}
\usepackage{tikz-cd}
\usepackage{hyperref}
\usepackage{comment}
\usepackage{array}
\usepackage{placeins}
\renewcommand{\thefootnote}{} 

\theoremstyle{plain} 
\newtheorem{theorem}{Theorem}[section]
\newtheorem{lemma}[theorem]{Lemma}
\newtheorem{corollary}[theorem]{Corollary}
\newtheorem{proposition}[theorem]{Proposition}

\theoremstyle{definition} 
\newtheorem{definition}[theorem]{Definition}
\newtheorem{remark}[theorem]{Remark}
\newtheorem{example}[theorem]{Example}

\newtheorem*{maintheorem1}{Theorem~\ref{thm:abelianization}}
\newtheorem*{maintheorem2}{Theorem~\ref{theorem_finitely_presented}}
\makeatletter
  
  \@addtoreset{equation}{section}
\makeatother

\makeatletter
\newcommand{\relationitem}[1]{%
  \item[\textup{(#1)}]%
  \phantomsection%
  \protected@edef\@currentlabel{#1}%
}
\makeatother

\DeclareRobustCommand{\relref}[1]{%
  \texorpdfstring{%
    \textup{relation~}\hyperref[#1]{\textup{(\ref*{#1})}}%
  }{%
    relation (#1)%
  }%
}

\newcommand{\cR}{\mathcal{R}}

\newcommand{\cT}{\mathcal{T}}

\newcommand{\bZ}{\mathbb{Z}}
\newcommand{\bN}{\mathbb{N}}
\newcommand{\sF}{\mathscr{F}}

\newcommand{\ov}[1]{\overline{#1}}
\newcommand{\Ker}{\operatorname{Ker}}
\allowdisplaybreaks[3] 

\newcommand\cinput[2]{\lower#1pt\hbox{\input{#2}}}

\makeatletter\let\@wraptoccontribs\wraptoccontribs\makeatother

\begin{document}

\title{Presentations of the cactus Thompson group}
\author{Yuya Kodama and Akihiro Takano}
\date{}
\renewcommand{\thefootnote}{\arabic{footnote}}  
\setcounter{footnote}{0} 

\begin{abstract}
  We give infinite and finite presentations of the cactus Thompson group.
  This group was introduced as $V_{\rm{mock}}$ by Witzel and Zaremsky, who combined Thompson's group with cactus groups.
\end{abstract}

\maketitle

\section{Introduction}
Thompson's groups $F$, $T$, and $V$ were defined by Richard Thompson in the 1960s.
While one of the original motivations for studying these groups was logic, they continue to be actively studied today due to their known connections to various fields.
One source of this versatility is the existence of several equivalent models for these groups.
In this paper, we use the model based on pairs of finite rooted binary trees.
Each element of $V$ can be represented by a pair of binary trees with the same number of leaves, together with a bijection between their sets of leaves.
The group $T$ is a subgroup of $V$ consisting of elements whose associated bijections are cyclic permutations.
Similarly, the group $F$ is a subgroup of $T$ in which the bijections are restricted to the identity permutation.
In our context, the bijection for an element of $V$ can be visualized as a collection of lines connecting the two sets of leaves.
From this perspective, one is naturally led to consider an ``Artinification'' of $V$.
That is, by replacing leaf permutations with braids, we can construct a new group $BV$.
This group was introduced independently by Brin \cite{brin2007bv, brin2006bv} and Dehornoy \cite{dehornoy2006bv}.

Witzel and Zaremsky \cite{witzel2018cloning} further developed this idea in their framework, introducing the notion of a cloning system, which provides a general framework for constructing generalized Thompson groups.
This concept was inspired directly by Brin's construction of the braided Thompson group $BV$.
In Brin's definition, the symmetric group actions appearing in the tree pair model of $V$ are replaced by braid groups.
The process of expanding trees requires a compatible way of ``duplicating'' strands in the braid.
Cloning systems axiomatize precisely this phenomenon.
Roughly speaking, a cloning system consists of a family of groups $(G_n)$ equipped with a family of cloning maps $(G_n \to G_{n+1})$, which describe how to duplicate strands in a manner analogous to expanding leaves in binary trees.
From such data, one can construct a generalized Thompson group whose elements are represented by tree pairs decorated with elements of the groups $(G_n)$.
This construction recovers the classical Thompson's group $V$ when $(G_n)$ is the family of symmetric groups, and yields braided Thompson groups when $(G_n)$ is the family of braid groups.
In this way, cloning systems provide a unified perspective that encompasses a wide range of groups, including $BV$, within a common framework.
In particular, Witzel and Zaremsky constructed a generalized Thompson group from the cloning system on cactus groups and denoted the resulting group by $V_{\rm{mock}}$.
In this paper, we study this generalized Thompson group arising from the cactus group.

The cactus group first appeared, under the name of the quasibraid group, in the work of Devadoss \cite{devadoss1999mosaic} in the study of the mosaic operad, in connection with the Deligne--Mumford compactification $\overline{M}_{0,n+1}(\mathbb{R})$ of the moduli space of configurations of $n+1$ marked points on stable real algebraic curves of genus zero.
Davis, Januszkiewicz and Scott \cite{davis2003blowup} also studied cactus groups under the name of mock reflection groups. 
Henriques and Kamnitzer \cite{henriques2006crystals} later introduced the term cactus group.
They showed that the category of crystals of a finite-dimensional complex reductive Lie algebra admits the structure of a coboundary category.
Analogously to the role of the braid group in a braided monoidal category, the cactus group acts naturally on tensor powers in a coboundary category.

Much like braid groups, cactus groups admit a convenient diagrammatic interpretation.
Elements of a cactus group can be represented by braid-like diagrams consisting of strands, where the generators correspond to reversing contiguous blocks of strands.
These diagrams can be composed by concatenation, similarly to braid diagrams, providing an intuitive way to visualize group operations.
This diagrammatic viewpoint plays an important role in understanding the structure of cactus groups and their relations, and will be useful in the construction considered in this paper.

Cactus groups have recently attracted attention from the viewpoint of geometric group theory.
Using median geometry, Genevois \cite{genevois2025cactus} proved, among other results, that cactus groups are virtually cocompact special and acylindrically hyperbolic.
These results demonstrate that cactus groups admit rich geometric and group-theoretic structures.

Since both Thompson's groups and cactus groups admit rich geometric structures, it is natural to investigate a generalized Thompson group that combines these two families.
We call this group the \textbf{cactus Thompson group} and denote it by $CV$.
As a first step, we study presentations of this group and properties of its subgroups.

Following \cite{brady2008bv}, we first give an explicit infinite presentation of $CV$; see Theorem~\ref{theorem_presentation}.
Using this presentation, we compute the abelianization of $CV$.
\begin{maintheorem1}
  The abelianization of $CV$ is isomorphic to $\bZ/2\bZ \oplus \bZ/2\bZ$.
\end{maintheorem1}

Many generalized Thompson's groups are known to satisfy strong finiteness properties \cite{witzel2018cloning}.
We also prove that $CV$ is finitely presented.

\begin{maintheorem2}
  The group $CV$ is finitely presented.
\end{maintheorem2}

In contrast to the derivation of the infinite presentation, the algebraic proof of finite presentability given in this paper is considerably more involved.
One source of difficulty is the difference between the standard generators of braid groups and those of cactus groups.
Whereas the standard braid generators involve adjacent strands, a standard generator $s_{i, j}$ of a cactus group simultaneously reverses an arbitrary interval of strands.
As a result, more types of relations must be considered than in the case of the braided Thompson group, making the argument more complicated.
The main part of the proof is therefore to show that all instances of the infinitely many defining relations follow from finitely many bounded-index instances.
We also observe some properties of subgroups of $CV$; see Proposition \ref{prop_PCV_nonfg} and Corollaries \ref{cor_CV_odd_torsion}, \ref{cor_CV_nonsplit}, and Proposition \ref{prop_CV_center}.

This paper is organized as follows:
In Section \ref{section_preliminaries}, we recall the definitions of Thompson's groups and cactus groups.
In Subsection \ref{subsection_def_CV}, we give two (equivalent) definitions of the cactus Thompson group $CV$, one using binary trees and the other using a cloning system.
In Subsection \ref{subsection_subgroups_CV}, we study some basic properties of subgroups of $CV$.
In Section \ref{section_infinitely_presented}, we give an infinite presentation of $CV$ and compute its abelianization.
In Section \ref{section_finitely_presented}, we show that $CV$ is finitely presented by propagating the defining relations from finitely many bounded-index instances.
\section{Preliminaries} \label{section_preliminaries}
In this section, we recall the definitions and properties of Thompson's groups and the cactus group.
\subsection{Thompson's groups}

\subsubsection{Definitions}
We summarize a definition of Thompson's group $V$, and its subgroups $F$ and $T$.
It is known that these groups have several (equivalent) definitions, see \cite{cannon1996introductory} for instance.
In this paper, we use binary trees to define these groups.

A \textbf{rooted binary tree} is either a single vertex or a finite tree in which every vertex has degree three, except the root, which has degree two, and the leaves, which have degree one.
We draw such trees with the root at the top and leaves arranged from left to right along a horizontal line at the bottom with edges directed downward.
A rooted binary tree is \textbf{trivial} if it has only one vertex.
A \textbf{caret} is a rooted binary tree which has three vertices.
Namely, a caret has two leaves, called the \textbf{left child} and \textbf{right child}.
An edge connecting the left (resp.~right) child and the root of a caret is called the \textbf{left}
(resp.~\textbf{right}) \textbf{edge}.
By attaching the root of a caret to a leaf of a rooted binary tree, we obtain a new rooted binary tree.
This operation is called an \textbf{attachment}.
If a rooted binary tree $T$ is obtained from another rooted binary tree $S$ by attaching some carets, then $T$ is called an \textbf{expansion} of $S$.
We number the leaves of a rooted binary tree with $n$ leaves from 1 to $n$ in order from left to right, see Figure \ref{tree_example}.
Such a tree is obtained by iterating attachments $n-1$ times.

\begin{figure}[tbp]
  \begin{center}
    \includegraphics[width=0.7\linewidth]{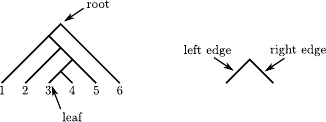}
  \end{center}
  \caption{A rooted binary tree and a caret}
  \label{tree_example}
\end{figure}

A \textbf{tree pair} is a triple $(T_+, \sigma, T_-)$ consisting of two rooted binary trees $T_+$ and $T_-$ with the same number of leaves $n$, and an element $\sigma$ in the symmetric group $S_n$.
The element $\sigma$ permutes the set of leaf numbers $\{1, 2, \ldots, n\}$.
When drawing a tree pair $(T_+, \sigma, T_-)$, we put $T_-$ upside down under $T_+$, and connect the $i$-th leaf of $T_+$ and the $\sigma(i)$-th leaf of $T_-$ with an edge for each $1 \leq i \leq n$.
This diagram is called a \textbf{tree diagram} of $(T_+, \sigma, T_-)$.
Let $\cT_{\text{symm}}$ be the set of all tree pairs.
An \textbf{expansion} is an operation on $\cT_{\text{symm}}$ defined as follows:
let $(T_+, \sigma, T_-)$ be a tree pair with $n$ leaves.
Attaching a caret to the $i$-th leaf of $T_+$ (resp.~$\sigma(i)$-th leaf of $T_-$), we have a rooted binary tree $T'_+$ (resp.~$T'_-$) with $n+1$ leaves.
Define a permutation $\sigma'$ in $S_{n+1}$ by
\begin{align*}
  \sigma'(j) = \begin{cases}
                 \sigma(j)     & (j \leq i, \sigma(j) \leq \sigma(i)), \\
                 \sigma(j)+1   & (j < i, \sigma(j) > \sigma(i)),       \\
                 \sigma(j-1)   & (j > i, \sigma(j-1) < \sigma(i)),     \\
                 \sigma(j-1)+1 & (j > i, \sigma(j-1) \geq \sigma(i)).
               \end{cases}
\end{align*}
Then we obtain a new tree pair $(T'_+, \sigma', T'_-)$ with $n+1$ leaves.
Diagrammatically, the edge connecting the $i$-th leaf of $T_+$ and the $\sigma(i)$-th leaf of $T_-$ is parallelized,  and the left (resp.~right) children of each attached caret  are connected by an edge.
The inverse operation of an expansion is called a \textbf{reduction}.
A tree pair is \textbf{reduced} if it admits no reduction.
Then we define an equivalence relation $\sim$ generated by an expansion and reduction on $\cT_{\text{symm}}$.
Figure \ref{tree_pair_expansion} is an example of an expansion and reduction.

\begin{figure}[tbp]
  \begin{center}
    \includegraphics[width=0.85\linewidth]{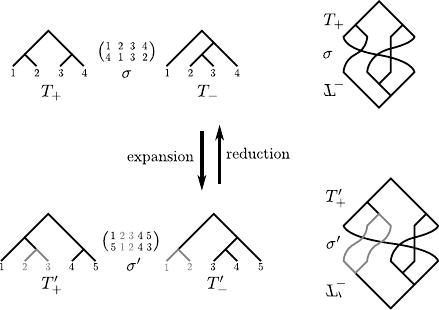}
  \end{center}
  \caption{A tree pair $(T_+, \sigma, T_-)$ (upper left) and its tree diagram (upper right). The rooted binary tree $T'_+$ (resp.~$T'_-$) is obtained by adding a caret to the leaf 2 of $T_+ $(resp.~leaf $1$ of $T_-$). The tree diagram of $(T'_+, \sigma', T'_-)$ is obtained from that of $(T_+, \sigma, T_-)$ by parallelizing the edge connecting the leaf 2 of $T_+$ and the leaf 1 of $T_-$.}
  \label{tree_pair_expansion}
\end{figure}

\textbf{Thompson's group} $V$ is a group $(\cT_{\text{symm}} /{\sim}, \cdot)$ where $\cdot$ is the following product:
let $(S_+, \sigma, S_-)$ and $(T_+, \tau, T_-)$ be in $\cT_{\text{symm}}$.
By expansion or reduction, we obtain tree pairs $(S'_+, \sigma', S'_-)$ and $(T'_+, \tau', T'_-)$ such that $(S_+, \sigma, S_-) \sim (S'_+, \sigma', S'_-)$ and $(T_+, \tau, T_-) \sim (T'_+, \tau', T'_-)$, and satisfying $S'_- = T'_+$.
Then the product of the equivalence classes of $(S_+, \sigma, S_-)$ and $(T_+, \tau, T_-)$ is the equivalence class of $(S'_+, \sigma' \tau', T'_-)$.
The identity element of $V$ is represented by $(T, 1, T)$, where $T$ is any rooted binary tree.
For an equivalence class of $(T_+, \sigma, T_-)$, its inverse element is the equivalence class of $(T_-, \sigma^{-1}, T_+)$.
The group $T$ is a subgroup of $V$ whose elements are equivalence classes of tree pairs of the form $(T_+, c, T_-)$, where $c$ is a cyclic permutation.
The subgroup $F (\subseteq T \subseteq V)$ consists of those elements whose permutations are trivial.
Thus any element of $F$ is simply represented by $(T_+, T_-)$, and is also called a tree pair.
The groups $F$ and $T$ are also called Thompson's groups.

There is a convenient method for computing the product of $V$ by deforming tree diagrams.
Let $(S_+, \sigma, S_-)$ and $(T_+, \tau, T_-)$ be tree pairs.
Put the tree diagram of $(T_+, \tau, T_-)$ immediately under the one of $(S_+, \sigma, S_-)$, and connect the root of $S_-$ and the root of $T_+$ by an edge.
We apply the following deformations
\begin{align*}
  \scalebox{1.5}{\cinput{0}{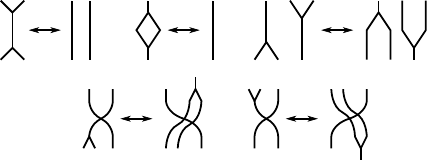}}
\end{align*}
to this diagram and then we obtain a new tree diagram.
In fact this is a tree diagram of the product $(S_+, \sigma, S_-) \cdot (T_+, \tau, T_-)$.
Figure \ref{V_product_diagram} is an example of this method.

\begin{figure}[tbp]
  \begin{center}
    \includegraphics[width=0.9\linewidth]{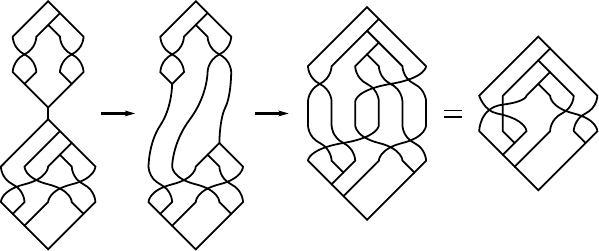}
  \end{center}
  \caption{A computation of the product in $V$}
  \label{V_product_diagram}
\end{figure}

\begin{remark}
  Thompson's group $F$ is understood as a subgroup of the homeomorphism group on the closed unit interval $[0,1]$ such that each element satisfies the following:
  \begin{enumerate}
    \item an orientation-preserving piecewise linear map,
    \item in each linear part, its slope is a power of 2, and
    \item each breakpoint is in $\bZ[1/2] \times \bZ[1/2]$.
  \end{enumerate}
  Also, the group $T$ is regarded as a subgroup of the homeomorphism group on the circle $S^1$ whose elements satisfy the above three conditions.
  The group $V$ can be regarded as the group of right-continuous bijections of $[0,1]$ that are piecewise linear on finitely many dyadic subintervals, with slopes powers of $2$ and discontinuities at dyadic rational points.
  Two rooted binary trees of each element in these groups correspond to the partitions of the domain and codomain, respectively.
  Using binary words, $V$ (and thus $F$ and $T$) can be a subgroup of the homeomorphism group on the Cantor set.
\end{remark}

\subsubsection{Presentations} \label{sec_V_presentation}
We provide a summary of both infinite and finite presentations of V.
In particular, a generating set for $V$ can be obtained by decomposing a given tree pair into a product of three tree pairs.
This method was used in \cite{burillo2009metric} for Thompson's group $T$ and in \cite{brady2008bv} for the braided Thompson group $BV$.

There are well-known presentations of $F$:
\begin{align*}
  F & \cong \langle x_0, x_1, x_2, \dots \mid x_i^{-1} x_j x_i=x_{j+1}\ (i<j) \rangle                     \\
    & \cong \langle x_0, x_1 \mid [x_0x_1^{-1}, x_0^{-1}x_1x_0], [x_0x_1^{-1}, x_0^{-2}x_1x_0^2] \rangle,
\end{align*}
where $[x,y]=xyx^{-1}y^{-1}$, and $x_i$ is a tree pair defined as
\begin{align*}
  x_i = \scalebox{2.5}{\cinput{16}{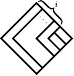}}.
\end{align*}
Moreover any element of $F$ has a unique normal form with respect to the infinite presentation above.
Every element of $F$ admits an expression of the form
\begin{equation*}
  x_{i_1}^{a_1} x_{i_2}^{a_2} \cdots x_{i_{l}}^{a_{l}} x_{j_m}^{-b_m} \cdots x_{j_2}^{-b_2} x_{j_1}^{-b_1},
\end{equation*}
where $0 \leq i_1 < i_2 < \cdots < i_{l}, 0 \leq j_1 <  j_2 < \cdots < j_m$ and $a_1, \ldots a_{l}, b_1, \ldots b_m$ are positive integers.
This expression is sometimes called a seminormal form, and is unique if we require an additional condition:
if both $x_i$ and $x_i^{-1}$ appear in this expression, then either $x_{i+1}$ or $x_{i+1}^{-1}$ (or both) also appears.
A (semi)normal form which has no generators with negative exponents is called a \textbf{positive word}.
Similarly, a \textbf{negative word} is a (semi)normal form containing no generators with positive exponents.
A (semi)normal form of any tree pair $(T_+, T_-)$ is given as follows:
only here, suppose that the leaves of $T_+$ are numbered from left to right beginning with zero.
Then the exponent $a_i$ of the generator $x_i$ in a (semi)normal form of $(T_+, T_-)$ coincides with the number of the left edges of carets on the path from the $i$-th leaf of $T_+$ to the root which do not reach the right-most path of $T_+$.
The exponent $b_i$ of $x_i^{-1}$ is obtained from $T_-$ similarly.
Figure \ref{F_exponent} is an example of this correspondence.

\begin{figure}[tbp]
  \begin{center}
    \includegraphics[width=0.4\linewidth]{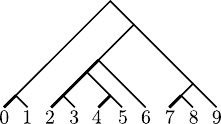}
  \end{center}
  \caption{For a binary rooted tree $T$ pictured above, the tree pair $(T, R_{10})$ is written as the word $x_0 x_2^3 x_4 x_7$. The exponent of each generator coincides with the number of thick edges in the path from the root to the corresponding leaf.}
  \label{F_exponent}
\end{figure}

Let $R_1$ be the trivial rooted binary tree and let $R_2$ be a caret.
For $i \geq 3$, define a rooted binary tree $R_i$ obtained by attaching a caret to the right-most leaf of $R_{i-1}$.
Namely $R_i$ is of the form
\begin{align*}
  R_i = \scalebox{2}{\cinput{12}{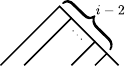}},
\end{align*}
and is called the \textbf{all-right binary tree} of $i$ leaves.
Given a tree pair $(T_+, T_-)$, consider a decomposition
\begin{align*}
  (T_+, T_-) = (T_+, R_n) \cdot (R_n, T_-) = (T_+, R_n) \cdot (T_-, R_n)^{-1},
\end{align*}
where $n$ is the number of leaves of $T_+$.
Then we can see that tree pairs $(T_+, R_n)$ and $(R_n, T_-)$ give the \textit{positive part} and \textit{negative part} of its (semi)normal form, respectively.
A reduced tree pair corresponds to normal form.

The argument above extends to Thompson's group $V$.
We decompose a tree pair $(T_+, \sigma, T_-)$ in $V$ into a product of three elements: the positive part, negative part, and the permutation part, that is,
\begin{align*}
  (T_+, \sigma, T_-) = (T_+, 1, R_n) \cdot (R_n, \sigma, R_n) \cdot (R_n, 1, T_-),
\end{align*}
where $n$ is the number of leaves of $T_+$.
Similarly to $F$, the positive part $(T_+, 1, R_n)$ (resp.~negative part $(R_n, 1, T_-)$) is written as a positive word (resp.~negative word) on the set $X \coloneq \{ x_0, x_1, x_2, \ldots \}$.
In order to represent a permutation part $(R_n, \sigma, R_n)$ by a word on some set, define elements $\sigma_i$ and $\tau_i\ (i \geq 1)$ of $V$ as follows:
\begin{align*}
   & \sigma_i \coloneq (R_{i+2}, s_i, R_{i+2}) = \scalebox{2.3}{\cinput{25}{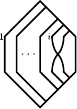}}\ ,
   & \tau_i \coloneq (R_{i+1}, s_i, R_{i+1}) = \scalebox{2.3}{\cinput{22}{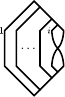}},
\end{align*}
where $s_i = \begin{pmatrix} i& i+1 \end{pmatrix}$ is the standard generator of the symmetric group.
Then the set $\{\sigma_1, \ldots, \sigma_{n-2}, \tau_{n-1}\}$ generates the copy of $S_n$ in $V$.
The elements of this set are called the \textbf{$S_n$ generators}.
Any element of $V$ represented by a tree pair of the form $(R_n, \sigma, R_n)$ is written as a word in the $S_n$ generators.
From the argument above, we have the following:

\begin{proposition} [cf. {\cite[Proposition 2.1]{brady2008bv}}]
  Thompson's group $V$ is generated by the elements $x_i\ (i \geq 0), \sigma_i\ (i \geq 1)$, and $\tau_i\ (i \geq 1)$.
\end{proposition}

\begin{proposition}
  Every element $(T_+, \sigma, T_-)$ of $V$ with $n$ leaves is written as a word of the form $w_1 w_2 w_3^{-1}$ where
  \begin{itemize}
    \item $w_1 = x_{i_1}^{a_1} x_{i_2}^{a_2} \cdots x_{i_{l}}^{a_{l}}$ with $0 \leq i_1 < i_2 < \cdots < i_{l}$ and $a_k \geq 1$ for all $1 \leq k \leq l$,
    \item $w_3 = x_{j_1}^{b_1} x_{j_2}^{b_2} \cdots x_{j_m}^{b_m}$ with $0 \leq j_1 < j_2 < \cdots < j_m$ and $b_k \geq 1$ for all $1 \leq k \leq m$,
    \item $w_2$ is a word in the $S_n$ generators.
  \end{itemize}
 \end{proposition}
  Moreover, we obtain infinite and finite presentations of $V$ as follows:

  \begin{proposition} [cf. {\cite[Theorem 2.4]{brady2008bv}}]
    Thompson's group $V$ admits the following infinite presentation$:$
    \begin{description}
      \item[Generators]\
            \begin{itemize}
              \item $x_i \quad (i \geq 0);$
              \item $\sigma_i \quad (i \geq 1);$
              \item $\tau_i \quad (i \geq 1)$.
            \end{itemize}
      \item[Relations]\
            \begin{enumerate}
              \relationitem{A} $x_j x_i = x_i x_{j+1} \quad (0 \leq i < j);$
              \relationitem{B0} $\sigma_i^2 = \tau_i^2 = 1 \quad (i \geq 1);$
              \relationitem{B1} $\sigma_i \sigma_j = \sigma_j \sigma_i \quad (j-i \geq 2);$
              \relationitem{B2} $\sigma_i \sigma_{i+1} \sigma_i = \sigma_{i+1} \sigma_i \sigma_{i+1} \quad (i \geq 1);$
              \relationitem{B3} $\sigma_i \tau_j = \tau_j \sigma_i \quad (j-i \geq 2);$
              \relationitem{B4} $\sigma_i \tau_{i+1} \sigma_i = \tau_{i+1} \sigma_i \tau_{i+1} \quad (i \geq 1);$
              \relationitem{C1} $\sigma_i x_j = x_j \sigma_i \quad (i < j);$
              \relationitem{C2} $\sigma_i x_j = x_{i-1} \sigma_{i+1} \sigma_i \quad (i = j);$
              \relationitem{C3} $\sigma_i x_j = x_{j+1} \sigma_i \sigma_{i+1} \quad (i = j+1);$
              \relationitem{C4} $\sigma_i x_j = x_j \sigma_{i+1} \quad (i \geq j+2);$
              \relationitem{D1} $\tau_i x_j = \sigma_i \tau_{i+1} \quad (i=j+1);$
              \relationitem{D2} $\tau_i x_j = x_j \tau_i \quad (i \geq j+2);$
              \relationitem{D3} $\tau_i = x_{i-1} \tau_{i+1} \sigma_i \quad (i \geq 1)$.
            \end{enumerate}
    \end{description}
  \end{proposition}

  \begin{proposition} [cf. {\cite[Theorem 3.1]{brady2008bv}}]
    Thompson's group $V$ admits a finite presentation with generators $x_0, x_1, \sigma_1, \tau_1$ and relations
  \begin{enumerate}
    \relationitem{a} $x_2 x_0 = x_0 x_3,\ x_3 x_1 = x_1 x_4$,
    \relationitem{b0} $\sigma_1^2 = \tau_1^2 = 1$,
    \relationitem{b1} $\sigma_1 \sigma_3 = \sigma_3 \sigma_1$,
    \relationitem{b2} $\sigma_1 \sigma_2 \sigma_1 = \sigma_2 \sigma_1 \sigma_2$,
    \relationitem{b3} $\sigma_1 \tau_3 = \tau_3 \sigma_1$,
    \relationitem{b4} $\sigma_1 \tau_2 \sigma_1 = \tau_2 \sigma_1 \tau_2$,
    \relationitem{c1} $\sigma_1 x_2 = x_2 \sigma_1,\ \sigma_1 x_3 = x_3 \sigma_1,\ \sigma_2 x_3 = x_3 \sigma_2,\ \sigma_2 x_4 = x_4 \sigma_2$,
    \relationitem{c3} $\sigma_1 x_0 = x_1 \sigma_1 \sigma_2,\ \sigma_2 x_1 = x_2 \sigma_2 \sigma_3$,
    \relationitem{c4} $\sigma_2 x_0 = x_0 \sigma_3,\ \sigma_3 x_1 = x_1 \sigma_4$,
    \relationitem{d1} $\tau_1 x_0 = \sigma_1 \tau_2,\ \tau_2 x_1 = \sigma_2 \tau_3$,
    \relationitem{d2} $\tau_2 x_0 = x_0 \tau_3,\ \tau_3 x_1 = x_1 \tau_4$,
  \end{enumerate}
  where $x_{i+2} = x_i^{-1} x_{i+1} x_i\ (i \geq 0), \sigma_{i+1} = x_{i-1}^{-1} \sigma_i x_i \sigma_i\ (i \geq 1)$, and $\tau_{i+1} = x_{i-1}^{-1} \tau_i \sigma_i\ (i \geq 1)$.
\end{proposition}

\subsection{Cactus groups} \label{sec_cactus}

Let $n \geq 2$ be an integer.
The \textbf{cactus group} $J_n$ of degree $n$ is defined by the generators $s_{i,j}\ (1 \leq i < j \leq n)$ with the following relations:
\begin{itemize}
  \item $s_{i,j}^2 = 1 \quad (1 \leq i < j \leq n)$;
  \item $s_{i,j} s_{k,l} = s_{k,l} s_{i,j} \quad (1 \leq i < j < k < l \leq n)$;
  \item $s_{i,j} s_{k,l} = s_{i+j-l,i+j-k} s_{i,j} \quad (1 \leq i  \leq k < l \leq j \leq n)$.
\end{itemize}

\begin{figure}[tbp]
  \begin{center}
    \includegraphics[width=0.25\linewidth]{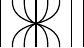}
  \end{center}
  \caption{A generator $s_{2,6}$ of the cactus group $J_8$}
  \label{cactus_example}
\end{figure}

\begin{figure}[tbp]
  \begin{minipage}[b]{0.47\columnwidth}
    \begin{center}
      \includegraphics[width=0.7\linewidth]{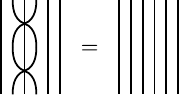}
    \end{center}
    \caption*{$s_{2,4}^2 = 1$}
  \end{minipage}
  \vspace{2em}
  \hfill
  \begin{minipage}[b]{0.47\columnwidth}
    \begin{center}
      \includegraphics[width=0.7\linewidth]{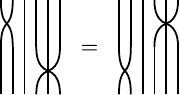}
    \end{center}
    \caption*{$s_{1,2} s_{4,6} = s_{4,6} s_{1,2}$}
  \end{minipage}
  \begin{minipage}[b]{0.47\columnwidth}
    \begin{center}
      \includegraphics[width=0.7\linewidth]{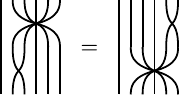}
    \end{center}
    \caption*{$s_{2,6} s_{2,3} = s_{5,6} s_{2,6}$}
  \end{minipage}
  \caption{Relations of the cactus group $J_6$}
  \label{cactus_relation}
\end{figure}

Every generator of the cactus group can be represented by a \textit{braid-like} diagram as follows:
the generator $s_{i,j}$ of $J_n$ is a \textit{braid} on $n$ strands where the strands $i, i+1, \ldots, j$ are gathered at a single point, and then their order is reversed.
Figure \ref{cactus_example} is an example of a generator.
Given two elements $s, t \in J_n$, their product $st$ is represented by a diagram obtained by concatenating a diagram representing $t$ below a diagram representing $s$.
Figure \ref{cactus_relation} consists of diagrams corresponding to the three types of relations of the cactus group.
We set $J_1\coloneqq\{1\}$.

From the description above, we have a natural homomorphism
\begin{align*}
  \pi_n \colon J_n \to S_n;\ s_{i,j} \mapsto \begin{pmatrix}
                                               i & i+1 & \cdots & j-1 & j \\
                                               j & j-1 & \cdots & i+1 & i
                                             \end{pmatrix}.
\end{align*}
It is easily seen that the map $\pi_n$ is surjective.
Its kernel is called the \textbf{pure cactus group} $PJ_n$.
Namely, the following exact sequence holds:
\begin{equation*}
  \begin{tikzcd}
    1 \arrow[r] & PJ_n \arrow[r] & J_n \arrow[r, "\pi_n"] & S_n \arrow[r] & 1.
  \end{tikzcd}
\end{equation*}
By (1) of the proposition below, this exact sequence does not split for $n \geq 3$.

\begin{proposition} [{\cite{bellingeri2024cactus}}] \label{cactus_property}
  The following hold$:$
  \begin{enumerate}
    \item The cactus group has no odd torsion.
    \item The cactus group $J_{2^k}$ has torsion of order $2^k$ for every integer $k \geq 1$.
    \item Any finite subgroup of the cactus group is a $2$-group, that is, its order is a power of 2.
    \item The pure cactus group is torsion-free.
    \item The center of the cactus group $J_n$ $($resp.~pure cactus group $PJ_n$$)$ is trivial for $n>2$ $($resp.~$n>3$$)$.
  \end{enumerate}
\end{proposition}

\begin{remark}
  The pure cactus group $PJ_n$ is isomorphic to the fundamental group of the Deligne--Mumford compactification of the moduli space of real genus-zero curves with $n+1$ marked points \cite[Theorem~9]{henriques2006crystals}. 
\end{remark}

More algebraic and geometric properties for the (pure) cactus group are investigated in \cite{bellingeri2024cactus} and \cite{genevois2025cactus}, for instance.

\section{Definitions of the cactus Thompson group} \label{section_CactusThompson}
\subsection{Definitions} \label{subsection_def_CV}

\subsubsection{Diagrammatic definition}
A triple $(T_+, s, T_-)$ is called a \textbf{cactus-tree pair} if $T_+$ and $T_-$ are rooted binary trees with the same number of leaves $n$ and $s$ is an element of the cactus group $J_n$.
Let $\cT_{\text{cac}}$ be the set of all cactus-tree pairs.
A \textbf{cactus-tree diagram} of a cactus-tree pair is defined in a similar way to a tree diagram.

Given a cactus-tree pair $(T_+, s, T_-)$, we attach a caret to the $i$-th leaf of $T_+$ (resp.~the $\overline{s}(i)$-th leaf of $T_-$), and obtain a rooted binary tree $T'_+$ (resp.~$T'_-$) with $n+1$ leaves, where $\overline{s} \coloneq \pi_n (s)$.
Suppose $s = s_{j_1,k_1} \cdots s_{j_l,k_l} \in J_n$.
Then we inductively define an element $s' \coloneq s'_1 \cdots s'_l$ in $J_{n+1}$ as follows:
first, $s'_1$ is defined by
\begin{align*}
  s'_1 \coloneq
  \begin{cases}
    s_{j_1,k_1}            & (k_1<i),               \\
    s_{i,i+1}s_{j_1,k_1+1} & (j_1 \leq i \leq k_1), \\
    s_{j_1+1,k_1+1}        & (i < j_1).
  \end{cases}
\end{align*}
If $s'_m\ (1 \leq m < l)$ is defined, then we define $s'_{m+1}$ by
\begin{align*}
  s'_{m+1} \coloneq
  \begin{cases}
    s_{j_{m+1},k_{m+1}}                                                                                                  & (k_{m+1} < \overline{s_{j_1,k_1} \cdots s_{j_m,k_m}}(i)),                 \\
    s_{\overline{s_{j_1,k_1} \cdots s_{j_m,k_m}}(i),\overline{s_{j_1,k_1} \cdots s_{j_m,k_m}}(i)+1}s_{j_{m+1},k_{m+1}+1} & (j_{m+1} \leq \overline{s_{j_1,k_1} \cdots s_{j_m,k_m}}(i) \leq k_{m+1}), \\
    s_{j_{m+1}+1,k_{m+1}+1}                                                                                              & (\overline{s_{j_1,k_1} \cdots s_{j_m,k_m}}(i) < j_{m+1}).
  \end{cases}
\end{align*}
We obtain a new cactus-tree pair $(T'_+, s', T'_-)$, which is called an \textbf{expansion} of $(T_+, s, T_-)$.
Conversely, $(T_+, s, T_-)$ is called a \textbf{reduction} of $(T'_+, s', T'_-)$.
A cactus-tree pair is \textbf{reduced} if it admits no reduction.
They are operations on $\cT_{\text{cac}}$, and define an equivalence relation $\sim$ on it generated by them.
Figure \ref{CV_element_example} is an example of an expansion.

\begin{figure}[tbp]
  \begin{center}
    \includegraphics[width=0.85\linewidth]{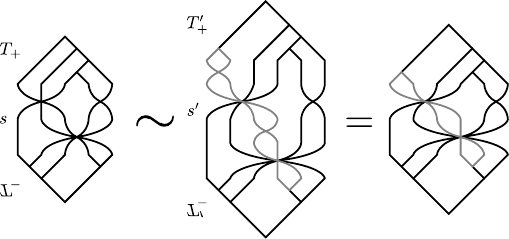}
  \end{center}
  \caption{A cactus-tree pair $(T'_+, s', T'_-)$ is an expansion of $(T_+, s, T_-)$ obtained by adding a caret to a leaf 1 of $T_+$ (resp.~leaf 4 of $T_-$). Since $s = s_{1,3} s_{4,5} s_{2,5}$, we have $s' = s_{1,2} s_{1,4} s_{5,6} s_{3,4} s_{2,6} = s_{1,4} s_{5,6} s_{2,6}$.}
  \label{CV_element_example}
\end{figure}

The \textbf{cactus Thompson group} $CV$ is a group $(\cT_{\text{cac}} /{\sim}, \cdot)$, where a product $\cdot$ is defined similarly to $V$.
Namely, let $(S_+, s, S_-)$ and $(T_+, t, T_-)$ be in $\cT_{\text{cac}}$.
By expansion or reduction, we obtain tree pairs $(S'_+, s', S'_-)$ and $(T'_+, t', T'_-)$ such that $(S_+, s, S_-) \sim (S'_+, s', S'_-)$ and $(T_+, t, T_-) \sim (T'_+, t', T'_-)$, and satisfying $S'_- = T'_+$.
Then the product of the equivalence classes of $(S_+, s, S_-)$ and $(T_+, t, T_-)$ is the equivalence class of $(S'_+, s' t', T'_-)$.
Similarly to $V$, the identity element of $CV$ is represented by $(T, 1, T)$, where $T$ is any rooted binary tree.
For an equivalence class of $(T_+, s, T_-)$, its inverse element is the equivalence class of $(T_-, s^{-1}, T_+)$.

Moreover we can compute the product diagrammatically by applying the deformations
\begin{align*}
  \scalebox{1.5}{\cinput{0}{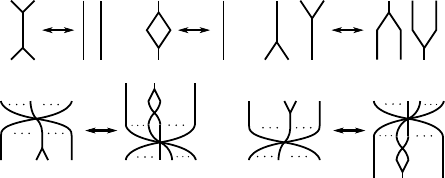}}
\end{align*}
to cactus-tree diagrams.
Figure \ref{CV_product_example} is an example of a diagrammatic computation.

\begin{figure}[tbp]
  \begin{center}
    \includegraphics[width=0.85\linewidth]{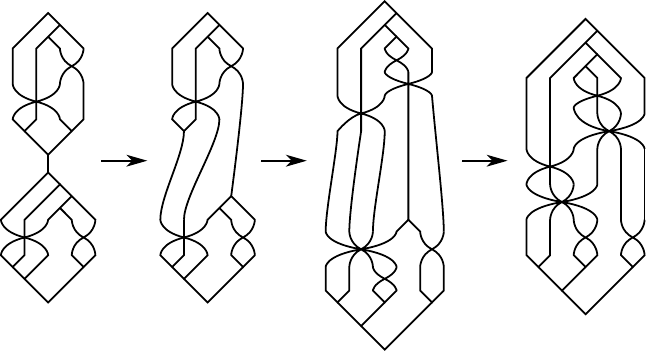}
  \end{center}
  \caption{A computation of the product in $CV$}
  \label{CV_product_example}
\end{figure}

\subsubsection{Definition by the cloning system of Witzel--Zaremsky}
Witzel and Zaremsky \cite{witzel2018cloning} defined the cloning system in order to define generalized Thompson's groups such as the braided Thompson group $BV$.
This is motivated by Brin's definition of $BV$.

Let $(G_n)$ be a direct system of groups with an injective homomorphism $\iota_n \colon G_n \to G_{n+1}$ for each $n \in \bN$, and $G = \varinjlim G_n$ its direct limit.
Then the generalized Thompson group $\cT(G_*)$ is a certain subgroup of the group of right fractions of the Brin--Zappa--Sz\'ep (BZS) product $G \bowtie \sF$, where $\sF$ is the monoid of binary forests.
The BZS product $G \bowtie \sF$ is defined by two actions $G \times \sF \to G$ and $G \times \sF \to \sF$ satisfying certain conditions.
See \cite{witzel2018cloning} for details.

The family of symmetric groups $(S_n)$ forms a direct system where the injective homomorphism $\lambda_n \colon S_n \to S_{n+1}$ is induced by the inclusion $\{ 1, \ldots, n \} \hookrightarrow \{ 1, \ldots, n ,n+1\}$.
The infinite symmetric group $S_{\infty}$ is regarded as the direct limit $\varinjlim S_n$.
Define a map $\varsigma_k^n \colon S_n \to S_{n+1}$ for $n \in \bN$ and $1 \leq k \leq n$ by
\begin{align*}
  ((\sigma)\varsigma_k^n)(i) \coloneq
  \begin{cases}
    \sigma(i)     & (i \leq k, \sigma(i) \leq \sigma(k)), \\
    \sigma(i)+1   & (i < k, \sigma(i) > \sigma(k)),       \\
    \sigma(i-1)   & (i > k, \sigma(i-1) < \sigma(k)),     \\
    \sigma(i-1)+1 & (i > k, \sigma(i-1) \geq \sigma(k)).
  \end{cases}
\end{align*}
Here these maps $\lambda_n$ and $\varsigma_k^n$ are written on the right.
They satisfy the equation
\begin{align*}
  \lambda_n \circ \varsigma_k^{n+1} =\varsigma_k^n \circ \lambda_{n+1}
\end{align*}
for any $n \in \bN$ and $1 \leq k \leq n$.

\begin{definition}
  Let $\rho_n \colon G_n \to S_n$ be a group homomorphism and $\kappa_k^n \colon G_n \to G_{n+1}$ an injective map for each $n \in \bN$ and $1 \leq k \leq n$.
  A \textbf{cloning system} is the quadruple $((G_n), (\iota_n), (\rho_n), (\kappa_k^n))$ satisfying the following conditions:
  \begin{enumerate}
    \item $\rho_{n+1} ((g)\iota_n) = (\rho_n(g)) \iota_n$ for all $n \in \bN$ and $g \in G_n$,
    \item $\iota_n \circ \kappa_k^{n+1} =\kappa_k^n \circ \iota_{n+1}$ for all $n \in \bN$ and $1 \leq k \leq n$,
    \item $(gh)\kappa_k^n = (g)\kappa_{\rho_n(h)(k)}^n (h)\kappa_k^n$ for all $n \in \bN, 1 \leq k \leq n$, and $g,h \in G_n$,
    \item $\kappa_l^n \circ \kappa_k^{n+1} = \kappa_k^n \circ \kappa_{l+1}^{n+1}$ for all $n \in \bN$ and $1 \leq k < l \leq n$,
    \item $\rho_{n+1}((g)\kappa_k^n)(i) = (\rho_n(g))\varsigma_k^n(i)$, for all $n \in \bN, 1 \leq k \leq n, g \in G_n$, and $i \neq k, k+1$,
  \end{enumerate}
  where we write the maps $\iota_n$ and $\kappa_k^n$ on the right.
  The maps $\kappa_k^n \colon G_n \to G_{n+1}$ are called the \textbf{cloning maps} of the cloning system $((G_n), (\iota_n), (\rho_n), (\kappa_k^n))$.
\end{definition}

Given a cloning system $((G_n), (\iota_n), (\rho_n), (\kappa_k^n))$, we can define a BZS product $G \bowtie \sF$, which is an Ore monoid, that is, a cancellative monoid with common right multiples.
In particular it has a group of right fractions, denoted by $\widehat{\cT}(G)$, and is called the \textbf{large generalized Thompson group} of $G$.
Moreover the subset of $\widehat{\cT}(G)$ consisting of \textit{simple} elements forms a subgroup of $\widehat{\cT}(G)$.
This subgroup is called the \textbf{generalized Thompson group} and denoted by $\cT(G_*)$.

\begin{example}
  If $G_n$'s are all trivial, then its generalized Thompson group is isomorphic to $F$.
  The quadruple $((S_n), (\lambda_n), (\operatorname{id}_{S_n}), (\varsigma_k^n))$ is a cloning system, and then the group $\cT(S_*)$ is isomorphic to $V$.
  Moreover, the braided Thompson group $BV$ is regarded as the generalized Thompson group of the braid group.
\end{example}

In Section 9 of \cite{witzel2018cloning}, Witzel and Zaremsky discussed the case of $G_n = J_n$ and $\rho_n = \pi_n \colon J_n \to S_n$.
Define an injective homomorphism $\iota_n \colon J_n \to J_{n+1}$ by $(s_{i,j})\iota_n \coloneq s_{i,j}$.
Also, the maps $\kappa_k^n \colon J_n \to J_{n+1}$ are determined recursively by 
\begin{align*}
  (s_{i,j})\kappa_k^n \coloneq
  \begin{cases}
    s_{i,j}            & (j < k),           \\
    s_{k,k+1}s_{i,j+1} & (i \leq k \leq j), \\
    s_{i+1,j+1}        & (k < i), 
  \end{cases}
\end{align*}
and 
\begin{align*}
  (st)\kappa_k^n \coloneq (s)\kappa_{\pi_n(t)(k)}^n (t)\kappa_k^n, 
\end{align*}
where $s$ is a generator and $t$ is a nonempty word in the generators of $J_n$. 
Then the quadruple $((J_n), (\iota_n), (\pi_n), (\kappa_k^n))$ becomes a cloning system.
This cloning system defines the generalized Thompson group $\cT(J_*)$, and it coincides with the cactus Thompson group $CV$.

\begin{remark} \label{remark_CV_embedding}
  From Observation 3.1 in \cite{witzel2018cloning}, the map $J_n \to CV; s \to (T, s, T)$ is injective, where $T$ is a rooted binary tree with $n$ leaves.
  Namely, $CV$ contains a subgroup isomorphic to $J_n$ for every integer $n \geq 2$, as well as a subgroup isomorphic to $J_\infty\coloneqq\varinjlim J_n$. 
\end{remark}


\subsection{Some subgroups of $CV$} \label{subsection_subgroups_CV}
In this subsection, we define some subgroups of $CV$ and study their algebraic properties.

\subsubsection{The pure cactus Thompson group $PCV$}
There is a surjective homomorphism $\overline{\pi} \colon CV \to V$ induced from $\pi_n \colon J_n \to S_n$, that is, $\overline{\pi}$ is defined by $\overline{\pi}((T_+, s, T_-)) \coloneq (T_+, \pi_n(s), T_-)$ for any cactus-tree pair $(T_+, s, T_-)$ with $n$ leaves.
Its kernel $\ker(\overline{\pi})$ is called the \textbf{pure cactus Thompson group}, denoted by $PCV$.
By definition, a cactus-tree pair $(T_+, s, T_-)$ with $n$ leaves is in $PCV$ if and only if $\pi_n(s) = 1$ (i.e. $s \in PJ_n$) and $T_+ = T_-$ hold.

\begin{proposition} [cf. {\cite[Theorem 4.1]{brady2008bv}}] \label{prop_PCV_nonfg}
  The pure cactus Thompson group $PCV$ is not finitely generated.
\end{proposition}

\begin{proof}
  Observe that given two rooted binary trees $S$ and $T$, there is a minimal rooted binary tree containing them, called the \textbf{least common expansion} of $S$ and $T$.
  If they have $m$ and $n$ leaves, respectively, then their least common expansion has fewer than $m+n$ leaves.
  Suppose that $PCV$ is generated by finitely many elements $(T_i, s_i, T_i)$ with $n_i$ leaves $(1 \leq i \leq k)$.
  Then every element of $PCV$ is represented by a cactus-tree pair whose rooted binary tree is a least common expansion of $T_i$'s, and has fewer than $\sum_{i=1}^k n_i$ leaves.
  However, there are elements of $PCV$ whose minimal representatives have more than $\sum_{i=1}^k n_i$ leaves.
  This is a contradiction.
\end{proof}

The next proposition states that every torsion element of $CV$ is represented by a torsion element of some cactus group.
This follows from the same argument as the proof of Proposition 6.1 in \cite{burillo2009metric}.

\begin{proposition} [cf. {\cite[Proposition 6.1]{burillo2009metric}}] \label{CV_torsion}
  Every torsion element of $CV$ is represented by a cactus-tree pair $(T, s, T)$, where $s$ is a torsion element in some $J_n$.
\end{proposition}

\begin{proof}
  Suppose that $f$ is a torsion element of $CV$.
  Let $(A, s, B)$ be a reduced cactus-tree pair representing $f$.
  We construct a cactus-tree pair $(A_k, s_k, B_k)$ representing $f^k$ using the equation $f^k = f^{k-1} f$ for $k \geq 2$.
  Namely, set $(A_1, s_1, B_1) \coloneq (A, s, B)$.
  Suppose that $(A_{k-1}, s_{k-1}, B_{k-1})$ is constructed.
  Then we have
  \begin{align*}
    (A_k, s_k, B_k) & = (A_{k-1}, s_{k-1}, B_{k-1}) \cdot (A_1, s_1, B_1)                                                                     \\
                    & = (\widetilde{A}_{k-1}, \widetilde{s}_{k-1}, E_{k-1}) \cdot (E_{k-1}, \widetilde{s}_1^{(k-1)}, \widetilde{B}_1^{(k-1)}) \\
                    & = (\widetilde{A}_{k-1}, \widetilde{s}_{k-1}\widetilde{s}_1^{(k-1)}, \widetilde{B}_1^{(k-1)}),
  \end{align*}
  where $E_{k-1}$ is the least common expansion of $B_{k-1}$ and $A_1 = A$.
  By construction, $A_k = \widetilde{A}_{k-1}$ is an expansion of $A_{k-1}$.
  Moreover, we see that $B_k = \widetilde{B}_1^{(k-1)}$ is an expansion of $B_{k-1}$ inductively.
  Indeed, $B_2$ is trivially an expansion of $B_1$.
  Suppose that $B_{k-1}$ is an expansion of $B_{k-2}$.
  Recall that $E_{k-1}$ is the least common expansion of $B_{k-1}$ and $A_1$, and thus $E_{k-1}$ is a common expansion of $B_{k-2}$ and $A_1$.
  Since $E_{k-2}$ is the least common expansion of $B_{k-2}$ and $A_1$, $E_{k-1}$ is an expansion of $E_{k-2}$.
  Therefore $B_k$ is an expansion of $B_{k-1}$.

  Suppose that $f^n = 1$ for a positive integer $n$.
  Namely, a cactus-tree pair $(A_n, s_n, B_n)$ represents the identity element, and thus $s_n = 1$ and $A_n = B_n$ hold.
  The observation above implies that $A_n = B_n$ is an expansion of $A_1$.
  Since $E_{n-1}$ is the least common expansion of $B_{n-1}$ and $A_1$ and $B_n$ is a common expansion of $B_{n-1}$ and $A_1$, $B_n$ is an expansion of $E_{n-1}$.
  On the other hand, $B_n = \widetilde{B}_1^{(n-1)}$ and $E_{n-1}$ have the same number of leaves.
  Thus $B_n = E_{n-1}$ holds.
  Consequently, $(E_{n-1}, \widetilde{s}_1^{(n-1)}, \widetilde{B}_1^{(n-1)} (= E_{n-1}))$ is the desired cactus-tree pair representing $f$.
\end{proof}

From Propositions \ref{CV_torsion} and \ref{cactus_property}, we have the following:

\begin{corollary} \label{cor_CV_odd_torsion}
  The following hold$:$
  \begin{enumerate}
    \item The group $CV$ has no odd torsion.
    \item The group $CV$ has torsion of order $2^k$ for every integer $k \geq 1$.
    \item Any finite subgroup of $CV$ is a $2$-group.
    \item The group $PCV$ is torsion-free.
  \end{enumerate}
\end{corollary}

\begin{corollary} \label{cor_CV_nonsplit}
  The short exact sequence
  \begin{equation*}
    \begin{tikzcd}
      1 \arrow[r] & PCV \arrow[r] & CV \arrow[r, "\overline{\pi}"] & V \arrow[r] & 1
    \end{tikzcd}
  \end{equation*}
  does not split.
\end{corollary}

\begin{proposition} \label{prop_CV_center}
  For $G = CV$ or $PCV$, its center $Z(G)$ is trivial.
\end{proposition}

\begin{proof}
  Let $(T_+, s, T_-)$ be a cactus-tree pair in $Z(CV)$.
  Since $\overline{\pi}$ is surjective, $\overline{\pi}((T_+, s, T_-))$ is in $Z(V)$.
  However simplicity of $V$ implies that $Z(V)$ is trivial.
  Therefore $(T_+, s, T_-)$ is in $PCV$, that is, $s \in PJ_n$ for some $n$ and $T_+ = T_- \eqcolon T$.
  By Remark \ref{remark_CV_embedding}, the set $\{ (T, s', T) \mid s' \in PJ_n \}$ forms a subgroup of $CV$ isomorphic to $PJ_n$.
  In particular $(T, s, T)$ is in the center of this subgroup.
  By expansion if necessary, $s$ can be in $PJ_n$ with $n > 3$.
  Since such $PJ_n$ has trivial center, we obtain $s = 1 \in PJ_n$ and $(T_+, s, T_-)$ represents the identity element in $CV$.
  Therefore $Z(CV)$ is trivial.
  Similarly, $Z(PCV)$ is trivial.
\end{proof}

\subsubsection{The cactus Thompson group $CF$}
Define a subgroup $CF \coloneq \overline{\pi}^{-1}(F)$ of $CV$, also called the cactus Thompson group.
Since every element of $F$ is represented by a tree pair with trivial permutation, a cactus-tree pair $(T_+, s, T_-)$ with $n$ leaves is in $CF$ if and only if $\pi_n(s) = 1$ (i.e. $s \in PJ_n$) holds.
Therefore $PCV$ is a subgroup of $CF$, and moreover is the kernel of $\overline{\pi}|_{CF}$.
Therefore the diagram below is commutative:
\begin{equation*}
  \begin{tikzcd}
    1 \arrow[r] & PCV \arrow[r] \arrow[d, equal] & CF \arrow[r, "\overline{\pi}|_{CF}"] \arrow[d, hook] & F \arrow[r] \arrow[d, hook] & 1\\
    1 \arrow[r] & PCV \arrow[r] & CV \arrow[r, "\overline{\pi}"] & V \arrow[r] & 1
  \end{tikzcd}
\end{equation*}
Define a map $\iota \colon F \to CF$ by $\iota((T_+, T_-)) \coloneq (T_+, 1, T_-)$, then we have $\overline{\pi}|_{CF} \circ \iota = \operatorname{id}_F$.
Namely, the first line of the above exact sequences splits.

\section{An infinite presentation of $CV$ and its abelianization} \label{section_infinitely_presented}
\subsection{An infinite presentation}
In order to compute an infinite presentation of $CV$, we find its generating set using the argument in Subsubsection \ref{sec_V_presentation}.
The main theorem of this section is Theorem \ref{theorem_presentation}, and its proof is based on \cite{brady2008bv}.
Namely, we decompose a cactus-tree pair $(T_+, s, T_-)$ in $CV$ into a product of three elements: the positive part, negative part, and the cactus part, that is,
\begin{align*}
  (T_+, s, T_-) = (T_+, 1, R_n) \cdot (R_n, s, R_n) \cdot (R_n, 1, T_-),
\end{align*}
where $n$ is the number of leaves of $T_+$.
The positive part $(T_+, 1, R_n)$ (resp.~negative part $(R_n, 1, T_-)$) is written as a positive word (resp.~negative word) on the set $X = \{ x_0, x_1, x_2, \ldots \}$ which corresponds to the generating set of $F$:
\begin{align*}
  x_i \coloneq \scalebox{2}{\cinput{24}{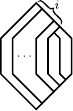}}\quad (i \geq 0).
\end{align*}
In order to represent a cactus part $(R_n, s, R_n)$ by a word on some set, define elements $\sigma_{i,j}$, and $\tau_{i,j}$ of $CV$ as follows:
\begin{align*}
   & \sigma_{i,j} = (R_{j+1}, s_{i,j}, R_{j+1}) = \scalebox{2}{\cinput{33}{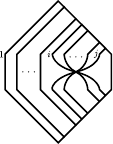}}\quad (1 \leq i < j), \\
   & \tau_{i,j} = (R_j, s_{i,j}, R_j) = \scalebox{2}{\cinput{30}{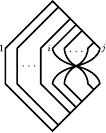}}\quad (1 \leq i < j),
\end{align*}
Then the set $\{\sigma_{i,j}\ (1 \leq i < j \leq n-1), \tau_{i,n}\ (1 \leq i \leq n-1)\}$ generates the copy of $J_n$ in $CV$.
The elements of this set are called the \textbf{$J_n$ generators}.
Any element of $CV$ represented by a cactus-tree pair of the form $(R_n, s, R_n)$ is written as a word in the $J_n$ generators.
From the argument above, we have the following:

\begin{proposition} \label{proposition_generator}
  The cactus Thompson group $CV$ is generated by
  \begin{itemize}
    \item $x_i \quad (i \geq 0);$
    \item $\sigma_{i,j} \quad (1 \leq i < j);$
    \item $\tau_{i,j} \quad (1 \leq i < j)$.
  \end{itemize}
\end{proposition}

The relations between the generators $x_i$'s and the $J_n$ generators are listed below:
\begin{itemize}
  \item $\sigma_{i,j} x_k = \begin{cases}
            x_k \sigma_{i,j}                            & (j \leq k),                    \\
            x_{i+j-k-2} \sigma_{i,j+1} \sigma_{k+1,k+2} & (i-1 \leq k \leq j-1),         \\
            x_k \sigma_{i+1,j+1}                        & (2 \leq i, 0 \leq k \leq i-2),
          \end{cases}$
  \item $\tau_{i,j} x_k = \begin{cases}
            x_{i-1} x_{i} \cdots x_{k-j+i-1} x_{k-j+i}^2 \tau_{i,k+3} \tau_{j,k+3} & (j-1 < k),                     \\
            x_{i-1}^2 \tau_{i,j+2} \tau_{j,j+2}                                    & (k = j-1),                     \\
            x_{i+j-k-2} \tau_{i,j+1} \sigma_{k+1,k+2}                              & (i \leq k \leq j-2),           \\
            \tau_{i,j+1} \sigma_{i,i+1}                                            & (k=i-1),                       \\
            x_k \tau_{i+1,j+1}                                                     & (2 \leq i, 0 \leq k \leq i-2).
          \end{cases}$.
\end{itemize}

In order to compute an infinite presentation of $CV$, we introduce the notion of blocks.
These words arise from the decomposition of a cactus-tree pair into three parts: the positive part, the negative part, and the cactus part.
For a word $w$ over $X$, let $N(w)$ denote the number of carets in either tree of its reduced cactus-tree pair.

\begin{definition} \label{definition_block}
  A word in the generators $x_i^{\pm1}, \sigma_{i,j}^{\pm1}$ and $\tau_{i,j}^{\pm1}$ is called a \textbf{block} if it is of the form $w_1 w_2 w_3^{-1}$ where
  \begin{enumerate}
    \item $w_1 = x_{i_1}^{a_1} x_{i_2}^{a_2} \cdots x_{i_{l}}^{a_{l}}$ with $0 \leq i_1 < i_2 < \cdots < i_{l}$ and $a_k \geq 1$ for all $1 \leq k \leq l$,
    \item $w_3 = x_{j_1}^{b_1} x_{j_2}^{b_2} \cdots x_{j_m}^{b_m}$ with $0 \leq j_1 < j_2 < \cdots < j_m$ and $b_k \geq 1$ for all $1 \leq k \leq m$,
    \item Let $N = \max(N(w_1), N(w_3))$.
          Then there exists an integer $n \geq N+1$ such that $w_2$ is a word in $J_n$ generators.
  \end{enumerate}
\end{definition}

From the definition of $CV$ and representatives of its identity element, we have the following:

\begin{lemma} \label{lemma_identity}
  A block $w_1 w_2 w_3^{-1}$ is the identity in $CV$ if and only if the following two conditions hold$:$
  \begin{itemize}
    \item $w_1$ and $w_3$ are the same word, and
    \item $w_2$ is the identity in the copy of the cactus group $J_n \subset CV$, generated by the $J_n$ generators.
  \end{itemize}
\end{lemma}

\begin{theorem} \label{theorem_presentation}
  The cactus Thompson group $CV$ admits the following presentation$:$
  \begin{description}
    \item[Generators]\
          \begin{itemize}
            \item $x_i \quad (i \geq 0);$
            \item $\sigma_{i,j} \quad (1 \leq i < j);$
            \item $\tau_{i,j} \quad (1 \leq i < j)$.
          \end{itemize}
    \item[Relations]\
          \begin{enumerate}
            \relationitem{A}\label{A} $x_j x_i = x_i x_{j+1} \quad (0 \leq i < j);$
            \relationitem{B1}\label{B1} $\sigma_{i,j}^2 = \tau_{i,j}^2 = 1 \quad (1 \leq i < j);$
            \relationitem{B2}\label{B2} $\sigma_{i,j} \sigma_{k,l} = \sigma_{k,l} \sigma_{i,j} \quad (1 \leq i < j < k < l);$
            \relationitem{B3}\label{B3} $\sigma_{i,j} \sigma_{k,l} = \sigma_{i+j-l,i+j-k} \sigma_{i,j} \quad (1 \leq i \leq k < l \leq j);$
            \relationitem{B4}\label{B4} $\tau_{i,j} \tau_{k,j} = \sigma_{i,i+j-k} \tau_{i,j} \quad (1 \leq i < k < j);$
            \relationitem{B5}\label{B5} $\sigma_{i,j} \tau_{k,l} = \tau_{k,l} \sigma_{i,j} \quad (1 \leq i < j < k < l);$
            \relationitem{B6}\label{B6} $\tau_{i,j} \sigma_{i,l} = \tau_{i+j-l,j} \tau_{i,j} \quad (1 \leq i < l < j);$
            \relationitem{B7}\label{B7} $\tau_{i,j} \sigma_{k,l} = \sigma_{i+j-l,i+j-k} \tau_{i,j} \quad (1 \leq i < k < l < j);$
            \relationitem{C1}\label{C1} $\sigma_{i,j} x_k = x_k \sigma_{i,j} \quad (j \leq k);$
            \relationitem{C2}\label{C2} $\sigma_{i,j} x_k = x_{i+j-k-2} \sigma_{i,j+1} \sigma_{k+1,k+2} \quad (i-1 \leq k \leq j-1);$
            \relationitem{C3}\label{C3} $\sigma_{i,j} x_k = x_k \sigma_{i+1,j+1} \quad (2 \leq i, 0 \leq k \leq i-2);$
            \relationitem{D1}\label{D1} $\tau_{i,j} x_k = x_{i-1} x_{i} \cdots x_{k-j+i-1} x_{k-j+i}^2 \tau_{i,k+3} \tau_{j,k+3} \quad (j-1 < k);$
            \relationitem{D2}\label{D2} $\tau_{i,j} x_k = x_{i-1}^2 \tau_{i,j+2} \tau_{j,j+2} \quad (k = j-1);$
            \relationitem{D3}\label{D3} $\tau_{i,j} x_k = x_{i+j-k-2} \tau_{i,j+1} \sigma_{k+1,k+2} \quad (i \leq k \leq j-2);$
            \relationitem{D4}\label{D4} $\tau_{i,j} x_k = \tau_{i,j+1} \sigma_{i,i+1} \quad (k=i-1);$
            \relationitem{D5}\label{D5} $\tau_{i,j} x_k = x_k \tau_{i+1,j+1} \quad (2 \leq i, 0 \leq k \leq i-2)$.
          \end{enumerate}
  \end{description}
\end{theorem}

Let $G$ be a group given by the presentation above.
Define a homomorphism $\varphi \colon G \to CV$ by sending each generator of $G$ to an element of CV with the same letter.
Since the cactus Thompson group $CV$ satisfies all the relations in $G$, the map $\varphi$ is well-defined.
Proposition \ref{proposition_generator} implies that $\varphi$ is surjective.
Hence we have to show that $\varphi$ is injective.
By Lemma \ref{lemma_identity}, it is true for the case that a word of $G$ is a block.
Thus we need to show that any word in $G$ can be written as a single block.
Note that any generator of $G$ is a block, and thus any word in $G$ is already a product of finitely many blocks.
Therefore, it is enough to show that a product of two blocks can be rewritten as a single block, which is proved in Lemma \ref{lemma_oneblock}.

As a preparation, we show the following three lemmas: Lemma \ref{lemma_commute} means that a generator corresponding to that of $F$ is commutative with cactus generators in a sense.
Lemma \ref{lemma_indexup} means that a word of $J_n$ generators is written as a word of $J_{n+1}$
generators possibly with an extra generator $x_i$.
Using these two lemmas, we show Lemma \ref{lemma_twoblock} which claims that a product of two blocks is written as a new product of two blocks where the sum of the length of the two middle words among the six subwords is reduced.
Finally, we show Lemma \ref{lemma_oneblock} and complete the proof of Theorem \ref{theorem_presentation}.

\begin{lemma} \label{lemma_commute}
  Let $w$ be a word in the $J_n$ generators and $k \leq n-2$.
  Then $w x_k$ is equal to either $x_{k'} \overline{w}$ or $\overline{w}$ in $G$, where $k' \leq n-2$ and $\overline{w}$ is a word in $J_{n+1}$ generators.
  Similarly, $x_k^{-1} w$ is equal to either $\overline{w} x_{k'}^{-1}$ or $\overline{w}$ in $G$ under the same conditions above.
\end{lemma}

\begin{proof}
  It is sufficient to consider the case where $w$ is one of $J_n$ generators:
  \begin{enumerate}
    \relationitem{C1} $\sigma_{i,j} x_k = x_k \sigma_{i,j}\ (j \leq k \leq n-2)$ holds, and thus $k' = k \leq n-2$ and $\overline{w} = \sigma_{i,j} \in J_n$;
    \relationitem{C2} $\sigma_{i,j} x_k = x_{i+j-k-2} \sigma_{i,j+1} \sigma_{k+1,k+2}\ (i-1 \leq k \leq j-1 \leq n-2)$ holds, and thus $k' = i+j-k-2 \leq j-1 \leq n-2$ and $\overline{w} = \sigma_{i, j+1} \sigma_{k+1,k+2} \in J_{n+1}$;
    \relationitem{C3} $\sigma_{i,j} x_k = x_k \sigma_{i+1,j+1}\ (2 \leq i, 0 \leq k \leq i-2 < j-2 \leq n-3)$ holds, and thus, $k' = k \leq n-2$ and $\overline{w} = \sigma_{i+1,j+1} \in J_{n+1}$;
    \relationitem{D3} $\tau_{i,n} x_k = x_{i+n-k-2} \tau_{i,n+1} \sigma_{k+1,k+2}\ (i \leq k \leq n-2)$ holds, and thus
    $k' = i+n-k-2 \leq n-2$ and $\overline{w} = \tau_{i,n+1} \sigma_{k+1,k+2} \in J_{n+1}$;
    \relationitem{D4} $\tau_{i,n} x_k = \tau_{i,n+1} \sigma_{i,i+1}\ (k=i-1 \leq n-2)$ holds, and thus $\overline{w} = \tau_{i,n+1} \sigma_{i,i+1} \in J_{n+1}$;
    \relationitem{D5} $\tau_{i,n} x_k = x_k \tau_{i+1,n+1}\ (2 \leq i, 0 \leq k \leq i-2 \leq n-3)$ holds, and thus $k' = k \leq n-2$ and $\overline{w} = \tau_{i+1,n+1} \in J_{n+1}$.
  \end{enumerate}
  Therefore $w x_k$ is rewritten to either $x_{k'} \overline{w}$ or $\overline{w}$, where $k' \leq n-2$ and $\overline{w}$ is a word in $J_{n+1}$ generators.
  The second case follows from a similar argument.
\end{proof}

\begin{lemma} \label{lemma_indexup}
  Let $w$ be a word in the $J_n$ generators.
  Then $w$ is equal to either $x_{k'} \overline{w}$ or $\overline{w}$ in $G$, where $k' \leq n-2$ and $\overline{w}$ is a word in $J_{n+1}$ generators.
  Similarly, $w$ is equal to either $\overline{w} x_{k'}^{-1}$ or $\overline{w}$ in $G$, where $k'$ and $\overline{w}$ satisfy the conditions above.
\end{lemma}

\begin{proof}
  Consider the rightmost occurrence of a $J_n$ generator $\tau_{i,n}$ in $w$ for some $i$, that is, $w = w_1 \tau_{i,n} w_2$ where $w_2$ consists of only $\sigma$ generators.
  We have
  \begin{align}
    w  = w_1 x_{i-1} \sigma_{i,i+1} \tau_{i,n+1} w_2 \tag{\relref{D4}},
  \end{align}
  and by Lemma \ref{lemma_commute}, $w_1 x_{i-1} = x_{k'} \overline{w_1}$ or $\overline{w_1}$ holds, where $k' \leq n-2$ and $\overline{w_1}$ is a word in the $J_{n+1}$ generators.
  Since $\sigma_{i,i+1}$ and $\tau_{i,n+1}$ are $J_{n+1}$ generators, we may set $\overline{w} = \overline{w_1} \sigma_{i,i+1} \tau_{i,n+1} w_2$.
\end{proof}

For a word $w$ in $G$, let $\ell(w)$ be the number of generators in $w$.

\begin{lemma} \label{lemma_twoblock}
  Let $w_1 w_2 w_3^{-1}, v_1 v_2 v_3^{-1}$ be blocks.
  Then their product $w_1 w_2 w_3^{-1} v_1 v_2 v_3^{-1}$ is equal to the word of the form $w_1' w_2' (w_3')^{-1} v_1' v_2' (v_3')^{-1}$ in $G$ such that $w_1' w_2' (w_3')^{-1}$ and $v_1' v_2' (v_3')^{-1}$ are blocks and $\ell(w_3') + \ell(v_1') < \ell(w_3) + \ell(v_1)$.
\end{lemma}

\begin{proof}
  Denote the first letters of $w_3$ and $v_1$ by $x_{k_{w}}$ and $x_{k_v}$, respectively.
  Namely, $w_3 = x_{k_w} s, v_1 = x_{k_v} t$, where $s, t$ are words on $X$.
  First, if $k_w = k_v$, then
  \begin{align*}
  w_3^{-1} v_1 = s^{-1} x_{k_w}^{-1} x_{k_v} t = s^{-1}t.
  \end{align*}
  Moreover $w_1 w_2 s^{-1}$ and $t v_2 v_3^{-1}$ are both blocks, and $\ell(s) + \ell(t) < \ell(w_3) + \ell(v_1)$ holds.

  Next we suppose $k_w > k_v$.
  Set $v_1' = t, v_2' = v_2, v_3' = v_3$.
  Since the indices of the generators of $X$ in $w_3$ are greater than $k_v$, \relref{A} implies that $w_3^{-1} x_{k_v} = x_{k_v} (w_{3}')^{-1}$, where $w_3'$ is the word on $X$ obtained by increasing the index of each generator in $w_3$ by 1.
  Note that $\ell(w_3') = \ell(w_3)$ and $N(w_3') = N(w_3) + 1$ by Proposition 2 in \cite{burillo2001metrics}.
  Then we need to rewrite $w_1 w_2 x_{k_v} = w_1' w_2'$ so that $w_1' w_2' (w_3')^{-1}$ is a block.
  Consider the following two cases:

  Case 1: Suppose $k_v \leq n-2$.
  By Lemma \ref{lemma_commute}, we have $w_2 x_{k_v} = x_{k} w_2'$, where $k \leq n-2$ and $w_2'$ is a word in the $J_{n+1}$ generators.
  Then by \relref{A}, we have $w_1 x_{k} = w_1'$, where $w_1'$ is a word on $X$ satisfying condition (1) of Definition \ref{definition_block}.
  From Proposition 2 in \cite{burillo2001metrics} again,
  \begin{align*}
  N(w_1') = \max(N(w_1)+1, k+2)
  \end{align*}
  holds, and the condition of $w_1 w_2 w_3^{-1}$ being a block implies $N(w_1') \leq n+1$.
  Therefore $w_1' w_2' (w_3')^{-1}$ is a block and $\ell(w_3') + \ell(v_1') < \ell(w_3) + \ell(v_1)$.

  Case 2: Suppose $k_v > n-2$.
  Lemma \ref{lemma_indexup} implies $w_2 = \overline{w_2}$ or $x_k \overline{w_2}$, where $k \leq n-2$ and $\overline{w_2}$ is a word in the $J_{n+1}$ generators.
  If $w_2 = x_k \overline{w_2}$, we use \relref{A} to rewrite $w_1 x_k = \overline{w_1}$ so that it satisfies condition (1) of Definition \ref{definition_block}.
  Repeating this operation $k_v - (n-2)$ times, we obtain $w_1 w_2 = \overline{w_1}\ \overline{w_2}$, where $\overline{w_1}$ is a word on $X$ satisfying condition (1) of Definition \ref{definition_block}, and $\overline{w_2}$ is a word in the $J_{k_v+2}$ generators.
  Note that
  \begin{align*}
  N(\overline{w_1}) \leq n + (k_v - (n-2)) = k_v+2.
  \end{align*}
  Moreover, from Lemma \ref{lemma_commute}, we have $\overline{w_2} x_{k_v} = x_{k'} w_2'$, where $k' \leq k_v$ and $w_2'$ is a word in the $J_{k_v+3}$ generators.
  Then by \relref{A}, $\overline{w_1} x_{k'} = w_1'$ holds, where $w_1'$ is a word on $X$ satisfying condition (1) of Definition \ref{definition_block}.
  By the same argument in Case 1, we have
  \begin{align*}
  N(w_1') = \max(N(\overline{w_1})+1, k'+2) \leq k_v+3.
  \end{align*}
  Therefore $w_1' w_2' (w_3')^{-1}$ is a block and $\ell(w_3') + \ell(v_1') < \ell(w_3) + \ell(v_1)$.

  The case $k_w < k_v$ also follows from a similar argument.
\end{proof}

\begin{lemma} \label{lemma_oneblock}
  Let $w_1 w_2 w_3^{-1}, v_1 v_2 v_3^{-1}$ be blocks.
  Then their product $w_1 w_2 w_3^{-1} v_1 v_2 v_3^{-1}$ is equal to a single block $u_1 u_2 u_3^{-1}$ in $G$.
\end{lemma}

\begin{proof}
  Applying Lemma \ref{lemma_twoblock} at most $\ell(w_3) + \ell(v_1)$ times, $w_1 w_2 w_3^{-1} v_1 v_2 v_3^{-1} = w_1' w_2' v_2' (v_3')^{-1}$ holds, where $w_2'$ is a word in the $J_n$ generators, $v_2'$ is a word in the $J_{n'}$ generators, $w_1'$ is a word on $X$ satisfying condition (1) of Definition \ref{definition_block} with $N(w_1') \leq n$, and $v_3'$ is a word on $X$ satisfying condition (2) of Definition \ref{definition_block} with $N(v_3') \leq n'$.
  If $n=n'$, then we may set $u_1 = w_1', u_2 = w_2' v_2'$ and $u_3 = v_3'$.
  Suppose $n < n'$.
  Applying Lemma \ref{lemma_indexup} $n'-n$ times, we have
  \begin{align*}
  w_2' = x_{i_{1}}^{\varepsilon_1} x_{i_2}^{\varepsilon_2} \cdots x_{i_{n'-n}}^{\varepsilon_{n'-n}} \overline{w_2},
  \end{align*} 
  where $\overline{w_2}$ is a word in the $J_{n'}$ generators, $i_j \leq n+j-3$, and $\varepsilon_j \in \{ 0, 1\}$ for $1\leq j \leq n'-n$.
  Since $N(w_1' x_{i_{1}}^{\varepsilon_1}) \leq \max(N(w_1')+1, i_1+2)) \leq n+1$, we inductively see that $N(w_1' x_{i_{1}}^{\varepsilon_1} x_{i_2}^{\varepsilon_2} \cdots x_{i_{n'-n}}^{\varepsilon_{n'-n}} ) \leq n'$.
  Then we may set $u_1 = w_1' x_{i_{1}}^{\varepsilon_1} x_{i_2}^{\varepsilon_2} \cdots x_{i_{n'-n}}^{\varepsilon_{n'-n}} , u_2 = \overline{w_2} v_2'$, and $u_3 = v_3'$.
  The case of $n > n'$ is also shown by a similar argument.
\end{proof}
This proves that $\varphi$ is injective and completes the proof of Theorem~\ref{theorem_presentation}.

As a corollary of Theorem \ref{theorem_presentation}, we observe that $CV$ is a finitely generated group.
\begin{corollary}\label{corollary_four_generators}
  The group $CV$ is generated by $\{x_0, x_1, \sigma_{1, 2}, \tau_{1, 2}\}$.
\end{corollary}

\begin{proof}
  Let
  \begin{align*}
    H
    \coloneqq
    \langle x_0, x_1, \sigma_{1,2}, \tau_{1,2} \rangle \leq CV.
  \end{align*}
  Since $x_0$ and $x_1$ generate Thompson's group $F$, it remains to show that every $\tau_{i, j}$ and $\sigma_{i, j}$ belongs to $H$.

  By \relref{D4}, \relref{B4} and \relref{B1}, we have
  \begin{align*}
    \tau_{1, 3}
     & =\tau_{1, 2} x_0 \sigma_{1, 2} \in H,       \\
    \tau_{2, 3}
     & =\tau_{1, 3}\sigma_{1, 2}\tau_{1, 3} \in H.
  \end{align*}
  On the other hand, by \relref{B3}, \relref{C2}, and \relref{B1}, we have
  \begin{align*}
    \sigma_{2, 3}= \sigma_{1, 3}\sigma_{1, 2}\sigma_{1, 3}=(x_1^{-1}\sigma_{1, 2}x_0\sigma_{1, 2})\sigma_{1, 2}(x_1^{-1}\sigma_{1, 2}x_0\sigma_{1, 2}) \in H.
  \end{align*}

  By repeatedly applying \relref{D4} and using \relref{B1}, we obtain
  \begin{align*}
    \tau_{1, j+1}=\tau_{1, j}x_0\sigma_{1, 2}\ (j \geq 2), \\
    \tau_{2, j+1}=\tau_{2, j}x_1\sigma_{2, 3}\ (j \geq 3).
  \end{align*}
  Thus, all elements $\tau_{1, j}$ and $\tau_{2, j}$ belong to $H$.
  Moreover, \relref{D5} gives $\tau_{i+1, j+1}=x_0^{-1}\tau_{i, j}x_0$.
  It follows that every $\tau_{i, j}$ belongs to $H$.
  Finally, by \relref{B4} and \relref{B1},
  \begin{align*}
    \sigma_{i, j}
    =
    \tau_{i,j+1}\tau_{i+1,j+1}\tau_{i,j+1},
  \end{align*}
  and hence every $\sigma_{i,j}$ belongs to $H$.
  Therefore $H=CV$.
\end{proof}

\subsection{The abelianization of $CV$}
In this subsection, we compute the abelianization of $CV$.
\begin{theorem} \label{thm:abelianization}
  The abelianization of $CV$ is isomorphic to $\bZ/2\bZ \oplus \bZ/2\bZ$.
\end{theorem}
\begin{proof}
  We use the presentation defined in Theorem~\ref{theorem_presentation} to calculate the abelianization of $CV$.
  We can observe that the map $\phi\colon CV \to \bZ/2\bZ \oplus \bZ/2\bZ$ is a surjective homomorphism, which is defined by
  \begin{align*}
    x_i           & \mapsto (0, 0) \quad (i \geq 0), \\
    \sigma_{i, j} & \mapsto
    \begin{cases}
      (0, 0) & \text{if } j-i \text{ is even}, \\
      (1, 0) & \text{if } j-i \text{ is odd},
    \end{cases}       \\
    \tau_{1, j}   & \mapsto
    \begin{cases}
      (0, 1) & \text{if } j \text{ is even}, \\
      (1, 1) & \text{if } j \text{ is odd},
    \end{cases}         \\
    \tau_{i, j}   & \mapsto
    \begin{cases}
      (0, 0) & \text{if } j-i \text{ is even and } i \geq 2, \\
      (1, 0) & \text{if } j-i \text{ is odd and } i \geq 2.
    \end{cases}
  \end{align*}
  Since the commutator subgroup $[CV, CV]$ is a subgroup of $\Ker \phi$, this map induces a surjective homomorphism $\hat{\phi}\colon (CV)^{\text{ab}} \to \bZ/2\bZ \oplus \bZ/2\bZ$, where $(CV)^{\text{ab}}$ is the abelianization of $CV$.
  Since $\hat{\phi}$ is surjective, we have $|(CV)^{\text{ab}}| \geq 4$.
  Therefore, since $|\bZ/2\bZ \oplus \bZ/2\bZ|=4$, to show that $(CV)^{\text{ab}}$ and $\bZ/2\bZ \oplus \bZ/2\bZ$ are isomorphic, it suffices to show that $|(CV)^{\text{ab}}| \leq 4$.

  For each $x \in CV$, we denote by $\ov{x}$ the corresponding element in $(CV)^{\text{ab}}$.
  We first observe that $\ov{x_i}=\ov{x_j}$ holds for any $i, j \geq 1$ by \relref{A} and $2\ov{\sigma_{i, j}}=2\ov{\tau_{i, j}}=0$ for any $1 \leq i<j$ by \relref{B1}.
  Next, from \relref{B3}, for any $1 \leq k<l$ (with $i=1$ and $j=l$), we see that $\ov{\sigma_{k, l}}=\ov{\sigma_{1, 1+l-k}}$ holds.
  Hence we write $\ov{\sigma_{k, l}}$ as $a_{l-k}$.
  Moreover, from \relref{B4} with $i=1$ and $k \geq 2$, we see that $\ov{\tau_{k, j}}=\ov{\sigma_{1, 1+j-k}}=a_{j-k}$ holds.
  To simplify the notation, we write $b_j$ instead of $\ov{\tau_{1, j}}$ with $j \geq 2$, and $x$ instead of $\ov{x_1}=\ov{x_2}= \cdots$.
  Under this notation, we see that $(CV)^{\text{ab}}$ is generated by $\ov{x_0}, x, a_1, a_2, \cdots$ and $b_2, b_3, \cdots$.

  From \relref{C2}, for $i, j, k$ with $1 \leq i<j$ and $i-1 \leq k \leq j-1$, we have
  \begin{align*}
    a_{j-i}+\ov{x_k}=\ov{x_{i+j-k-2}}+a_{j+1-i}+a_1.
  \end{align*}
  First, if $i=1$, $k=1$, and $j=2$, then we have
  \begin{align*}
    a_1+x=\ov{x_0}+a_2+a_1,
  \end{align*}
  and hence we obtain
  \begin{align*}
    x=\ov{x_0}+a_2. 
  \end{align*}
  Second, $i=1$ and $k=j-2$ with $j \geq 3$, then we have
  \begin{align*}
    a_{j-1}+x=x+a_j+a_1,
  \end{align*}
  and hence we obtain
  \begin{align}
    a_{j-1}=a_j+a_1. \label{eq_i1kj-2jgt2}
  \end{align}
  Finally, if $i=1$ and $k=j-1$ with $j \geq 3$, we have
  \begin{align*}
    a_{j-1}+x=\ov{x_0}+a_j+a_1.
  \end{align*}
  By equation \eqref{eq_i1kj-2jgt2}, this equation implies that $x=\ov{x_0}$ holds and hence we also have $a_2=0$.
  Since equation \eqref{eq_i1kj-2jgt2} is equivalent to $a_j=a_{j-1}+a_1$, for $d \geq 1$, we now obtain
  \begin{align*}
    a_d=
    \begin{cases}
      0   & \text{if } d \text{ is even}, \\
      a_1 & \text{if } d \text{ is odd}.
    \end{cases}
  \end{align*}

  On the other hand, from \relref{D3}, for $i, j, k$ with $1\leq i<j$ and $i \leq k \leq j-2$, we have
  \begin{align*}
    \ov{\tau_{i, j}}+x=x+\ov{\tau_{i, j+1}}+a_1,
  \end{align*}
  and hence we have $\ov{\tau_{i, j}}=\ov{\tau_{i, j+1}}+a_1$.
  Note that by \relref{D4}, we also have
  \begin{align*}
    \ov{\tau_{i, j}}+x=\ov{\tau_{i, j+1}}+a_1,
  \end{align*}
  and therefore we get $x=\ov{x_0}=0$.

  By \relref{D4} with $i=1$, $k=0$, and $j \geq 2$, we have
  \begin{align*}
    b_j=b_{j+1}+a_1.
  \end{align*}
  This equation is equivalent to $b_{j+1}=b_j+a_1$.
  Therefore, for $d \geq 2$, we now obtain
  \begin{align*}
    b_d=
    \begin{cases}
      b_2     & \text{if } d \text{ is even}, \\
      b_2+a_1 & \text{if } d \text{ is odd}.
    \end{cases}
  \end{align*}

  Consequently, the group $(CV)^{\text{ab}}$ is generated by $a_1$ and $b_2$.
  Since $2a_1=2b_2=0$ and $(CV)^{\text{ab}}$ is abelian, it follows that $|(CV)^{\text{ab}}| \leq 4$.
  This completes the proof.
\end{proof}
The following result follows immediately from Theorem~\ref{thm:abelianization}.
\begin{corollary}\label{cor:notperfect}
  The group $CV$ is not perfect.
\end{corollary}
This result contrasts with the fact that $BV$ is perfect (cf.~\cite{zaremsky2018normal}).
Finally, we state a result on the abelianization of cactus groups.
This should also be compared with Theorem \ref{thm:abelianization}.
\begin{proposition}
  $J_n^{\text{ab}}$ is isomorphic to ${(\mathbb{Z}/2\mathbb{Z})}^{n-1}$ for every $n \geq 2$.
\end{proposition}
\section{A finite presentation of $CV$} \label{section_finitely_presented}
In this section, we show the following theorem algebraically.
\begin{theorem} \label{theorem_finitely_presented}
  The group $CV$ is finitely presented.
\end{theorem}

The proof follows a standard strategy.
We first introduce another infinite presentation of $CV$ with fewer
generators and prove that it defines a group isomorphic to $CV$.
We then introduce a finitely presented group and construct an
isomorphism between these two groups.

In Subsection
\ref{subsection_finitely_presented_group_and_map},
we introduce these two groups, prove the first isomorphism, and define
the map to the finitely presented group.
Subsections
\ref{subsection_firstrelation}--\ref{subsection_lastrelation}
are devoted to proving that this map is a homomorphism.
To prove that the map is well-defined, we verify that every defining
relation in the reduced infinite presentation holds in the finitely
presented group.
Since these verifications are closely interdependent, the order in
which the relations are established is essential.
We therefore provide detailed proofs, avoiding omissions as much as
possible and citing previously propagated relations only through the
corresponding theorems.

In the following, we often use the standard fact that \relref{A} (= \relref{TA}) is obtained from two relators $[x_0x_1^{-1}, x_0^{-1}x_1x_0]=1$ and $[x_0x_1^{-1}, x_0^{-2}x_1x_0^2]=1$, where $[g, h]$ denotes $ghg^{-1}h^{-1}$.
See \cite{cannon1996introductory} for details.
\subsection{Another infinite presentation of $CV$ and a finitely presented group} \label{subsection_finitely_presented_group_and_map}
We begin by reducing the number of generators in order to simplify the proof.
\begin{proposition} \label{Proposition_CV_CVT}
  The group $CV$ is isomorphic to the group $CV_T$ defined by the following presentation:
  \begin{description}
    \item[Generators]\
          \begin{itemize}
            \item $x_i \quad (i \geq 0);$
            \item $\tau_{i,j} \quad (1 \leq i < j)$.
          \end{itemize}
    \item[Relations]\
          \begin{enumerate}
            \relationitem{TA} \label{TA}
            $
              x_jx_i=x_ix_{j+1}
              \quad
              (0\le i<j)
            ;$

            \relationitem{TB1} \label{TB1}
            $
              \tau_{i,j}^2=1
              \quad
              (1\le i<j)
            ;$

            \relationitem{TB2} \label{TB2}
            $
              \widehat\sigma_{i,j}\widehat\sigma_{k,l}
              =
              \widehat\sigma_{k,l}\widehat\sigma_{i,j}
              \quad
              (1\le i<j<k<l)
            ;$

            \relationitem{TB3} \label{TB3}
            $
              \widehat\sigma_{i,j}\widehat\sigma_{k,l}
              =
              \widehat\sigma_{i+j-l,i+j-k}\widehat\sigma_{i,j}
              \quad
              (1\le i\le k<l\le j)
            ;$

            \relationitem{TB5} \label{TB5}
            $
              \widehat\sigma_{i,j}\tau_{k,l}
              =
              \tau_{k,l}\widehat\sigma_{i,j}
              \quad
              (1\le i<j<k<l)
            ;$

            \relationitem{TB6} \label{TB6}
            $
              \tau_{i,j}\widehat\sigma_{i,l}
              =
              \tau_{i+j-l,j}\tau_{i,j}
              \quad
              (1\le i<l<j)
            ;$

            \relationitem{TB7} \label{TB7}
            $
              \tau_{i,j}\widehat\sigma_{k,l}
              =
              \widehat\sigma_{i+j-l,i+j-k}\tau_{i,j}
              \quad
              (1\le i < k<l<j)
            ;$

            \relationitem{TC1} \label{TC1}
            $
              \widehat\sigma_{i,j}x_k
              =
              x_k\widehat\sigma_{i,j}
              \quad
              (j\le k)
            ;$

            \relationitem{TC2} \label{TC2}
            $
              \widehat\sigma_{i,j}x_k
              =
              x_{i+j-k-2}\widehat\sigma_{i,j+1}\widehat\sigma_{k+1,k+2}
              \quad
              (i-1\le k\le j-1)
            ;$

            \relationitem{TC3} \label{TC3}
            $
              \widehat\sigma_{i,j}x_k
              =
              x_k\widehat\sigma_{i+1,j+1}
              \quad
              (2\le i,\ 0\le k\le i-2)
            ;$

            \relationitem{TD1} \label{TD1}
            $
              \tau_{i,j}x_k
              =
              x_{i-1}x_i\cdots x_{k-j+i-1}x_{k-j+i}^2
              \tau_{i,k+3}\tau_{j,k+3}
              \quad
              (k>j-1)
            ;$

            \relationitem{TD2} \label{TD2}
            $
              \tau_{i,j}x_{j-1}
              =
              x_{i-1}^2\tau_{i,j+2}\tau_{j,j+2}
              \quad
              (1\le i<j)
            ;$

            \relationitem{TD3} \label{TD3}
            $
              \tau_{i,j}x_k
              =
              x_{i+j-k-2}\tau_{i,j+1}\widehat\sigma_{k+1,k+2}
              \quad
              (i\le k\le j-2)
            ;$

            \relationitem{TD4} \label{TD4}
            $
              \tau_{i,j}x_{i-1}
              =
              \tau_{i,j+1}\widehat\sigma_{i,i+1}
              \quad
              (1\le i<j)
            ;$

            \relationitem{TD5} \label{TD5}
            $
              \tau_{i,j}x_k
              =
              x_k\tau_{i+1,j+1}
              \quad
              (2\le i,\ 0\le k\le i-2)
            $,
          \end{enumerate}
          where $\widehat\sigma_{i,j} \coloneqq \tau_{i, j+1} \tau_{i+1, j+1} \tau_{i, j+1}$ for $1 \leq i <j$.
  \end{description}
\end{proposition}
\begin{proof}
  Define a map $\phi \colon CV \to CV_T$ by
  \begin{align*}
    \phi(x_i)=x_i, \phi(\tau_{i, j})=\tau_{i, j}, \phi(\sigma_{i, j})=\widehat\sigma_{i, j}=\tau_{i, j+1} \tau_{i+1, j+1} \tau_{i, j+1}.
  \end{align*}
  We first claim that this map is a homomorphism.
  It suffices to show that $\phi$ respects \relref{B1} and \relref{B4}.
  For \relref{B1}, we observe that
  \begin{align*}
    \widehat\sigma_{i, j}^2=(\tau_{i, j+1} \tau_{i+1, j+1} \tau_{i, j+1})( \tau_{i, j+1} \tau_{i+1, j+1} \tau_{i, j+1})=1 \in CV_T
  \end{align*}
  by \relref{TB1}.
  For \relref{B4}, we have
  \begin{align*}
    \widehat\sigma_{i, i+j-k} \tau_{i, j}=\tau_{i, j}\tau_{i+j-(i+j-k), j}=\tau_{i, j} \tau_{k, j} \in CV_T
  \end{align*}
  by inverting both sides of \relref{TB6} with $l=i+j-k$.

  By \relref{B4} and \relref{B1} with $(i, k, j) = (i, i+1, j+1)$, we have
  \begin{align*}
    \sigma_{i, j}=\tau_{i, j+1} \tau_{i+1, j+1} \tau_{i, j+1}
  \end{align*}
  in $CV$.
  Therefore the map $\psi \colon CV_T \to CV$ defined by $\psi(x_i)=x_i$ and $\psi(\tau_{i, j})=\tau_{i, j}$ is the inverse of $\phi$, and the conclusion follows.
\end{proof}
Next, we define a finitely presented group which is isomorphic to $CV_T$ (and hence to $CV$).
In fact, one can reduce the number of relations further; however, for readability, we include some redundant relations in $\cR$.
For the relations that are actually used in the proof, we refer the reader to Subsection \ref{subsection_minimal_relations}.
\begin{definition}
  Let $\cR$ be the set of all instances of \relref{TB1}--\relref{TD5} for which, after replacing each $\widehat\sigma_{i,j}$ by $\tau_{i,j+1}\tau_{i+1,j+1}\tau_{i,j+1}$, all indices appearing in the relation are at most $7$.

  We define $CV_7$ to be the group given by the following presentation:
  \begin{description}
    \item[Generators]\
          \begin{itemize}
            \item $x_0, x_1$;
            \item $\tau_{i,j} \quad (1 \leq i < j \leq 7)$.
          \end{itemize}
    \item[Relations]\
          \begin{itemize}
            \item $[x_0x_1^{-1}, x_0^{-1}x_1x_0]=1$;
            \item $[x_0x_1^{-1}, x_0^{-2}x_1x_0^2]=1$;
            \item $\cR$,
          \end{itemize}
          where every occurrence of $\widehat \sigma_{i, j}$ in $\cR$ denotes the word $\tau_{i, j+1} \tau_{i+1, j+1} \tau_{i, j+1}$.
  \end{description}
\end{definition}
In the rest of the paper, to distinguish this group clearly from $CV_T$, we denote $x_i$, $\tau_{i, j}$, and $\widehat \sigma_{i, j}$ by $X_i$, $T_{i, j}$, and $\Sigma_{i, j}$, respectively.
We now proceed to prove that the map $\varphi \colon CV_T \to CV_7$, given by
\begin{align*}
  x_i         & \mapsto
  \begin{cases}
    X_0                                      & \text{if } i=0,      \\
    X_i \coloneqq X_0^{-(i-1)}X_1X_0^{(i-1)} & \text{if } i \geq 1,
  \end{cases} \\
  \tau_{i, j} & \mapsto
  \begin{cases}
    T_{1, j} \coloneqq T_{1, 2}(X_0 T_{1,3}T_{2, 3}T_{1,3})^{j-2}                          & \text{if } i=1, j \geq 2, \\
    T_{i, j} \coloneqq X_0^{-(i-2)}T_{2, 3}(X_1 T_{2, 4}T_{3, 4}T_{2,4})^{j-i-1} X_0^{i-2} & \text{if } 2 \leq i <j.
  \end{cases}
\end{align*}
is well-defined.
The relations in $\cR$ corresponding to \relref{TB1}, \relref{TD4}, and \relref{TD5} ensure that, for all $1\leq i<j\leq 7$, the generators $T_{i,j} \in CV_7$ coincide with the corresponding words appearing in the above assignment.
Thus the notation is compatible with the generators of $CV_7$ in low indices.
For notational convenience, we define
\begin{align*}
  \Sigma_{i,j} \coloneqq T_{i,j+1} T_{i+1,j+1} T_{i,j+1}
\end{align*}
for $1 \leq i <j$.
With this notation, these words can be written as
\begin{align*}
  T_{1,j}=T_{1,2}(X_0\Sigma_{1,2})^{j-2}
  \quad (j\geq 2),
\end{align*}
and
\begin{align*}
  T_{i,j}
  =
  X_0^{-(i-2)}T_{2,3}(X_1\Sigma_{2,3})^{j-i-1}X_0^{i-2}
  \quad (2\leq i<j).
\end{align*}
\subsection{Propagation of \relref{TB1}} \label{subsection_firstrelation}
In this subsection, we show that
\begin{align*}
  T_{i, j}^2=1 \quad (1 \leq i <j)
\end{align*}
holds.
From the definition of $T_{i, j}$, it suffices to consider the cases $i=1$ and $i=2$.
\subsubsection{The case $i=1$}
We first note that $T_{1,2}(X_0\Sigma_{1,2})T_{1,2}=(X_0\Sigma_{1,2})^{-1}$.
Indeed, by the compatibility established above and the relations
$T_{1,2}^2=T_{1,3}^2=1$ in $\cR$, we have
\begin{align*}
  T_{1,2}(X_0\Sigma_{1,2})T_{1,2}(X_0\Sigma_{1,2})
  =
  T_{1,3}^2
  =
  1.
\end{align*}
Therefore, we have
\begin{align}
  T_{1,j}^2 & =T_{1,2}(X_0\Sigma_{1,2})^{j-2}T_{1,2}(X_0\Sigma_{1,2})^{j-2}   \notag \\
            & =(T_{1,2}(X_0\Sigma_{1,2})T_{1,2})^{j-2}(X_0\Sigma_{1,2})^{j-2} \notag \\
            & = (X_0\Sigma_{1,2})^{-(j-2)}(X_0\Sigma_{1,2})^{j-2}           \notag   \\
            & = 1 \label{eq:T_1j^2=1}
\end{align}
\subsubsection{The case $i=2$}
Similarly, by the compatibility established above and the relations
$T_{2,3}^2=T_{2,4}^2=1$ in $\cR$, we have $T_{2,3}(X_1\Sigma_{2,3})T_{2,3}=(X_1\Sigma_{2,3})^{-1}$.
Hence we have
\begin{align}
  T_{2, j}^2 & = T_{2,3}((X_1 \Sigma_{2,3}))^{j-3} T_{2,3} (X_1 \Sigma_{2,3})^{j-3} \notag \\
             & = (X_1 \Sigma_{2,3})^{-(j-3)}(X_1 \Sigma_{2,3})^{j-3} \notag                \\
             & =1, \label{eq:T_ij^2=1}
\end{align}
which is desired.
By combining equations \eqref{eq:T_1j^2=1} and \eqref{eq:T_ij^2=1}, we obtain:
\begin{theorem}\label{theorem_TB1}
  For any $1 \leq i < j$, we have
  \begin{align*}
    T_{i,j}^2=1. \tag{\relref{TB1}}
  \end{align*}
\end{theorem}
The following is an immediate consequence of this fact.
\begin{corollary}\label{cor_TB1}
  For any $1 \leq i < j$, we have
  \begin{align*}
    \Sigma_{i,j}^2=1.
  \end{align*}
\end{corollary}
\begin{proof}
  By Theorem \ref{theorem_TB1}, we have
  $T_{i,j+1}^2=T_{i+1,j+1}^2=1$. Hence
  \begin{align*}
    \Sigma_{i,j}^2
     & =
    (T_{i,j+1}T_{i+1,j+1}T_{i,j+1})
    (T_{i,j+1}T_{i+1,j+1}T_{i,j+1})          \\
     & =
    T_{i,j+1}T_{i+1,j+1}T_{i+1,j+1}T_{i,j+1} \\
     & =
    T_{i,j+1}^2                              \\
     & =
    1,
  \end{align*}
  which completes the proof.
\end{proof}
\begin{remark}
  For simplicity, we shall use the relation $T_{i, j}^2=\Sigma_{i, j}^2=1$ in what follows, without explicitly mentioning Theorem \ref{theorem_TB1} and Corollary \ref{cor_TB1}.
  Similarly, we shall also omit references for $T_{i, j}^{-1}=T_{i, j}$ and $\Sigma_{i, j}^{-1}=\Sigma_{i, j}$, which follow from these relations.
\end{remark}
\subsection[Propagation of relation (TD4)]{Propagation of \relref{TD4}}
In this subsection, we show that
\begin{align*}
  T_{i, j}X_{i-1}=T_{i, j+1}\Sigma_{i, i+1}
\end{align*}
holds.
From the definition of $T_{i, j}$, it suffices to show the cases $i=1$ and $i=2$.
To see this, we first observe how the indices of $T_{i, j}$ and $\Sigma_{i, j}$ change under conjugation by $X_0$.
\begin{lemma}\label{lemma_TS_conj_x0}
  For $2 \leq i <j$, we have
  \begin{align*}
    X_0^{-1}T_{i,j}X_0=T_{i+1, j+1}, \quad X_0^{-1}\Sigma_{i,j}X_0=\Sigma_{i+1, j+1}.
  \end{align*}
\end{lemma}
\begin{proof}
  By the definition of $T_{i, j}$, we have
  \begin{align*}
    X_0^{-1}T_{i,j}X_0
    =X_0^{-(i-1)}T_{2,3}(X_1\Sigma_{2,3})^{j-i-1}X_0^{i-1}
  \end{align*}
  and
  \begin{align*}
    T_{i+1, j+1}=X_0^{-((i+1)-2)}T_{2,3}(X_1 \Sigma_{2,3})^{(j+1)-(i+1)-1}X_0^{i+1-2}=X_0^{-(i-1)}T_{2,3}(X_1\Sigma_{2,3})^{j-i-1}X_0^{i-1},
  \end{align*}
  and hence $X_0^{-1}T_{i,j}X_0=T_{i+1,j+1}$ follows.
  Using this fact, we can also see that
  \begin{align*}
    X_0^{-1}\Sigma_{i, j}X_0
     & =(X_0^{-1}T_{i, j+1}X_0)(X_0^{-1}T_{i+1, j+1}X_0)(X_0^{-1}T_{i, j+1}X_0) \\
     & =T_{i+1, j+2}T_{i+2,j+2} T_{i+1,j+2}                                     \\
     & =\Sigma_{i+1, j+1},
  \end{align*}
  which completes the proof.
\end{proof}
By Lemma \ref{lemma_TS_conj_x0}, we can lower the indices by conjugation.
\begin{theorem}\label{theorem_TD4}
  For any $1 \leq i <j$, we have
  \begin{align*}
    T_{i, j}X_{i-1}=T_{i, j+1}\Sigma_{i, i+1}. \tag{\relref{TD4}}
  \end{align*}
\end{theorem}
\begin{proof}
  If $i \geq 3$, by Lemma \ref{lemma_TS_conj_x0}, the desired equality is equivalent to the following equality:
  \begin{align*}
    T_{2, j-(i-2)}X_1=T_{2, (j+1)-(i-2)} \Sigma_{2, 3}.
  \end{align*}
  Hence it suffices to show the cases $i=1$ and $i=2$.

  If  $i=1$, then from the definition of $T_{1, j}$, we have
  \begin{align*}
    T_{1, j+1}=T_{1, 2} (X_0 \Sigma_{1, 2})^{j-1}=T_{1, 2} (X_0 \Sigma_{1, 2})^{j-2}  (X_0 \Sigma_{1, 2})=T_{1, j} X_0 \Sigma_{1, 2}.
  \end{align*}
  Since $\Sigma_{1, 2}^2=1$, this completes the proof in the case $i=1$.

  If $i=2$, then for any $j>2$, the definition of $T_{2, j}$ gives
  \begin{align*}
    T_{2, j+1}=T_{2, 3}(X_1 \Sigma_{2, 3})^{j-2}=T_{2, 3} (X_1 \Sigma_{2, 3})^{j-3}(X_1 \Sigma_{2, 3})=T_{2, j} X_1 \Sigma_{2, 3}.
  \end{align*}
  Since $\Sigma_{2, 3}^2=1$, this completes the proof in the case $i=2$.
\end{proof}

\subsection[Propagation of relation (TD5)]{Propagation of \relref{TD5}}
In this subsection, we show that
\begin{align*}
  T_{i, j}X_k=X_k T_{i+1, j+1}
\end{align*}
holds for $2 \leq i <j$ with $0 \leq k \leq i-2$.
If $k=0$, the equality follows directly from Lemma \ref{lemma_TS_conj_x0}.
If $k \geq 2$, again by Lemma \ref{lemma_TS_conj_x0}, the equality is equivalent to the following:
\begin{align*}
  T_{i-k+1, j-k+1} X_1=X_1 T_{i-k+2, j-k+2}.
\end{align*}
Hence we can assume $k=1$ without loss of generality.

We prepare a few lemmas.
\begin{lemma}\label{lemma_commutator_imply_D5}
  Let $3 \leq i <j$.
  If $[X_0 X_1^{-1}, T_{i, j}]=1$ holds, then we have
  \begin{align*}
    T_{i, j}X_1 = X_1 T_{i+1, j+1}.
  \end{align*}
\end{lemma}
\begin{proof}
  We note that $[X_1 X_0^{-1}, T_{i, j}]=1$ also holds.
  Then by Lemma \ref{lemma_TS_conj_x0}, we have
  \begin{align*}
    T_{i, j} X_1= T_{i, j} X_1 X_0^{-1} X_0 =X_1 X_0^{-1} T_{i, j} X_0=X_1 X_0^{-1} X_0 T_{i+1, j+1}=X_1 T_{i+1, j+1},
  \end{align*}
  which completes the proof.
\end{proof}

We first show the sufficient condition for the case $j=i+1, i+2$:

\begin{lemma}\label{lemma_sufficient_condition_basestep}
  For any $i \geq 3$, we have
  \begin{align*}
    [X_0 X_1^{-1}, T_{i, i+1}]=[X_0 X_1^{-1}, T_{i, i+2}]=1.
  \end{align*}
\end{lemma}
\begin{proof}
  We show this equality by induction.
  If $i = 3, 4$, these two equalities follow from \relref{TD5} in $\cR$.
  Assume that both equalities hold for some $i \geq 4$.
  From \relref{TD5} in $\cR$, we have
  \begin{align*}
    T_{4, 5} = X_1^{-1} T_{3, 4} X_1,\quad T_{4, 6} = X_1^{-1} T_{3, 5} X_1.
  \end{align*}
  Applying Lemma \ref{lemma_TS_conj_x0}  $i-3$ times yields the formulas
  \begin{align*}
    T_{i+1, i+2}=X_{i-2}^{-1}T_{i, i+1}X_{i-2}, \quad T_{i+1, i+3}=X_{i-2}^{-1}T_{i, i+2}X_{i-2}.
  \end{align*}
  Since $[X_0 X_1^{-1}, X_{i-2}]=1$ holds for $i-2 \geq 2$, by combining inductive hypotheses, we obtain the desired result.
\end{proof}

Now we show that the sufficient condition holds for any $3 \leq i<j$.

\begin{proposition}\label{proposition_commutator_general_ij}
  For any $3 \leq i<j$, we have
  \begin{align*}
    [X_0 X_1^{-1}, T_{i, j}]=1.
  \end{align*}
\end{proposition}

\begin{proof}
  Fix some $i \geq 3$.
  We show this by induction on $j$.
  The base case $j=i+1$ follows from Lemma \ref{lemma_sufficient_condition_basestep}.
  Assume that $[X_0 X_1^{-1}, T_{i, j}]=1$ for some $j \geq i+1$.

  By Theorem \ref{theorem_TD4}, we have $T_{i, j+1}=T_{i, j}X_{i-1}\Sigma_{i, i+1}$.
  Hence it suffices to show that $X_0X_1^{-1}$ commutes with $T_{i, j}$, $X_{i-1}$, and $\Sigma_{i, i+1}$.
  As for $T_{i, j}$, it is clear from the inductive hypothesis.
  Also, for $X_{i-1}$, it is clear from the assumption that $i \geq 3$.
  For $\Sigma_{i, i+1}$, by Theorem \ref{theorem_TD4} with $j=i+1$, we have $\Sigma_{i, i+1}=T_{i, i+2}T_{i, i+1}X_{i-1}$.
  In this case, the commutativity follows from Lemma \ref{lemma_sufficient_condition_basestep} again.
  Since all elements commute with $X_0X_1^{-1}$, we obtain the conclusion.
\end{proof}
From these preparations, we immediately obtain the following:
\begin{theorem}\label{theorem_TD5}
  For any $2 \leq i<j$ with $0 \leq k \leq i-2$, we have
  \begin{align*}
    T_{i, j}X_k=X_k T_{i+1, j+1} \tag{\relref{TD5}}
  \end{align*}
\end{theorem}
\begin{proof}
  As mentioned at the beginning of this subsection, it is sufficient to show the case where $k=1$.
  That is, we only need to show that $T_{i, j}X_1=X_1 T_{i+1, j+1}$ holds for $3 \leq i <j$.
  By Lemma \ref{lemma_commutator_imply_D5}, it remains to verify that $[X_0 X_1^{-1}, T_{i, j}]=1$.
  Then, it follows from Proposition \ref{proposition_commutator_general_ij} that the equality holds.
\end{proof}

\subsection[Propagation of relation (TC3)]{Propagation of \relref{TC3}}
From the definition of $\Sigma_{i, j}$, \relref{TC3} follows immediately.
\begin{theorem} \label{theorem_TC3}
  For any $2 \leq i <j$ with $0 \leq k \leq i-2$, we have
  \begin{align*}
    \Sigma_{i, j}X_k=X_k \Sigma_{i+1, j+1}. \tag{\relref{TC3}}
  \end{align*}
\end{theorem}
\begin{proof}
  Recall that $\Sigma_{i, j}=T_{i, j+1} T_{i+1, j+1} T_{i, j+1}$ holds by the definition.
  Then by Theorem \ref{theorem_TD5}, we have
  \begin{align*}
    T_{i, j+1}X_k=X_k T_{i+1, j+2}, \quad T_{i+1, j+1}X_k=X_k T_{i+2, j+2}.
  \end{align*}
  Therefore we have
  \begin{align*}
    \Sigma_{i, j}X_k=T_{i, j+1} T_{i+1, j+1} T_{i, j+1} X_k=X_k T_{i+1, j+2}T_{i+2, j+2} T_{i+1, j+2}=X_k \Sigma_{i+1, j+1},
  \end{align*}
  which completes the proof.
\end{proof}
\subsection[Propagation of relation (TD2)]{Propagation of \relref{TD2}}
In this subsection, we will show that for any $1 \leq i <j$, we have
\begin{align*}
  T_{i, j}X_{j-1}=X_{i-1}^2 T_{i, j+2} T_{j, j+2}.
\end{align*}
We begin by establishing a few lemmas on the commutativity with $\Sigma_{2, 3}$ and $\Sigma_{1, 2}$.
\begin{lemma} \label{lemma_commute_Sigma23Xm}
  For any $m \geq 3$, we have $[\Sigma_{2, 3}, X_m]=1$.
\end{lemma}
\begin{proof}
  For $m=3, 4$, the assertion follows from \relref{TC1} in $\cR$.
  Assume that $[\Sigma_{2, 3}, X_{m-1}]=[\Sigma_{2,3}, X_m]=1$ holds for some $m \geq 4$.
  Then we have
  \begin{align*}
    \Sigma_{2, 3} X_{m+1}=\Sigma_{2, 3} X_{m-1}^{-1} X_m X_{m-1}= X_{m-1}^{-1} X_m X_{m-1}\Sigma_{2, 3}=X_{m+1} \Sigma_{2, 3},
  \end{align*}
  which completes the proof.
\end{proof}
\begin{lemma}\label{lemma_commute_Sigma12Xm}
  For any $m \geq 2$, we have $[\Sigma_{1,2}, X_m]=1$.
\end{lemma}
\begin{proof}
  As in the proof of Lemma \ref{lemma_commute_Sigma23Xm}, the base cases $m=2, 3$ follow from \relref{TC1} in $\cR$, and the inductive step can be proved by using $X_{m+1}=X_{m-1}^{-1} X_m X_{m-1}$.
\end{proof}

\begin{lemma} \label{lemma_commute_Sigma23Tm+1m+3}
  For any $m \geq 3$, we have $[\Sigma_{2, 3}, T_{m+1, m+3}]=1$.
\end{lemma}
\begin{proof}
  For $m=3, 4$, the assertion follows from \relref{TB5} in $\cR$.
  Assume that $[\Sigma_{2, 3}, T_{m+1, m+3}]=1$ holds for some $m \geq 4$.
  Note that by Theorem \ref{theorem_TD5}, we have $T_{m+2, m+4}=X_{m-1}^{-1}T_{m+1, m+3}X_{m-1}$.
  Then since $m-1 \geq 3$, by Lemma \ref{lemma_commute_Sigma23Xm} and inductive hypothesis, we have
  \begin{align*}
    \Sigma_{2, 3} T_{m+2, m+4}= \Sigma_{2, 3} X_{m-1}^{-1}T_{m+1, m+3}X_{m-1} =  X_{m-1}^{-1}T_{m+1, m+3}X_{m-1}\Sigma_{2, 3}= T_{m+2, m+4}\Sigma_{2, 3},
  \end{align*}
  which completes the proof.
\end{proof}
\begin{lemma} \label{lemma_commute_Sigma12Tmm+2}
  For any $m \geq 3$, we have $[\Sigma_{1, 2}, T_{m, m+2}]=1$.
\end{lemma}
\begin{proof}
  As in the proof of Lemma \ref{lemma_commute_Sigma23Tm+1m+3}, for $m=3, 4$, the assertion follows from \relref{TB5} in $\cR$.
  Assume that $[\Sigma_{1, 2}, T_{m, m+2}]=1$ holds for some $m \geq 4$.
  By Theorem \ref{theorem_TD5}, Lemma \ref{lemma_commute_Sigma12Xm}, and the inductive hypothesis, we have
  \begin{align*}
    \Sigma_{1, 2}T_{m+1, m+3}=\Sigma_{1, 2} X_{m-2}^{-1} T_{m, m+2} X_{m-2}=  X_{m-2}^{-1} T_{m, m+2} X_{m-2}\Sigma_{1, 2}=T_{m+1, m+3} \Sigma_{1, 2},
  \end{align*}
  which completes the proof.
\end{proof}
We next prove that \relref{TD2} holds for the cases $i=1$ and $i=2$.
\begin{proposition}\label{proposition_TD2_i2}
  For any $m \geq 3$, we have
  \begin{align*}
    T_{2, m}X_{m-1}=X_1^2 T_{2, m+2}T_{m, m+2}.
  \end{align*}
\end{proposition}
\begin{proof}
  We show this by induction on $m$.
  For $m=3$, the assertion follows directly from \relref{TD2} in $\cR$.
  Assume that the equality holds for some $m \geq 3$.
  By Theorem \ref{theorem_TD4}, $T_{2, m+1}=T_{2, m}X_1\Sigma_{2, 3}$ holds.
  Then by Lemmas \ref{lemma_commute_Sigma23Xm}, \ref{lemma_commute_Sigma23Tm+1m+3}, inductive hypothesis, Theorems \ref{theorem_TD5} and \ref{theorem_TD4}, we have
  \begin{align*}
    T_{2, m+1} X_m & = T_{2, m}X_1\Sigma_{2, 3} X_m                               \\
                   & =T_{2, m} X_1 X_m \Sigma_{2, 3}                              \\
                   & = T_{2, m}X_{m-1} X_1 \Sigma_{2, 3}                          \\
                   & =X_1^2 T_{2, m+2}T_{m, m+2} X_1 \Sigma_{2, 3}                \\
                   & = X_1^2 T_{2, m+2} X_1 T_{m+1, m+3} \Sigma_{2, 3}            \\
                   & = X_1^2 T_{2, m+3} \Sigma_{2, 3}  T_{m+1, m+3} \Sigma_{2, 3} \\
                   & = X_1^2 T_{2, m+3} T_{m+1, m+3}\Sigma_{2, 3}\Sigma_{2, 3}    \\
                   & = X_1^2 T_{2, m+3} T_{m+1, m+3},
  \end{align*}
  which is the desired result.
\end{proof}
\begin{proposition}\label{proposition_TD2_i1}
  For any $m \geq 2$, we have
  \begin{align*}
    T_{1, m}X_{m-1}=X_0^2 T_{1, m+2}T_{m, m+2}.
  \end{align*}
\end{proposition}
\begin{proof}
  As in the proof of Proposition \ref{proposition_TD2_i2}, we show this by induction on $m$.
  For $m=2$, the assertion follows from \relref{TD2} in $\cR$.
  Assume that the equality holds for some $m \geq 2$.
  Then by Theorems \ref{theorem_TD4}, \ref{theorem_TD5} (or Lemma \ref{lemma_TS_conj_x0}), Lemmas \ref{lemma_commute_Sigma12Xm}, \ref{lemma_commute_Sigma12Tmm+2}, the inductive hypothesis,
  \begin{align*}
    T_{1, m+1}X_m & = T_{1, m} X_0 \Sigma_{1, 2} X_m                            \\
                  & = T_{1, m} X_0 X_m \Sigma_{1, 2}                            \\
                  & = T_{1, m} X_{m-1} X_0 \Sigma_{1, 2}                        \\
                  & = X_0^2 T_{1, m+2} T_{m, m+2} X_0 \Sigma_{1, 2}             \\
                  & = X_0^2 T_{1, m+2} X_0 T_{m+1, m+3} \Sigma_{1, 2}           \\
                  & = X_0^2 T_{1, m+3} \Sigma_{1, 2} T_{m+1, m+3} \Sigma_{1, 2} \\
                  & = X_0^2 T_{1, m+3} T_{m+1, m+3} \Sigma_{1, 2} \Sigma_{1, 2} \\
                  & = X_0^2 T_{1, m+3} T_{m+1, m+3},
  \end{align*}
  which is the desired result.
\end{proof}
We are now ready to prove \relref{TD2} for all cases.
\begin{theorem}\label{theorem_TD2}
  For any $1 \leq i <j$, we have
  \begin{align*}
    T_{i, j}X_{j-1}=X_{i-1}^2 T_{i, j+2} T_{j, j+2}. \tag{\relref{TD2}}
  \end{align*}
\end{theorem}
\begin{proof}
  If $i \geq 2$, then by Theorem \ref{theorem_TD5} (or Lemma \ref{lemma_TS_conj_x0}), we have
  \begin{align*}
    X_0^{i-2} T_{i, j} X_0^{-(i-2)}   & =T_{2, j-(i-2)},         \\
    X_0^{i-2} T_{i, j+2} X_0^{-(i-2)} & =T_{2, j+2-(i-2)},       \\
    X_0^{i-2} T_{j, j+2} X_0^{-(i-2)} & =T_{j-(i-2), j+2-(i-2)}.
  \end{align*}
  Therefore, the desired relation is equivalent to
  \begin{align*}
    T_{2, j-(i-2)} X_{j-1-(i-2)}=X_{i-1-(i-2)}^2 T_{2, j+2-(i-2)} T_{j-(i-2), j+2-(i-2)}.
  \end{align*}
  By setting $m \coloneqq j-(i-2) \geq 3$, this can be further rewritten as
  \begin{align*}
    T_{2, m} X_{m-1}=X_{1}^2 T_{2, m+2} T_{m, m+2}.
  \end{align*}
  Hence by Proposition \ref{proposition_TD2_i2}, the assertion holds.
  If $i=1$, the equality also holds by Proposition \ref{proposition_TD2_i1}.
\end{proof}
\subsection[Propagation of relation (TD3)]{Propagation of \relref{TD3}}
In this subsection, we show that for any $1 \leq i \leq k \leq j-2$, we have
\begin{align*}
  T_{i, j}X_k=X_{i+j-k-2}T_{i, j+1} \Sigma_{k+1, k+2}.
\end{align*}
\subsubsection{Commutativity relations}
Before proceeding to the proof, we first perform a partial propagation of several other relations.
\begin{lemma}\label{lemma_commute_Sigmaii+1Xm}
  For any $i \geq 1$ and $m \geq i+1$, we have $[\Sigma_{i, i+1}, X_m]=1$.
\end{lemma}
\begin{proof}
  If $i=1$, it follows directly from Lemma \ref{lemma_commute_Sigma12Xm}.
  If $i \geq 2$, then the equality is equivalent to $[\Sigma_{2, 3}, X_{m-(i-2)}]=1$ by Lemma \ref{lemma_TS_conj_x0}.
  Hence by Lemma \ref{lemma_commute_Sigma23Xm}, we obtain the desired result.
\end{proof}
\begin{proposition}\label{proposition_commute_adjacent_SigmaSigma}
  For any $i, k \geq 1$ with $|i-k| \geq 2$, we have $[\Sigma_{i, i+1}, \Sigma_{k, k+1}]=1$.
\end{proposition}
\begin{proof}
  We can assume that $k \geq i+2$ without loss of generality.
  Similar to previous arguments, we show for the cases $i=1$ and $i=2$.

  Assume that $i=1$.
  By \relref{TB2} in $\cR$, we have
  \begin{align*}
    \Sigma_{1, 2}\Sigma_{3, 4}=\Sigma_{3, 4}\Sigma_{1, 2},\quad \Sigma_{1, 2}\Sigma_{4, 5}=\Sigma_{4, 5}\Sigma_{1, 2}.
  \end{align*}
  Assume that $[\Sigma_{1, 2}, \Sigma_{k, k+1}]=1$ holds for some $k \geq 4$.
  Then by Theorem \ref{theorem_TC3}, Lemma \ref{lemma_commute_Sigma12Xm} (or Lemma \ref{lemma_commute_Sigmaii+1Xm}) and the inductive hypothesis, we have
  \begin{align*}
    \Sigma_{1, 2} \Sigma_{k+1, k+2}
    =\Sigma_{1, 2} X_{k-2}^{-1} \Sigma_{k, k+1} X_{k-2}
    =X_{k-2}^{-1} \Sigma_{k, k+1} X_{k-2}\Sigma_{1, 2}
    =\Sigma_{k+1, k+2}\Sigma_{1, 2},
  \end{align*}
  which is the desired equality.

  Next, assume that $i=2$.
  By \relref{TB2} in $\cR$, we have
  \begin{align*}
    \Sigma_{2, 3}\Sigma_{4, 5}
    =\Sigma_{4, 5}\Sigma_{2, 3}, \quad
    \Sigma_{2, 3}\Sigma_{5, 6}
    =\Sigma_{5, 6}\Sigma_{2, 3}.
  \end{align*}
  Assume that $[\Sigma_{2, 3}, \Sigma_{k, k+1}]=1$ holds for some $k \geq 5$.
  Similar to the case $i=1$, by Theorem \ref{theorem_TC3}, Lemma \ref{lemma_commute_Sigma23Xm} (or Lemma \ref{lemma_commute_Sigmaii+1Xm}) and the inductive hypothesis, we have
  \begin{align*}
    \Sigma_{2,3} \Sigma_{k+1, k+2}
    = \Sigma_{2,3}X_{k-2}^{-1} \Sigma_{k, k+1} X_{k-2}
    = X_{k-2}^{-1} \Sigma_{k, k+1} X_{k-2}\Sigma_{2,3}
    =\Sigma_{k+1, k+2}\Sigma_{2,3},
  \end{align*}
  which completes the proof of $i=2$.

  Finally, if $i \geq 3$, then $[\Sigma_{i, i+1}, \Sigma_{k, k+1}]=1$ is equivalent to $[\Sigma_{2, 3}, \Sigma_{k-(i-2), k+1-(i-2)}]=1$ by Lemma \ref{lemma_TS_conj_x0}.
  This allows us to conclude the proof for all cases.
\end{proof}
\begin{lemma}\label{lemma_commute_Sigmaii+1Tkk+2}
  For any $i \geq 1$ and $k \geq i+2$, we have $[\Sigma_{i, i+1}, T_{k, k+2}]=1$.
\end{lemma}
\begin{proof}
  For the cases $i=1$ and $i=2$, the assertion follows from Lemmas \ref{lemma_commute_Sigma12Tmm+2} and \ref{lemma_commute_Sigma23Tm+1m+3}, respectively.
  For $i \geq 3$, by Lemma \ref{lemma_TS_conj_x0}, the desired equality is equivalent to $[\Sigma_{2, 3}, T_{k-(i-2), k+2-(i-2)}]=1$.
  Hence, the proof is now complete.
\end{proof}
\begin{proposition} \label{proposition_commute_Sigmaii+1Tkl}
  For any $i \geq 1$ and $i+1<k<l$, we have $[\Sigma_{i, i+1}, T_{k, l}]=1$.
\end{proposition}
\begin{proof}
  We fix $i, k$ and show this by induction on $l$.
  Note that $i+2 \leq k$ holds.
  First, we assume that $l=k+1$.
  Then by Theorem \ref{theorem_TD4}, Lemmas \ref{lemma_commute_Sigmaii+1Tkk+2}, \ref{lemma_commute_Sigmaii+1Xm} and Proposition \ref{proposition_commute_adjacent_SigmaSigma} we have
  \begin{align*}
    \Sigma_{i, i+1}T_{k, k+1}
    =\Sigma_{i, i+1} T_{k, k+2} \Sigma_{k, k+1} X_{k-1}^{-1}
    = T_{k, k+2} \Sigma_{k, k+1} X_{k-1}^{-1}\Sigma_{i, i+1}
    = T_{k, k+1} \Sigma_{i, i+1}.
  \end{align*}

  Next, assume that $[\Sigma_{i, i+1}, T_{k, l}]=1$ holds for some $l>k$.
  Then by the inductive hypothesis, Theorem \ref{theorem_TD4}, Lemma \ref{lemma_commute_Sigmaii+1Xm}, and Proposition \ref{proposition_commute_adjacent_SigmaSigma}, we have
  \begin{align*}
    \Sigma_{i, i+1} T_{k, l+1}
    = \Sigma_{i, i+1} T_{k, l}X_{k-1} \Sigma_{k, k+1}
    =T_{k, l}X_{k-1} \Sigma_{k, k+1} \Sigma_{i, i+1}
    =T_{k, l+1} \Sigma_{i, i+1},
  \end{align*}
  which completes the proof.
\end{proof}
\subsubsection{Propagation of \relref{TD3}}
We now proceed to the proof of \relref{TD3}.
Analogous to the previous arguments, we split the proof into the cases of $i=1$ and $i \geq 2$.
For fixed $i$ and $k$, the relation is propagated by increasing $j$, except when $k=j-3$ or $k=j-2$.

We begin by examining the exceptional cases.
\begin{proposition}\label{proposition_TD3_exceptional_case}
  For every $i \geq 2$, the following hold$:$
  \begin{enumerate}
    \item For every $j \geq i+3$,
          \begin{align*}
            T_{i, j}X_{j-3}=X_{i+1}T_{i, j+1}\Sigma_{j-2, j-1}.
          \end{align*}
    \item For every $j \geq i+2$,
          \begin{align*}
            T_{i, j}X_{j-2}=X_{i}T_{i, j+1}\Sigma_{j-1, j}.
          \end{align*}
  \end{enumerate}
\end{proposition}
\begin{proof}
  We first show them for the case $i=2$ by induction on $j$.
  By \relref{TD3} in $\cR$, we have
  \begin{align*}
    T_{2, 5}X_2=X_3T_{2, 6}\Sigma_{3, 4},\quad T_{2, 4}X_2=X_2T_{2, 5}\Sigma_{3, 4}.
  \end{align*}

  For (1), assume that $T_{2, j}X_{j-3}=X_{3}T_{2, j+1}\Sigma_{j-2, j-1}$ for some $j \geq 5$.
  Then by Theorems \ref{theorem_TD4}, \ref{theorem_TC3}, Lemma \ref{lemma_commute_Sigma23Xm} (or \ref{lemma_commute_Sigmaii+1Xm}), Proposition \ref{proposition_commute_adjacent_SigmaSigma}, and the inductive hypothesis, we have
  \begin{align*}
    T_{2, j+1}X_{j-2}
     & =T_{2, j}X_1\Sigma_{2, 3} X_{j-2}                   \\
     & =T_{2, j}X_1X_{j-2} \Sigma_{2, 3}                   \\
     & =T_{2, j}X_{j-3}X_1\Sigma_{2, 3}                    \\
     & =X_{3}T_{2, j+1}\Sigma_{j-2, j-1} X_1 \Sigma_{2, 3} \\
     & =X_{3}T_{2, j+1}X_1 \Sigma_{j-1, j} \Sigma_{2, 3}   \\
     & =X_{3}T_{2, j+1}X_1 \Sigma_{2, 3} \Sigma_{j-1, j}   \\
     & =X_3 T_{2, j+2} \Sigma_{j-1, j},
  \end{align*}
  which completes the proof of (1) for the case $i=2$.

  Similarly, for (2), assume that $T_{2, j}X_{j-2}=X_2T_{2, j+1}\Sigma_{j-1, j}$ holds for some $j \geq 4$.
  Then by Theorems \ref{theorem_TD4}, \ref{theorem_TC3}, Lemma \ref{lemma_commute_Sigma23Xm} (or \ref{lemma_commute_Sigmaii+1Xm}), Proposition \ref{proposition_commute_adjacent_SigmaSigma}, and the inductive hypothesis, we have
  \begin{align*}
    T_{2, j+1}X_{j-1}
     & =T_{2, j}X_1\Sigma_{2, 3} X_{j-1}               \\
     & =T_{2, j}X_1 X_{j-1}\Sigma_{2, 3}               \\
     & =T_{2, j}X_{j-2} X_1\Sigma_{2, 3}               \\
     & =X_2T_{2, j+1}\Sigma_{j-1, j}X_1\Sigma_{2, 3}   \\
     & =X_2T_{2, j+1}X_1 \Sigma_{j, j+1}\Sigma_{2, 3}  \\
     & =X_2T_{2, j+1}X_1 \Sigma_{2, 3} \Sigma_{j, j+1} \\
     & = X_2 T_{2, j+2} \Sigma_{j, j+1},
  \end{align*}
  which completes the proof of (2) for the case $i=2$.

  For $i \geq 3$, both cases (1) and (2) reduce to the case of $i=2$ by Lemma \ref{lemma_TS_conj_x0}.
  This completes the proof for all cases.
\end{proof}
We now conclude the proof for the case of $i \geq 2$.
\begin{proposition}\label{proposition_TD3_igeq2}
  For every $2 \leq i \leq k \leq j-2$, we have
  \begin{align*}
    T_{i, j}X_k=X_{i+j-k-2}T_{i, j+1} \Sigma_{k+1, k+2}.
  \end{align*}
\end{proposition}
\begin{proof}
  We show this by induction on $d \coloneqq j-i$.
  If $d=2$, then it follows from Proposition \ref{proposition_TD3_exceptional_case} (2) since $k=i=j-2$ holds.
  If $d=3$, then $k=i=j-3$ or $k=i+1=j-2$ holds.
  Hence, again by Proposition \ref{proposition_TD3_exceptional_case} (1) and (2), we have the desired equality.

  Assume that $T_{i, j}X_k=X_{i+j-k-2}T_{i, j+1} \Sigma_{k+1, k+2}$ holds for any $2 \leq i \leq k \leq j - 2$ satisfying $j - i = d$ for some $d \geq 2$.
  We then show that $T_{i, j}X_k=X_{i+j-k-2}T_{i, j+1} \Sigma_{k+1, k+2}$ holds for any $2 \leq i \leq k \leq j - 2$ satisfying $j - i = d+2$.
  Let $i, j$, and $k$ be arbitrary integers satisfying $2 \leq i \leq k \leq j-2$ and $j-i=d+2$.
  First, we consider the case $k \leq j-4$.
  In this case, by Theorems \ref{theorem_TD2},  \ref{theorem_TD5}, Lemma \ref{lemma_commute_Sigmaii+1Xm}, Proposition \ref{proposition_commute_Sigmaii+1Tkl}, and the inductive hypothesis, we have
  \begin{align*}
    T_{i, j}X_k
     & = X_{i-1}^{-2}T_{i, j-2} X_{j-3} T_{j-2, j} X_k                                 \\
     & = X_{i-1}^{-2}T_{i, j-2} X_{j-3}  X_k T_{j-1, j+1}                              \\
     & = X_{i-1}^{-2} T_{i, j-2} X_k X_{j-2} T_{j-1, j+1}                              \\
     & = X_{i-1}^{-2} X_{i+(j-2)-k-2}T_{i, j-1} \Sigma_{k+1, k+2} X_{j-2} T_{j-1, j+1} \\
     & = X_{i+j-k-2} X_{i-1}^{-2} T_{i, j-1} \Sigma_{k+1, k+2} X_{j-2} T_{j-1, j+1}    \\
     & = X_{i+j-k-2} X_{i-1}^{-2} T_{i, j-1}  X_{j-2}\Sigma_{k+1, k+2} T_{j-1, j+1}    \\
     & =  X_{i+j-k-2} X_{i-1}^{-2} T_{i, j-1}  X_{j-2}T_{j-1, j+1}\Sigma_{k+1, k+2}    \\
     & = X_{i+j-k-2} T_{i, j+1} \Sigma_{k+1, k+2},
  \end{align*}
  which is the desired equality.
  For the cases of $k=j-3, j-2$, by Proposition \ref{proposition_TD3_exceptional_case}, we also obtain the desired result.
\end{proof}
It remains to consider the case where $i = 1$.
\begin{proposition} \label{proposition_TD3_i=1}
  For any $1 \leq k \leq j-2$, we have
  \begin{align*}
    T_{1, j}X_k=X_{1+j-k-2} T_{1, j+1} \Sigma_{k+1, k+2}.
  \end{align*}
\end{proposition}
\begin{proof}
  We show this by induction on $k$.
  We first assume that $k=1$ and we prove this base case by induction on $j \geq 3$.
  For the case $j=3, 4$, it follows from \relref{TD3} in $\cR$.
  Assume that $T_{1, j}X_1=X_{j-2} T_{1, j+1} \Sigma_{2, 3}$ holds for some $j \geq 3$.
  Then, completely analogous to the proof of Proposition \ref{proposition_TD3_igeq2} (which we reproduce here for the reader's convenience as we are dealing with $i=1$), Theorems \ref{theorem_TD2},  \ref{theorem_TD5}, Lemma \ref{lemma_commute_Sigmaii+1Xm} (or \ref{lemma_commute_Sigma23Xm}), Proposition \ref{proposition_commute_Sigmaii+1Tkl} (or Lemma \ref{lemma_commute_Sigmaii+1Tkk+2} or \ref{lemma_commute_Sigma23Tm+1m+3}), and the inductive hypothesis yield the following:
  \begin{align*}
    T_{1, j+2}X_1
     & = X_{0}^{-2}T_{1, j} X_{j-1}T_{j, j+2}X_1                    \\
     & = X_{0}^{-2}T_{1, j} X_{j-1}X_1 T_{j+1, j+3}                 \\
     & = X_0^{-2} T_{1, j} X_1 X_j T_{j+1, j+3}                     \\
     & = X_0^{-2} X_{j-2} T_{1, j+1} \Sigma_{2, 3} X_j T_{j+1, j+3} \\
     & = X_j X_0^{-2} T_{1, j+1} \Sigma_{2, 3} X_j T_{j+1, j+3}     \\
     & = X_j X_0^{-2} T_{1, j+1}  X_j \Sigma_{2, 3} T_{j+1, j+3}    \\
     & = X_j X_0^{-2} T_{1, j+1}  X_j  T_{j+1, j+3}\Sigma_{2, 3}    \\
     & = X_j T_{1, j+3} \Sigma_{2, 3},
  \end{align*}
  which completes the proof of the case $k=1$.

  Next, we assume that for some $k \geq 1$, we have $T_{1, j}X_k=X_{1+j-k-2}T_{1, j+1}\Sigma_{k+1, k+2}$ for all $j \geq k+2$.
  We will show that $T_{1, j}X_{k+1}=X_{1+j-(k+1)-2}T_{1, j+1}\Sigma_{(k+1)+1, (k+1)+2}$ also holds for any $j \geq (k+1)+2$.

  By Theorems \ref{theorem_TD4}, \ref{theorem_TC3} (or Lemma \ref{lemma_TS_conj_x0}), Lemma \ref{lemma_commute_Sigma12Xm}, Proposition \ref{proposition_commute_adjacent_SigmaSigma}, and inductive hypothesis, we have
  \begin{align*}
    T_{1, j}X_{k+1}
     & =T_{1, j-1}X_0\Sigma_{1, 2} X_{k+1}                               \\
     & =T_{1, j-1}X_0 X_{k+1}\Sigma_{1, 2}                               \\
     & =T_{1, j-1} X_k X_0 \Sigma_{1, 2}                                 \\
     & =X_{1+(j-1)-k-2}T_{1, (j-1)+1}\Sigma_{k+1, k+2} X_0 \Sigma_{1, 2} \\
     & =X_{1+j-(k+1)-2}T_{1, j}X_0 \Sigma_{k+2, k+3} \Sigma_{1, 2}       \\
     & =X_{1+j-(k+1)-2}T_{1, j}X_0 \Sigma_{1, 2} \Sigma_{k+2, k+3}       \\
     & =X_{1+j-(k+1)-2} T_{1, j+1} \Sigma_{k+2, k+3},
  \end{align*}
  which is the desired result.
\end{proof}
From Propositions \ref{proposition_TD3_igeq2} and \ref{proposition_TD3_i=1}, we obtain:
\begin{theorem} \label{theorem_TD3}
  For any $1 \leq i \leq k \leq j-2$, we have
  \begin{align*}
    T_{i, j}X_k=X_{i+j-k-2}T_{i, j+1} \Sigma_{k+1, k+2}. \tag{\relref{TD3}}
  \end{align*}
\end{theorem}
\subsection[Propagation of relation (TC2)]{Propagation of \relref{TC2}}
In this subsection, we show that for any $i-1 \leq k \leq j-1$ with $1 \leq i<j$,  we have
\begin{align*}
  \Sigma_{i, j}X_k=X_{i+j-k-2} \Sigma_{i, j+1}\Sigma_{k+1, k+2}.
\end{align*}
Before proving this, we first provide partial propagation for several related relations.
\subsubsection[Propagation of relation (TB7)]{Partial propagation of \relref{TB7}}
In this subsubsection, we show that for any $1 \leq i <k<k+1<j$, we have
\begin{align*}
  T_{i, j}\Sigma_{k, k+1}=\Sigma_{i+j-k-1, i+j-k}T_{i, j}.
\end{align*}
For this proof, we prepare two lemmas about the induction step.
\begin{lemma}\label{lemma_adjacent_TB7_inductive_step1}
  Assume that
  \begin{align*}
    T_{i, j}\Sigma_{k, k+1}=\Sigma_{i+j-k-1, i+j-k}T_{i, j}.
  \end{align*}
  holds for some $1 \leq i <k<k+1<j$.
  Then, we have
  \begin{align*}
    T_{i, j+2}\Sigma_{k, k+1}=\Sigma_{i+(j+2)-k-1, i+(j+2)-k}T_{i, j+2}.
  \end{align*}
\end{lemma}
\begin{proof}
  By Theorems \ref{theorem_TD2}, \ref{theorem_TC3}, Proposition \ref{proposition_commute_Sigmaii+1Tkl}, Lemma \ref{lemma_commute_Sigmaii+1Xm} and the assumption, we have
  \begin{align*}
    T_{i, j+2} \Sigma_{k, k+1}
     & =X_{i-1}^{-2}T_{i, j}X_{j-1}T_{j, j+2} \Sigma_{k, k+1}          \\
     & =X_{i-1}^{-2}T_{i, j}X_{j-1} \Sigma_{k, k+1}T_{j, j+2}          \\
     & =X_{i-1}^{-2}T_{i, j} \Sigma_{k, k+1}X_{j-1}T_{j, j+2}          \\
     & =X_{i-1}^{-2}\Sigma_{i+j-k-1, i+j-k}T_{i, j}X_{j-1}T_{j, j+2}   \\
     & =\Sigma_{i+j-k+1, i+j-k+2}X_{i-1}^{-2}T_{i, j}X_{j-1}T_{j, j+2} \\
     & =\Sigma_{i+j-k+1, i+j-k+2}T_{i, j+2},
  \end{align*}
  which completes the proof.
\end{proof}
\begin{lemma}\label{lemma_adjacent_TB7_inductive_step2}
  Assume that
  \begin{align*}
    T_{i, j}\Sigma_{k, k+1}=\Sigma_{i+j-k-1, i+j-k}T_{i, j}.
  \end{align*}
  holds for some $1 \leq i <k<k+1<j$.
  Then, we have
  \begin{align*}
    T_{i, j+1}\Sigma_{k+1, k+2}=\Sigma_{i+j-k-1, i+j-k}T_{i, j+1}.
  \end{align*}
\end{lemma}
\begin{proof}
  By Theorems \ref{theorem_TD4}, \ref{theorem_TC3}, Proposition \ref{proposition_commute_adjacent_SigmaSigma}, and the assumption, we have
  \begin{align*}
    T_{i, j+1} \Sigma_{k+1, k+2}
     & = T_{i, j}X_{i-1}\Sigma_{i, i+1} \Sigma_{k+1, k+2}      \\
     & = T_{i, j}X_{i-1} \Sigma_{k+1, k+2}\Sigma_{i, i+1}      \\
     & = T_{i, j} \Sigma_{k, k+1}X_{i-1}\Sigma_{i, i+1}        \\
     & = \Sigma_{i+j-k-1, i+j-k}T_{i, j}X_{i-1}\Sigma_{i, i+1} \\
     & = \Sigma_{i+j-k-1, i+j-k}T_{i, j+1},
  \end{align*}
  which completes the proof.
\end{proof}
With these two lemmas, we can now show the desired equation and finish this subsubsection.
\begin{proposition}\label{proposition_adjacent_TB7}
  For any $1 \leq i <k<k+1<j$, we have
  \begin{align*}
    T_{i, j}\Sigma_{k, k+1}=\Sigma_{i+j-k-1, i+j-k}T_{i, j}.
  \end{align*}
\end{proposition}
\begin{proof}
  We first note that by \relref{TB7} in $\cR$, we have
  \begin{alignat*}{2}
    T_{1, 4}\Sigma_{2, 3} & =\Sigma_{2, 3}T_{1, 4},
                          &\quad T_{1, 5}\Sigma_{2, 3}   & =\Sigma_{3, 4}T_{1, 5}, \\
    T_{2, 5}\Sigma_{3, 4} & =\Sigma_{3, 4}T_{2, 5},
                          &\quad T_{2, 6}\Sigma_{3, 4}   & =\Sigma_{4, 5}T_{2, 6}.
  \end{alignat*}

  For $i=1$, let $d \coloneqq j-k \geq 2$.
  We first consider the case where $k=2$.
  In this case, we have $j=d+2$, and the desired equation is
  \begin{align*}
    T_{1, d+2}\Sigma_{2, 3}=\Sigma_{d, d+1}T_{1, d+2}.
  \end{align*}
  If $d$ is even, then we can write $d=2+2n$ for some integer $n \geq 0$.
  The case $n=0$ corresponds to the equality
  \begin{align*}
    T_{1, 4}\Sigma_{2, 3}=\Sigma_{2, 3}T_{1, 4}.
  \end{align*}
  By applying Lemma \ref{lemma_adjacent_TB7_inductive_step1} $n$ times starting from this equality, we obtain
  \begin{align*}
    T_{1, 4+2n}\Sigma_{2, 3}=\Sigma_{2+2n, 3+2n}T_{1, 4+2n}.
  \end{align*}
  Since $4+2n=d+2$, $2+2n=d$, and $3+2n=d+1$, we obtain
  \begin{align*}
    T_{1, d+2}\Sigma_{2, 3}=\Sigma_{d, d+1}T_{1, d+2}.
  \end{align*}
  If $d$ is odd, then we can write $d=3+2n$ for some integer $n \geq 0$.
  The case $n=0$ corresponds to the equality
  \begin{align*}
    T_{1, 5}\Sigma_{2, 3}=\Sigma_{3, 4}T_{1, 5}.
  \end{align*}
  By applying Lemma \ref{lemma_adjacent_TB7_inductive_step1} $n$ times starting from this equality, we obtain
  \begin{align*}
    T_{1, 5+2n}\Sigma_{2, 3}=\Sigma_{3+2n, 4+2n}T_{1, 5+2n}.
  \end{align*}
  Since $5+2n=d+2$, $3+2n=d$, and $4+2n=d+1$, we obtain
  \begin{align*}
    T_{1, d+2}\Sigma_{2, 3}=\Sigma_{d, d+1}T_{1, d+2}.
  \end{align*}
  Hence, in the case where $i=1$ and $k=2$, we have
  \begin{align*}
    T_{1, d+2}\Sigma_{2, 3}=\Sigma_{d, d+1}T_{1, d+2}
  \end{align*}
  for every $d \geq 2$.

  Now we return to general $k$.
  By applying Lemma \ref{lemma_adjacent_TB7_inductive_step2} $k-2$ times to the above equality, we have
  \begin{align*}
    T_{1, d+k}\Sigma_{k, k+1}=\Sigma_{d, d+1}T_{1, d+k}.
  \end{align*}
  Since $d=j-k$, we have
  \begin{align*}
    T_{1, j}\Sigma_{k, k+1}=\Sigma_{j-k, j-k+1}T_{1, j}.
  \end{align*}
  Hence, the proof for the case of $i=1$ is completed.

  Next, we prove the case where $i=2$.
  As in the case of $i=1$, let $d \coloneqq j-k \geq 2$.
  We first consider the case where $k=3$.
  In this case, we have $j=d+3$, and the desired equation is
  \begin{align*}
    T_{2, d+3}\Sigma_{3, 4}=\Sigma_{d+1, d+2}T_{2, d+3}.
  \end{align*}
  If $d$ is even, then this follows from
  \begin{align*}
    T_{2, 5}\Sigma_{3, 4}=\Sigma_{3, 4}T_{2, 5}
  \end{align*}
  by applying Lemma \ref{lemma_adjacent_TB7_inductive_step1} repeatedly.
  If $d$ is odd, then the same argument starting from
  \begin{align*}
    T_{2, 6}\Sigma_{3, 4}=\Sigma_{4, 5}T_{2, 6}
  \end{align*}
  gives the same conclusion.
  Hence, in the case where $i=2$ and $k=3$, we have
  \begin{align*}
    T_{2, d+3}\Sigma_{3, 4}=\Sigma_{d+1, d+2}T_{2, d+3}
  \end{align*}
  for every $d \geq 2$.

  Now we return to general $k$.
  By applying Lemma \ref{lemma_adjacent_TB7_inductive_step2} $k-3$ times to the above equality, we have
  \begin{align*}
    T_{2, d+k}\Sigma_{k, k+1}=\Sigma_{d+1, d+2}T_{2, d+k}.
  \end{align*}
  Since $d=j-k$, this equality is equivalent to
  \begin{align*}
    T_{2, j}\Sigma_{k, k+1}=\Sigma_{j-k+1, j-k+2}T_{2, j}.
  \end{align*}
  Hence, the proof for the case of $i=2$ is completed.

  Finally, we prove the case where $i \geq 3$.
  By applying Lemma \ref{lemma_TS_conj_x0}, the desired equality is equivalent to
  \begin{align*}
    T_{2, j-(i-2)}\Sigma_{k-(i-2), k+1-(i-2)}
    =
    \Sigma_{j-k+1, j-k+2}T_{2, j-(i-2)},
  \end{align*}
  which is already proved in the case $i=2$ above.
  Hence we complete the proof.
\end{proof}

\subsubsection[Propagation of relation (TB6)]{Partial propagation of \relref{TB6}}
In this subsubsection, we show the following proposition:
\begin{proposition}\label{proposition_adjacent_TB6}
  For any $2 \leq i+1<j$, we have
  \begin{align*}
    T_{i, j}\Sigma_{i, i+1}=T_{j-1, j}T_{i, j}.
  \end{align*}
\end{proposition}
\begin{proof}
  By \relref{TB6} in $\cR$, we have
  \begin{alignat*}{2}
    T_{1, 3}\Sigma_{1, 2} & =T_{2, 3}T_{1, 3},
                          &\quad
    T_{1, 4}\Sigma_{1, 2} & =T_{3, 4}T_{1, 4}, \\
    T_{2, 4}\Sigma_{2, 3} & =T_{3, 4}T_{2, 4},
                          &\quad
    T_{2, 5}\Sigma_{2,3}  & =T_{4, 5}T_{2, 5}.
  \end{alignat*}
  We first note that if
  \begin{align*}
    T_{i, j}\Sigma_{i, i+1}=T_{j-1, j}T_{i, j}
  \end{align*}
  holds, then
  \begin{align*}
    T_{i, j+2}\Sigma_{i, i+1}=T_{j+1, j+2}T_{i, j+2}
  \end{align*}
  also holds.
  Indeed, by Theorems \ref{theorem_TD2}, \ref{theorem_TD5}, Proposition \ref{proposition_commute_Sigmaii+1Tkl}, Lemma \ref{lemma_commute_Sigmaii+1Xm}, and the assumption, we have
  \begin{align*}
    T_{i, j+2} \Sigma_{i, i+1}
     & = X_{i-1}^{-2}T_{i, j}X_{j-1}T_{j, j+2} \Sigma_{i, i+1} \\
     & = X_{i-1}^{-2}T_{i, j}X_{j-1}\Sigma_{i, i+1} T_{j, j+2} \\
     & = X_{i-1}^{-2}T_{i, j}\Sigma_{i, i+1}X_{j-1} T_{j, j+2} \\
     & = X_{i-1}^{-2}T_{j-1, j}T_{i, j}X_{j-1} T_{j, j+2}      \\
     & = T_{j+1, j+2}X_{i-1}^{-2} T_{i, j}X_{j-1} T_{j, j+2}   \\
     & = T_{j+1, j+2}T_{i, j+2}.
  \end{align*}
  When $i=1,2$, by applying this implication iteratively to the above relations in $\cR$, we have
  \begin{align*}
    T_{i, j} \Sigma_{i, i+1} = T_{j-1, j}T_{i, j}
  \end{align*}
  for every $j > i+1$.

  Finally, let $i \geq 3$ and $j >i+1$.
  Then by Lemma \ref{lemma_TS_conj_x0}, the desired equality is equivalent to
  \begin{align*}
    T_{2, j-(i-2)}\Sigma_{2, 3}=T_{j-1-(i-2), j-(i-2)}T_{2, j-(i-2)}.
  \end{align*}
  Hence, by the case of $i=2$, we complete the proof.
\end{proof}
\subsubsection[Propagation of relation (TB3)]{Partial propagation of \relref{TB3}}
In this subsubsection, we show that
\begin{align*}
  \Sigma_{i, j}\Sigma_{k, k+1}=\Sigma_{i+j-k-1, i+j-k}\Sigma_{i, j}
\end{align*}
holds for any $1 \leq i \leq k<k+1 \leq j$.
Although this result is not needed in the proof of \relref{TC2}, it is a natural consequence of the two preceding subsubsections, and we record it here for later use.

We divide the proof into several cases according to the indices.
\begin{lemma}\label{lemma_adjacent_TB3_inner}
  For any $1 \leq i<k<k+1<j$, we have
  \begin{align*}
    \Sigma_{i, j}\Sigma_{k, k+1}=\Sigma_{i+j-k-1, i+j-k}\Sigma_{i, j}.
  \end{align*}
\end{lemma}
\begin{proof}
  By the definition of $\Sigma_{i, j}$ and Proposition \ref{proposition_adjacent_TB7}, we have
  \begin{align*}
    \Sigma_{i, j}\Sigma_{k, k+1}
     & = T_{i, j+1}T_{i+1, j+1}T_{i, j+1}\Sigma_{k, k+1}         \\
     & = T_{i, j+1}T_{i+1, j+1}\Sigma_{i+j-k,i+j-k+1}T_{i, j+1}  \\
     & = T_{i, j+1} \Sigma_{k+1, k+2}T_{i+1, j+1}T_{i, j+1}      \\
     & = \Sigma_{i+j-k-1, i+j-k}T_{i, j+1}T_{i+1, j+1}T_{i, j+1} \\
     & = \Sigma_{i+j-k-1, i+j-k}\Sigma_{i, j},
  \end{align*}
  which completes the proof.
\end{proof}
\begin{lemma} \label{lemma_adjacent_TB3_left}
  For any $1 \leq i < j$, we have
  \begin{align*}
    \Sigma_{i, j}\Sigma_{i, i+1}=\Sigma_{j-1, j}\Sigma_{i, j}.
  \end{align*}
\end{lemma}
\begin{proof}
  It is clear when $j=i+1$.
  Assume $j > i+1$.
  By the definition of $\Sigma_{i, j}$ and Propositions \ref{proposition_adjacent_TB6}, \ref{proposition_adjacent_TB7}, we have
  \begin{align*}
    \Sigma_{i, j}\Sigma_{i,i+1}
     & = T_{i, j+1} T_{i+1, j+1} T_{i, j+1} \Sigma_{i, i+1}  \\
     & = T_{i, j+1} T_{i+1, j+1} T_{j, j+1} T_{i, j+1}       \\
     & = T_{i, j+1} \Sigma_{i+1, i+2}T_{i+1, j+1} T_{i, j+1} \\
     & = \Sigma_{j-1, j}T_{i, j+1} T_{i+1, j+1} T_{i, j+1}   \\
     & =\Sigma_{j-1, j} \Sigma_{i, j},
  \end{align*}
  which completes the proof.
\end{proof}
By combining these lemmas, we now show the desired equation:
\begin{proposition} \label{proposition_adjacent_TB3}
  For any  $1 \leq i \leq k<k+1 \leq j$, we have
  \begin{align*}
    \Sigma_{i, j}\Sigma_{k, k+1}=\Sigma_{i+j-k-1, i+j-k}\Sigma_{i, j}.
  \end{align*}
\end{proposition}
\begin{proof}
  If $i < k$ and $k+1<j$, then it follows from Lemma \ref{lemma_adjacent_TB3_inner}.
  If $i=k$, it follows from Lemma \ref{lemma_adjacent_TB3_left}.
  If $j=k+1$, again it follows from Lemma \ref{lemma_adjacent_TB3_left} since by taking the inverse of both sides, we have
  \begin{align*}
    \Sigma_{i, j}\Sigma_{j-1, j}=\Sigma_{i, i+1}\Sigma_{i, j},
  \end{align*}
  which is the desired equality.
\end{proof}
\subsubsection[Propagation of relation (TC2)]{Propagation of \relref{TC2}}
Using the obtained results, we now prove the desired equality of this subsection.
We divide the proof into several lemmas according to the indices.

\begin{lemma}\label{lemma_TC2_inner}
  For any $1 \leq i \leq k \leq j-2$, we have
  \begin{align*}
    \Sigma_{i, j}X_k=X_{i+j-k-2}\Sigma_{i, j+1}\Sigma_{k+1, k+2}.
  \end{align*}
\end{lemma}
\begin{proof}
  By the definition of $\Sigma_{i, j}$, Theorem \ref{theorem_TD3}, and Proposition \ref{proposition_adjacent_TB7}, we have
  \begin{align*}
    \Sigma_{i, j}X_k
     & =T_{i, j+1}T_{i+1, j+1}T_{i, j+1}X_k                                                                             \\
     & =T_{i, j+1}T_{i+1, j+1}X_{i+j+1-k-2}T_{i, j+2}\Sigma_{k+1, k+2}                                                  \\
     & =T_{i, j+1}X_{(i+1)+(j+1)-(i+j+1-k-2)-2}T_{i+1, j+2}\Sigma_{i+j+1-k-2+1, i+j+1-k-2+2}T_{i, j+2}\Sigma_{k+1, k+2} \\
     & =T_{i, j+1}X_{k+1}T_{i+1, j+2}\Sigma_{i+j-k, i+j-k+1}T_{i, j+2}\Sigma_{k+1, k+2}                                 \\
     & =X_{i+j+1-(k+1)-2}T_{i, j+2}\Sigma_{k+2, k+3}T_{i+1, j+2}\Sigma_{i+j-k, i+j-k+1}T_{i, j+2}\Sigma_{k+1, k+2}      \\
     & =X_{i+j-k-2}T_{i, j+2}T_{i+1, j+2}T_{i, j+2}\Sigma_{k+1, k+2}                                                    \\
     & = X_{i+j-k-2}\Sigma_{i, j+1}\Sigma_{k+1, k+2},
  \end{align*}
  which completes the proof.
\end{proof}
Next, we consider the case $k=i-1$.
\begin{lemma}\label{lemma_TC2_left}
  For any $1 \leq i <j$, we have
  \begin{align*}
    \Sigma_{i, j}X_{i-1}=X_{j-1}\Sigma_{i, j+1}\Sigma_{i, i+1}.
  \end{align*}
\end{lemma}
\begin{proof}
  By the definition of $\Sigma_{i, j}$, Theorems \ref{theorem_TD4} and \ref{theorem_TD3}, we have
  \begin{align*}
    \Sigma_{i, j}X_{i-1}
     & =T_{i, j+1}T_{i+1, j+1}T_{i, j+1} X_{i-1}                                                         \\
     & =T_{i, j+1}T_{i+1, j+1}T_{i, j+2}\Sigma_{i, i+1}                                                  \\
     & =T_{i, j+1} X_i \Sigma_{i+1, i+2}T_{i+1, j+2} T_{i, j+2}\Sigma_{i, i+1}                           \\
     & =X_{i+j+1-i-2}T_{i, j+2} \Sigma_{i+1, i+2}\Sigma_{i+1, i+2}T_{i+1, j+2} T_{i, j+2}\Sigma_{i, i+1} \\
     & =X_{j-1}T_{i, j+2}T_{i+1, j+2} T_{i, j+2}\Sigma_{i, i+1}                                          \\
     & =X_{j-1}\Sigma_{i, j+1}\Sigma_{i, i+1},
  \end{align*}
  which completes the proof.
\end{proof}
Finally, we consider the case $k=j-1$.
\begin{lemma} \label{lemma_TC2_right}
  For any $1 \leq i <j$, we have
  \begin{align*}
    \Sigma_{i, j}X_{j-1}=X_{i-1} \Sigma_{i, j+1}\Sigma_{j, j+1}.
  \end{align*}
\end{lemma}
\begin{proof}
  By the definition of $\Sigma_{i, j}$, Theorems \ref{theorem_TD3}, \ref{theorem_TD4}, and Proposition \ref{proposition_adjacent_TB6}, we have
  \begin{align*}
    \Sigma_{i, j}X_{j-1}
     & = T_{i, j+1} T_{i+1, j+1} T_{i, j+1} X_{j-1}                                                 \\
     & = T_{i, j+1} T_{i+1, j+1} X_{i+j+1-(j-1)-2} T_{i, j+2} \Sigma_{j, j+1}                       \\
     & = T_{i, j+1} T_{i+1, j+1} X_{i} T_{i, j+2} \Sigma_{j, j+1}                                   \\
     & = T_{i, j+1} T_{i+1, j+2} \Sigma_{i+1, i+2} T_{i, j+2} \Sigma_{j, j+1}                       \\
     & = X_{i-1}\Sigma_{i, i+1}T_{i, j+2} T_{i+1, j+2} \Sigma_{i+1, i+2} T_{i, j+2} \Sigma_{j, j+1} \\
     & = X_{i-1} T_{i, j+2}T_{j+1, j+2} T_{i+1, j+2} \Sigma_{i+1, i+2} T_{i, j+2} \Sigma_{j, j+1}   \\
     & = X_{i-1} T_{i, j+2}T_{j+1, j+2} T_{j+1, j+2}T_{i+1, j+2}T_{i, j+2} \Sigma_{j, j+1}          \\
     & = X_{i-1} T_{i, j+2} T_{i+1, j+2}T_{i, j+2} \Sigma_{j, j+1}                                  \\
     & = X_{i-1} \Sigma_{i, j+1} \Sigma_{j, j+1},
  \end{align*}
  which completes the proof.
\end{proof}

By combining Lemmas \ref{lemma_TC2_inner}, \ref{lemma_TC2_left}, and \ref{lemma_TC2_right}, we obtain:
\begin{theorem}\label{theorem_TC2}
  For any $i-1 \leq k \leq j-1$ with $1 \leq i<j$,  we have
  \begin{align*}
    \Sigma_{i, j}X_k=X_{i+j-k-2} \Sigma_{i, j+1}\Sigma_{k+1, k+2}. \tag{\relref{TC2}}
  \end{align*}
\end{theorem}
\subsection[Propagation of relation (TB5)]{Propagation of \relref{TB5}}
In this subsection, we show the following:
\begin{theorem} \label{theorem_TB5}
  For any $1 \leq i<j<k<l$, we have
  \begin{align*}
    \Sigma_{i, j}T_{k, l}=T_{k, l}\Sigma_{i, j}. \tag{\relref{TB5}}
  \end{align*}
\end{theorem}
\begin{proof}
  Let $d \coloneqq j-i$.
  We show this by induction on $d$.

  If $d=1$, then $j=i+1$ holds.
  By Proposition \ref{proposition_commute_Sigmaii+1Tkl}, we have
  \begin{align*}
    \Sigma_{i, i+1}T_{k, l}=T_{k, l}\Sigma_{i, i+1},
  \end{align*}
  which completes the case $d=1$.

  Assume that $d \geq 2$ and the equality holds for all indices with $j-i=d-1$.
  By Theorems \ref{theorem_TC2}, \ref{theorem_TD5}, Proposition \ref{proposition_commute_Sigmaii+1Tkl}, and the induction hypothesis, we have
  \begin{align*}
    \Sigma_{i, j}T_{k, l}
     & = X_{j-3}^{-1}\Sigma_{i, j-1}X_i \Sigma_{i+1, i+2}T_{k, l}      \\
     & =  X_{j-3}^{-1}\Sigma_{i, j-1}X_i T_{k, l}\Sigma_{i+1, i+2}     \\
     & = X_{j-3}^{-1}\Sigma_{i, j-1} T_{k-1, l-1}X_i \Sigma_{i+1, i+2} \\
     & = X_{j-3}^{-1} T_{k-1, l-1}\Sigma_{i, j-1}X_i \Sigma_{i+1, i+2} \\
     & = T_{k, l}X_{j-3}^{-1}\Sigma_{i, j-1}X_i \Sigma_{i+1, i+2}      \\
     & = T_{k, l}\Sigma_{i, j},
  \end{align*}
  which completes the proof.
\end{proof}
\subsection[Propagation of relation (TB2)]{Propagation of \relref{TB2}}
In this subsection, we show the following:
\begin{theorem} \label{theorem_TB2}
  For any $1 \leq i<j<k<l$, we have
  \begin{align*}
    \Sigma_{i, j}\Sigma_{k, l}=\Sigma_{k, l}\Sigma_{i, j} \tag{\relref{TB2}}
  \end{align*}
\end{theorem}
\begin{proof}
  By the definition of $\Sigma_{k, l}$ and Theorem \ref{theorem_TB5}, we have
  \begin{align*}
    \Sigma_{i, j}\Sigma_{k, l}
     & = \Sigma_{i, j}T_{k, l+1}T_{k+1, l+1}T_{k, l+1} \\
     & = T_{k, l+1}\Sigma_{i, j}T_{k+1, l+1}T_{k, l+1} \\
     & = T_{k, l+1}T_{k+1, l+1}\Sigma_{i, j}T_{k, l+1} \\
     & = T_{k, l+1}T_{k+1, l+1}T_{k, l+1}\Sigma_{i, j} \\
     & =\Sigma_{k, l}\Sigma_{i, j},
  \end{align*}
  which completes the proof.
\end{proof}
\subsection[Propagation of relation (TC1)]{Propagation of \relref{TC1}}
In this subsection, we show the following:
\begin{theorem} \label{theorem_TC1}
  For any $1 \leq i < j \leq k$, we have
  \begin{align*}
    \Sigma_{i, j}X_k=X_k\Sigma_{i, j}. \tag{\relref{TC1}}
  \end{align*}
\end{theorem}
\begin{proof}
  Let $d \coloneqq j-i$.
  We show this by induction on $d$.
  If $d=1$, then it follows from Lemma \ref{lemma_commute_Sigmaii+1Xm} since $j=i+1$.

  Assume that $d \geq 2$ and the equality holds for all indices with $j-i=d-1$.
  By Theorem \ref{theorem_TC2}, Lemma \ref{lemma_commute_Sigmaii+1Xm}, and the inductive hypothesis, we have
  \begin{align*}
    \Sigma_{i, j}X_k & = X_{j-3}^{-1}\Sigma_{i, j-1}X_i \Sigma_{i+1, i+2} X_k    \\
                     & = X_{j-3}^{-1}\Sigma_{i, j-1}X_i X_k \Sigma_{i+1, i+2}    \\
                     & = X_{j-3}^{-1}\Sigma_{i, j-1}X_{k-1}X_i \Sigma_{i+1, i+2} \\
                     & = X_{j-3}^{-1}X_{k-1}\Sigma_{i, j-1}X_i \Sigma_{i+1, i+2} \\
                     & = X_{k}X_{j-3}^{-1}\Sigma_{i, j-1}X_i \Sigma_{i+1, i+2}   \\
                     & = X_k \Sigma_{i, j},
  \end{align*}
  which completes the proof.
\end{proof}
\subsection[Propagation of relation (TB6)]{Propagation of \relref{TB6}}
In this subsection, we prove \relref{TB6}.
We show it by induction on $j-l$.
The proof splits into two cases, depending on the parity of the gap.
The odd case is direct, while the even case is reduced to an auxiliary equality.

We first prove the step used in the induction.
\begin{lemma} \label{lemma_TB6_inductive_step}
  For any $1 \leq i < l <j$, if
  \begin{align*}
    T_{i, j}\Sigma_{i, l}=T_{i+j-l, j}T_{i, j}
  \end{align*}
  holds, then
  \begin{align*}
    T_{i, j+2}\Sigma_{i, l}=T_{i+j+2-l, j+2}T_{i, j+2}
  \end{align*}
  also holds.
\end{lemma}
\begin{proof}
  By Theorems \ref{theorem_TD2}, \ref{theorem_TB5}, \ref{theorem_TC1}, and \ref{theorem_TD5}, together with the assumption, we have
  \begin{align*}
    T_{i, j+2}\Sigma_{i, l}
     & = X_{i-1}^{-2}T_{i, j}X_{j-1}T_{j, j+2}\Sigma_{i, l}    \\
     & = X_{i-1}^{-2}T_{i, j}X_{j-1} \Sigma_{i, l}T_{j, j+2}   \\
     & = X_{i-1}^{-2}T_{i, j}\Sigma_{i, l}X_{j-1}T_{j, j+2}    \\
     & = X_{i-1}^{-2}T_{i+j-l, j}T_{i, j}X_{j-1}T_{j, j+2}     \\
     & = T_{i+j-l+2, j+2}X_{i-1}^{-2}T_{i, j}X_{j-1}T_{j, j+2} \\
     & = T_{i+j+2-l, j+2}T_{i, j+2},
  \end{align*}
  which completes the proof.
\end{proof}

Using this lemma, we first prove the case where $j-l$ is odd.
\begin{lemma} \label{lemma_TB6_odd_case}
  For any $1 \leq i <l < j$ with $j-l$ odd, we have
  \begin{align*}
    T_{i, j}\Sigma_{i, l}=T_{i+j-l, j}T_{i, j}.
  \end{align*}
\end{lemma}
\begin{proof}
  We first note that
  \begin{align*}
    T_{i, l+1}\Sigma_{i, l}=T_{i, l+1}T_{i, l+1}T_{i+1, l+1}T_{i, l+1}=T_{i+1, l+1}T_{i, l+1}
  \end{align*}
  holds by the definition of $\Sigma_{i, l}$.
  If $j-l$ is odd, then there exists an integer $n \geq 0$ such that $j=l+1+2n$.
  By applying Lemma \ref{lemma_TB6_inductive_step} $n$ times to the above equality, we obtain the desired result.
\end{proof}

Next, we prove the auxiliary equality needed for the even case.
\begin{proposition} \label{proposition_auxiliary_TB6}
  For any $1 \leq i < r <j$ with $r-i$ even, we have
  \begin{align*}
    T_{i, j}T_{r, j}=\Sigma_{i, i+j-r}T_{i, j}.
  \end{align*}
\end{proposition}
\begin{proof}
  We first consider the case $i=1$ and $r=3$.
  By Theorems \ref{theorem_TD4}, \ref{theorem_TD5}, Proposition \ref{proposition_commute_Sigmaii+1Tkl}, and the definition of $\Sigma_{1, j-2}$, we have
  \begin{align*}
    T_{1, j}T_{3, j}
     & = T_{1, j-1}X_0\Sigma_{1, 2}T_{3, j}       \\
     & = T_{1, j-1}X_0 T_{3, j}\Sigma_{1, 2}      \\
     & = T_{1, j-1} T_{2, j-1}X_0\Sigma_{1, 2}    \\
     & =\Sigma_{1, j-2}T_{1, j-1}X_0\Sigma_{1, 2} \\
     & = \Sigma_{1, j-2}T_{1, j},
  \end{align*}
  which is the desired equality.

  Next, we consider the case $i=2$ and $r=4$.
  Similar to the previous case, by Theorems \ref{theorem_TD4}, \ref{theorem_TD5}, Proposition \ref{proposition_commute_Sigmaii+1Tkl}, and the definition of $\Sigma_{2, j-2}$, we have
  \begin{align*}
    T_{2, j}T_{4, j}
     & =T_{2, j-1}X_1 \Sigma_{2, 3} T_{4, j}        \\
     & = T_{2, j-1}X_1 T_{4, j}\Sigma_{2, 3}        \\
     & = T_{2, j-1} T_{3, j-1}X_1 \Sigma_{2, 3}     \\
     & = \Sigma_{2, j-2}T_{2, j-1}X_1 \Sigma_{2, 3} \\
     & = \Sigma_{2, j-2}T_{2, j},
  \end{align*}
  which is the desired equality.

  The cases $i=1$ and $i=2$ above, together with Lemma \ref{lemma_TS_conj_x0}, imply that
  \begin{align}
    T_{i, j}T_{i+2, j}=\Sigma_{i, j-2}T_{i, j} \label{eq:secTB6_s=0}
  \end{align}
  holds for any $1 \leq i$ and $j>i+2$.

  Finally, fix $1 \leq i <r<j$ with $r-i$ even.
  We show that
  \begin{align*}
    T_{i, j}T_{r, j}=\Sigma_{i, i+j-r}T_{i, j}
  \end{align*}
  holds.
  Since $r-i$ is even, there exists an integer $s \geq 0$ such that $r=i+2+2s$ holds.
  If $s=0$, then this equality is equation \eqref{eq:secTB6_s=0}.
  Assume $s>0$ and let $j^\prime \coloneqq j-2s$.
  Since $j^\prime > i+2$, again by equation \eqref{eq:secTB6_s=0}, we have
  \begin{align*}
    T_{i, j^\prime}T_{i+2, j^\prime}=\Sigma_{i, j^\prime-2}T_{i, j^\prime}=\Sigma_{i, i+j-r}T_{i, j^\prime}.
  \end{align*}
  By taking the inverse of both sides, we have
  \begin{align*}
    T_{i, j^\prime}\Sigma_{i, i+j-r}=T_{i+2, j^\prime}T_{i, j^\prime}.
  \end{align*}
  By applying Lemma \ref{lemma_TB6_inductive_step} $s$ times to this equality, we have
  \begin{align*}
    T_{i, j}\Sigma_{i, i+j-r}=T_{i+2+2s, j}T_{i, j}=T_{r, j}T_{i, j},
  \end{align*}
  which is equivalent to $T_{i, j}T_{r, j}=\Sigma_{i, i+j-r}T_{i, j}$.
\end{proof}

We now combine the previous results to prove the following:
\begin{theorem} \label{theorem_TB6}
  For any $1 \leq i <l < j$, we have
  \begin{align*}
    T_{i, j}\Sigma_{i, l}=T_{i+j-l, j}T_{i, j}.  \tag{\relref{TB6}}
  \end{align*}
\end{theorem}
\begin{proof}
  If $j-l$ is odd, then it follows from Lemma \ref{lemma_TB6_odd_case}.
  If $j-l$ is even, then let $r \coloneqq i+j-l$.
  Since $1 \leq i < r < j$ holds and $r-i=j-l$ is even, by Proposition \ref{proposition_auxiliary_TB6}, we have
  \begin{align*}
    T_{i, j}T_{r, j}=\Sigma_{i, i+j-r}T_{i, j}.
  \end{align*}
  Since $i+j-r=l$, this becomes
  \begin{align*}
    T_{i, j}T_{i+j-l, j}=\Sigma_{i, l}T_{i, j}.
  \end{align*}
  Taking the inverse of both sides, we have
  \begin{align*}
    T_{i, j} \Sigma_{i, l}=T_{i+j-l, j}T_{i, j},
  \end{align*}
  which completes the proof.
\end{proof}
\subsection[Propagation of relation (TB7)]{Propagation of \relref{TB7}}
In this subsection, we show that
\begin{align*}
  T_{i, j}\Sigma_{k, l}=\Sigma_{i+j-l, i+j-k}T_{i, j}
\end{align*}
holds for any $1 \leq i<k<l<j$.
We prove this by first proving another equation.
\begin{lemma} \label{lemma_TB7_equivalent_relations}
  For any $i \geq 1$ and $j \geq i+2$, the following are equivalent$:$
  \begin{enumerate}
    \item $\Sigma_{i, j}\Sigma_{i, j-1}\Sigma_{i, j}=\Sigma_{i+1, j}$ holds$;$
    \item $T_{i, j+1}\Sigma_{i+1, j}=\Sigma_{i+1, j}T_{i, j+1}$ holds.
  \end{enumerate}
\end{lemma}
\begin{proof}
  We first note that
  \begin{align*}
    \Sigma_{i, j}\Sigma_{i, j-1}\Sigma_{i, j}=T_{i, j+1}\Sigma_{i+1, j}T_{i, j+1}
  \end{align*}
  holds.
  Indeed, by Theorem \ref{theorem_TB6}, we have
  \begin{align*}
    T_{i, j+1}\Sigma_{i, j-1}=T_{i+2, j+1}T_{i, j+1}.
  \end{align*}
  Since this equality is equivalent to $T_{i, j+1}\Sigma_{i, j-1}T_{i, j+1}=T_{i+2, j+1}$, by the definition of $\Sigma_{i, j}$ and $\Sigma_{i+1, j}$, we have
  \begin{align*}
    \Sigma_{i, j}\Sigma_{i, j-1}\Sigma_{i, j}
     & =T_{i, j+1}T_{i+1, j+1}T_{i, j+1}\Sigma_{i, j-1}T_{i, j+1}T_{i+1, j+1}T_{i, j+1} \\
     & =T_{i, j+1}T_{i+1, j+1}T_{i+2, j+1}T_{i+1, j+1}T_{i, j+1}                        \\
     & =T_{i, j+1}\Sigma_{i+1, j}T_{i, j+1}.
  \end{align*}

  Assume that (1) holds.
  Then, by the above equality, we have
  \begin{align*}
    T_{i, j+1}\Sigma_{i+1, j}T_{i, j+1}=\Sigma_{i, j}\Sigma_{i, j-1}\Sigma_{i, j}=\Sigma_{i+1, j},
  \end{align*}
  which is equivalent to $T_{i, j+1}\Sigma_{i+1, j}=\Sigma_{i+1, j}T_{i, j+1}$.

  Conversely, assume that (2) holds.
  Then we have
  \begin{align*}
    \Sigma_{i+1, j}=T_{i, j+1}\Sigma_{i+1, j}T_{i, j+1}=\Sigma_{i, j}\Sigma_{i, j-1}\Sigma_{i, j},
  \end{align*}
  which is (1).
\end{proof}

Next, we prove the first condition of Lemma \ref{lemma_TB7_equivalent_relations}.
\begin{proposition} \label{proposition_TB7_induction2}
  For any $i \geq 1$ and $j \geq i+2$, we have
  \begin{align*}
    \Sigma_{i, j}\Sigma_{i, j-1}\Sigma_{i, j}=\Sigma_{i+1, j}.
  \end{align*}
\end{proposition}
\begin{proof}
  We first show that
  \begin{align*}
    \Sigma_{i, i+2}\Sigma_{i, i+1}\Sigma_{i, i+2}=\Sigma_{i+1, i+2}
  \end{align*}
  holds for any $i \geq 1$.
  If $i=1$ or $2$, then it follows from \relref{TB3} in $\cR$.
  For the case of $i \geq 3$, by applying Lemma \ref{lemma_TS_conj_x0} $i-2$ times to $\Sigma_{2, 4}\Sigma_{2, 3}\Sigma_{2, 4}=\Sigma_{3, 4}$, which is the case of $i=2$, we obtain the desired equality.

  Next, we see that $T_{i, j}\Sigma_{i+1, i+2}=\Sigma_{j-2, j-1}T_{i, j}$ holds for any $i \geq 1$ and $j > i+2$.
  Indeed, by the above equality and Theorem \ref{theorem_TB6}, we have
  \begin{align}
    T_{i, j}\Sigma_{i+1, i+2}T_{i, j}
     & =T_{i, j}\Sigma_{i, i+2} \Sigma_{i, i+1}\Sigma_{i, i+2}T_{i, j} \notag                                      \\
     & =(T_{i, j}\Sigma_{i, i+2}T_{i, j})(T_{i, j}\Sigma_{i, i+1}T_{i, j})(T_{i, j}\Sigma_{i, i+2}T_{i, j}) \notag \\
     & =T_{j-2, j}T_{j-1, j}T_{j-2, j} \notag                                                                      \\
     & =\Sigma_{j-2, j-1}, \label{eq:TB7_induction1}
  \end{align}
  which is equivalent to the desired equality.

  Finally, we show that $T_{i, j+1}\Sigma_{i+1, j}=\Sigma_{i+1, j}T_{i, j+1}$ holds for any $i \geq 1$ and $j \geq i+2$ by induction on $d \coloneqq j-i \geq 2$.
  If $d=2$, then $T_{i, i+3}\Sigma_{i+1, i+2}=\Sigma_{i+1, i+2}T_{i, i+3}$ follows from equation \eqref{eq:TB7_induction1} with $j=i+3$.

  Assume that $d \geq 3$ and
  \begin{align*}
    T_{i, j}\Sigma_{i+1, j-1}=\Sigma_{i+1, j-1}T_{i, j}
  \end{align*}
  holds.
  Then we have $T_{i, j+1}\Sigma_{i+1, j}=\Sigma_{i+1, j}T_{i, j+1}$.
  Indeed, by Proposition \ref{proposition_adjacent_TB3}, Theorems \ref{theorem_TD3}, \ref{theorem_TC2}, the inductive hypothesis, and \eqref{eq:TB7_induction1}, we have
  \begin{align*}
    T_{i, j+1}\Sigma_{i+1, j}
     & = T_{i, j+1}\Sigma_{i+1, j}\Sigma_{i+1, i+2} \Sigma_{i+1, i+2}           \\
     & = T_{i, j+1}\Sigma_{j-1, j}\Sigma_{i+1, j} \Sigma_{i+1, i+2}             \\
     & = X_i^{-1}X_i T_{i, j+1}\Sigma_{j-1, j}\Sigma_{i+1, j} \Sigma_{i+1, i+2} \\
     & = X_i^{-1} T_{i, j}X_{j-2}\Sigma_{i+1, j}\Sigma_{i+1, i+2}               \\
     & = X_i^{-1} T_{i, j}\Sigma_{i+1, j-1}X_i                                  \\
     & = X_i^{-1} \Sigma_{i+1, j-1}T_{i, j} X_i                                 \\
     & = X_i^{-1} \Sigma_{i+1, j-1} X_{j-2}T_{i, j+1}\Sigma_{i+1, i+2}          \\
     & = X_i^{-1} X_i \Sigma_{i+1, j}\Sigma_{j-1, j}T_{i, j+1}\Sigma_{i+1, i+2} \\
     & = \Sigma_{i+1, j}\Sigma_{j-1, j}\Sigma_{j-1, j}T_{i, j+1}                \\
     & = \Sigma_{i+1, j}T_{i, j+1}.
  \end{align*}
  Hence, condition (2) of Lemma \ref{lemma_TB7_equivalent_relations} holds for all $i \geq 1$ and $j \geq i+2$.
  Therefore condition (1) also holds for all such $i$ and $j$.
  This completes the proof.
\end{proof}

We now use the previous proposition to prove \relref{TB7}.
\begin{theorem}\label{theorem_TB7}
  For any $1 \leq i < k<l<j$, we have
  \begin{align*}
    T_{i, j}\Sigma_{k, l}=\Sigma_{i+j-l, i+j-k}T_{i, j}. \tag{\relref{TB7}}
  \end{align*}
\end{theorem}
\begin{proof}
  We show this by induction on $d \coloneqq k-i$.
  First, assume that $d=1$.
  By Proposition \ref{proposition_TB7_induction2}, Theorem \ref{theorem_TB6}, and the definition of $\Sigma_{i+j-l, j-1}$, we have
  \begin{align*}
    T_{i, j}\Sigma_{i+1, l}T_{i, j}
     & = T_{i, j}\Sigma_{i, l}\Sigma_{i, l-1}\Sigma_{i, l}T_{i, j}                                 \\
     & = (T_{i, j}\Sigma_{i, l}T_{i, j})(T_{i, j}\Sigma_{i, l-1}T_{i, j})(T_{i, j}\Sigma_{i, l}T_{i, j}) \\
     & = T_{i+j-l, j}T_{i+j-l+1, j}T_{i+j-l, j}                                                    \\
     & = \Sigma_{i+j-l, j-1},
  \end{align*}
  which gives the desired equality.

  Assume that $d>1$ and the equality holds for all indices with $k-i=d-1$.
  Namely, we assume that
  \begin{align*}
    T_{i, j}\Sigma_{k-1, l}
     & =\Sigma_{i+j-l, i+j-(k-1)}T_{i, j},    \\
    T_{i, j}\Sigma_{k-1, l-1}
     & =\Sigma_{i+j-(l-1), i+j-(k-1)}T_{i, j}
  \end{align*}
  hold.
  Then by Proposition \ref{proposition_TB7_induction2} and the induction hypothesis, we have
  \begin{align*}
    T_{i, j}\Sigma_{k, l}
     & = T_{i, j} \Sigma_{k-1, l}\Sigma_{k-1, l-1}\Sigma_{k-1, l}                                \\
     & = \Sigma_{i+j-l, i+j-(k-1)}T_{i, j}\Sigma_{k-1, l-1}\Sigma_{k-1, l}                       \\
     & = \Sigma_{i+j-l, i+j-(k-1)}\Sigma_{i+j-(l-1), i+j-(k-1)}T_{i, j}\Sigma_{k-1, l}           \\
     & = \Sigma_{i+j-l, i+j-(k-1)}\Sigma_{i+j-(l-1), i+j-(k-1)}\Sigma_{i+j-l, i+j-(k-1)}T_{i, j} \\
     & = \Sigma_{i+j-l, i+j-k}T_{i, j},
  \end{align*}
  which completes the proof.
\end{proof}
\subsection[Propagation of relation (TB3)]{Propagation of \relref{TB3}}
We divide the proof into several lemmas according to the indices.
\begin{lemma}\label{lemma_TB3_inner}
  For any $1 \leq i < k <l <j$, we have
  \begin{align*}
    \Sigma_{i, j}\Sigma_{k, l}=\Sigma_{i+j-l, i+j-k}\Sigma_{i, j}.
  \end{align*}
\end{lemma}
\begin{proof}
  By the definition of $\Sigma_{i, j}$ and Theorem \ref{theorem_TB7}, we have
  \begin{align*}
    \Sigma_{i, j}\Sigma_{k, l}
     & = T_{i, j+1}T_{i+1, j+1}T_{i, j+1} \Sigma_{k, l}                                             \\
     & = T_{i, j+1}T_{i+1, j+1}\Sigma_{i+(j+1)-l, i+(j+1)-k}T_{i, j+1}                              \\
     & = T_{i, j+1} \Sigma_{(i+1)+(j+1)-(i+(j+1)-k), (i+1)+(j+1)-(i+(j+1)-l)}T_{i+1, j+1}T_{i, j+1} \\
     & = T_{i, j+1} \Sigma_{k+1, l+1}T_{i+1, j+1}T_{i, j+1}                                         \\
     & = \Sigma_{i+(j+1)-(l+1), i+(j+1)-(k+1)}T_{i, j+1}T_{i+1, j+1}T_{i, j+1}                      \\
     & = \Sigma_{i+j-l, i+j-k}\Sigma_{i, j},
  \end{align*}
  which completes the proof.
\end{proof}
\begin{lemma} \label{lemma_TB3_left}
  For any $1 \leq i < l < j$, we have
  \begin{align*}
    \Sigma_{i, j}\Sigma_{i, l}=\Sigma_{i+j-l, j}\Sigma_{i, j}.
  \end{align*}
\end{lemma}
\begin{proof}
  By the definition of $\Sigma_{i,j}$, Theorem \ref{theorem_TB6}, and Theorem \ref{theorem_TB7}, we have
  \begin{align}
    \Sigma_{i, j}\Sigma_{i, l}
     & = T_{i,j+1}T_{i+1, j+1}T_{i, j+1}\Sigma_{i, l}  \notag                                  \\
     & = T_{i, j+1}T_{i+1, j+1}T_{i+j+1-l, j+1}T_{i, j+1} \notag                               \\
     & = T_{i, j+1}\Sigma_{i+1, l+1}T_{i+1, j+1}T_{i, j+1}      \label{eq:TB3_usingTB6inverse} \\
     & = \Sigma_{i+j+1-(l+1), i+j+1-(i+1)}T_{i, j+1}T_{i+1, j+1}T_{i, j+1} \notag              \\
     & = \Sigma_{i+j-l, j}\Sigma_{i, j}, \notag
  \end{align}
  where \eqref{eq:TB3_usingTB6inverse} follows from the inverse of both sides of the equation in Theorem \ref{theorem_TB6} with $(i, l, j)=(i+1, l+1, j+1)$.
  This completes the proof.
\end{proof}

\begin{lemma}\label{lemma_TB3_right}
  For any $1 \leq i < k < j$, we have
  \begin{align*}
    \Sigma_{i, j}\Sigma_{k, j}=\Sigma_{i, i+j-k}\Sigma_{i, j}.
  \end{align*}
\end{lemma}
\begin{proof}
  By Lemma \ref{lemma_TB3_left} with $l=i+j-k$, we have
  \begin{align*}
    \Sigma_{i, j}\Sigma_{i, i+j-k}=\Sigma_{k, j}\Sigma_{i, j}.
  \end{align*}
  Taking the inverse of both sides gives
  \begin{align*}
    \Sigma_{i, j}\Sigma_{k, j}=\Sigma_{i, i+j-k}\Sigma_{i, j},
  \end{align*}
  as desired.
\end{proof}

\begin{theorem} \label{theorem_TB3}
  For any $1 \leq i \leq k < l \leq j$, we have
  \begin{align*}
    \Sigma_{i, j}\Sigma_{k, l}=\Sigma_{i+j-l, i+j-k}\Sigma_{i, j}. \tag{\relref{TB3}}
  \end{align*}
\end{theorem}
\begin{proof}
  If $i<k<l<j$, then the assertion follows from Lemma \ref{lemma_TB3_inner}.
  If $i=k<l<j$, then it follows from Lemma \ref{lemma_TB3_left}.
  If $i<k<l=j$, then it follows from Lemma \ref{lemma_TB3_right}.
  Finally, if $i=k$ and $l=j$, then it is trivial since both sides are equal to $1$.
\end{proof}
\subsection[Propagation of relation (TD1)]{Propagation of \relref{TD1}} \label{subsection_lastrelation}
In this subsection, we show that
\begin{align*}
  T_{i, j}X_k=X_{i-1}X_i \cdots X_{k-j+i-1}X_{k-j+i}^2T_{i, k+3}T_{j, k+3}
\end{align*}
holds for any $1 \leq i <j$ with $k > j-1$.
We first prove an auxiliary equality, and then use it to obtain the full propagation of \relref{TD1}.

\begin{lemma}\label{lemma_auxiliary_TD1}
  For any $1 \leq i <j$ and $s \geq -1$, we have
  \begin{align*}
    X_{i+s}T_{i, j+s+5}T_{j, j+s+5}=T_{i, j+s+4}T_{j, j+s+4}X_{j+s}.
  \end{align*}
\end{lemma}
\begin{proof}
  We first show the case $s=-1$.
  Note that $X_{i-1}=T_{i, j+3}T_{i, j+4}\Sigma_{i, i+1}$ holds by Theorem \ref{theorem_TD4}.
  Then by Theorems \ref{theorem_TB6} and \ref{theorem_TD4}, we have
  \begin{align*}
    X_{i-1}T_{i, j+4}T_{j, j+4}
     & =T_{i, j+3}T_{i, j+4}\Sigma_{i, i+1}T_{i, j+4}T_{j, j+4}      \\
     & =T_{i, j+3}T_{i+j+4-(i+1), j+4}T_{i, j+4}T_{i, j+4}T_{j, j+4} \\
     & =T_{i, j+3}T_{j+3, j+4}T_{j, j+4}                             \\
     & =T_{i, j+3}T_{j, j+4}\Sigma_{j, j+1}                          \\
     & =T_{i, j+3}T_{j, j+3}X_{j-1},
  \end{align*}
  which is the desired equality.

  Next, assume that $s \geq 0$.
  By Theorems \ref{theorem_TD3} and \ref{theorem_TB7}, we have
  \begin{align*}
    T_{i, j+s+4}T_{j, j+s+4}X_{j+s}
     & = T_{i, j+s+4} X_{j+j+s+4-(j+s)-2} T_{j, j+s+5}\Sigma_{j+s+1, j+s+2}                \\
     & = T_{i, j+s+4} X_{j+2}T_{j, j+s+5}\Sigma_{j+s+1, j+s+2}                             \\
     & = X_{i+j+s+4-(j+2)-2}T_{i, j+s+5}\Sigma_{j+3, j+4}T_{j, j+s+5}\Sigma_{j+s+1, j+s+2} \\
     & = X_{i+s}T_{i, j+s+5}T_{j, j+s+5},
  \end{align*}
  which completes the proof.
\end{proof}
Next, for each $s \geq -1$, we define a word $P_s$ as follows:
\begin{align*}
  P_s
  \coloneqq
  \begin{cases}
    X_{i-1}^2                            & \text{if } s=-1,    \\
    X_{i-1}X_i^2                         & \text{if } s=0,     \\
    X_{i-1}X_i \cdots X_{i+s-1}X_{i+s}^2 & \text{if } s\geq 1.
  \end{cases}
\end{align*}
The following fact is easy to check.
\begin{lemma} \label{lemma_TD1_Ps}
  For any $s \geq -1$, we have
  \begin{align*}
    P_sX_{i+s+2}^2=P_{s+1}X_{i+s}.
  \end{align*}
\end{lemma}
\begin{proof} This follows directly from the standard relations of Thompson's group $F$, and we leave the details to the reader.
\end{proof}

We now turn to the last remaining propagation, namely that of \relref{TD1}.

\begin{theorem} \label{theorem_TD1}
  For any $1 \leq i < j$ with $k > j-1$, we have
  \begin{align*}
    T_{i, j}X_k=X_{i-1}X_i \cdots X_{k-j+i-1}X_{k-j+i}^2T_{i, k+3}T_{j, k+3}.  \tag{\relref{TD1}}
  \end{align*}
\end{theorem}
\begin{proof}
  We first prove the following slightly stronger claim for every integer $s \geq -1$ by induction on $s$:
  \begin{align*}
    T_{i, j}X_{j+s}=P_sT_{i, j+s+3}T_{j, j+s+3}.
  \end{align*}
  If $s=-1$, then the claim follows directly from Theorem \ref{theorem_TD2}.

  Assume that $T_{i, j}X_{j+s}=P_sT_{i, j+s+3}T_{j, j+s+3}$ holds for some $s \geq -1$.
  By the inductive hypothesis, Theorems \ref{theorem_TD2}, \ref{theorem_TD3}, \ref{theorem_TC2}, \ref{theorem_TB3}, \ref{theorem_TB6}, and Lemmas \ref{lemma_TD1_Ps}, \ref{lemma_auxiliary_TD1}, we have
  \begin{align*}
    T_{i, j}X_{j+s+1}
     & = T_{i, j} X_{j+s}X_{j+s+2}X_{j+s}^{-1}                                                                                                \\
     & = P_sT_{i, j+s+3}T_{j, j+s+3} X_{j+s+2}X_{j+s}^{-1}                                                                                    \\
     & = P_s T_{i, j+s+3}X_{j-1}^2T_{j, j+s+5}T_{j+s+3, j+s+5} X_{j+s}^{-1}                                                                   \\
     & = P_s X_{i+j+s+3-(j-1)-2}T_{i, j+s+4}\Sigma_{j, j+1} X_{j-1}T_{j, j+s+5}T_{j+s+3, j+s+5} X_{j+s}^{-1}                                  \\
     & = P_s X_{i+s+2}T_{i, j+s+4} X_{j+j+1-(j-1)-2} \Sigma_{j, j+2}\Sigma_{j, j+1}T_{j, j+s+5}T_{j+s+3, j+s+5} X_{j+s}^{-1}                  \\
     & = P_s X_{i+s+2}T_{i, j+s+4} X_j \Sigma_{j, j+2}\Sigma_{j, j+1}T_{j, j+s+5}T_{j+s+3, j+s+5} X_{j+s}^{-1}                                \\
     & = P_s X_{i+s+2} X_{i+j+s+4-j-2} T_{i, j+s+5} \Sigma_{j+1, j+2} \Sigma_{j, j+2}\Sigma_{j, j+1}T_{j, j+s+5}T_{j+s+3, j+s+5} X_{j+s}^{-1} \\
     & = P_s X_{i+s+2}^2 T_{i, j+s+5} \Sigma_{j+1, j+2} \Sigma_{j, j+2}\Sigma_{j, j+1}T_{j, j+s+5}T_{j+s+3, j+s+5} X_{j+s}^{-1}               \\
     & = P_s X_{i+s+2}^2 T_{i, j+s+5} \Sigma_{j, j+2} \Sigma_{j, j+1}  \Sigma_{j, j+1}T_{j, j+s+5}T_{j+s+3, j+s+5} X_{j+s}^{-1}               \\
     & = P_s X_{i+s+2}^2 T_{i, j+s+5} \Sigma_{j, j+2} T_{j, j+s+5}T_{j+s+3, j+s+5} X_{j+s}^{-1}                                               \\
     & = P_s X_{i+s+2}^2 T_{i, j+s+5} T_{j, j+s+5} T_{j+s+3, j+s+5} T_{j+s+3, j+s+5} X_{j+s}^{-1}                                               \\
     & = P_s X_{i+s+2}^2 T_{i, j+s+5} T_{j, j+s+5}  X_{j+s}^{-1}                                                                              \\
     & = P_{s+1}X_{i+s} T_{i, j+s+5} T_{j, j+s+5}  X_{j+s}^{-1}                                                                               \\
     & = P_{s+1}T_{i, j+s+4}T_{j, j+s+4}X_{j+s}  X_{j+s}^{-1}                                                                                 \\
     & = P_{s+1}T_{i, j+s+4}T_{j, j+s+4},
  \end{align*}
  which proves the induction step.

  Taking $s=k-j$ and using the definition of $P_s$, we obtain
  \begin{align*}
    T_{i, j} X_k = X_{i-1} X_i \cdots X_{k-j+i-1}X_{k-j+i}^2 T_{i, k+3}T_{j, k+3},
  \end{align*}
  which completes the proof.
\end{proof}
\subsection{Proof of Theorem \ref{theorem_finitely_presented} and the required relations in the proof} \label{subsection_minimal_relations}
By the propagation results proved in the previous subsections, the words $X_i$, $T_{i,j}$, and $\Sigma_{i,j}$ satisfy all the defining relations of $CV_T$ for every admissible choice of indices.
Therefore, the map $\varphi\colon CV_T\to CV_7$ defined in Subsection \ref{subsection_finitely_presented_group_and_map} is a well-defined homomorphism.

\begin{proof}[Proof of Theorem \ref{theorem_finitely_presented}]
  Define a map
  \begin{align*}
    \theta\colon CV_7 & \longrightarrow CV_T
  \end{align*}
  by
  \begin{align*}
    \theta(X_0)= x_0, \quad
    \theta(X_1)= x_1, \quad
    \theta(T_{i,j})= \tau_{i,j}
    \quad
    (1\leq i<j\leq 7).
  \end{align*}
  By \relref{TA}, $CV_T$ satisfies the two relations for $X_0$ and $X_1$.
  Moreover, each relation in $\cR$ is a defining relation of $CV_T$ with indices at most $7$.
  Hence, $\theta$ is a well-defined homomorphism.

  We next observe the images of the words under $\theta$.
  By the definition of $X_i$ and \relref{TA}, we have $\theta(X_i)=x_i$ for every $i\geq 0$.
  We also have $\theta(T_{i,j})=\tau_{i,j}$ for every $1\leq i<j$.
  Indeed, for $i=1$, \relref{TD4} and \relref{TB1} give $\tau_{1,j+1}=\tau_{1,j}x_0\widehat{\sigma}_{1,2}$, and hence the definition of $T_{1,j}$ gives
  \begin{align*}
    \theta(T_{1,j})
    =
    \tau_{1,2}(x_0\widehat{\sigma}_{1,2})^{j-2}
    =
    \tau_{1, j}.
  \end{align*}
  Similarly, \relref{TD4} gives $\tau_{2,m}=\tau_{2,3}(x_1\widehat{\sigma}_{2,3})^{m-3}$ for every $m\geq 3$.
  Since \relref{TD5} with $k=0$ gives $x_0^{-1}\tau_{i,j}x_0=\tau_{i+1,j+1}$, for every $2\leq i<j$, we have
  \begin{align*}
    \theta(T_{i,j})
    =
    x_0^{-(i-2)}\tau_{2,3}(x_1\widehat{\sigma}_{2,3})^{j-i-1}x_0^{i-2}
    =
    x_0^{-(i-2)}\tau_{2,j-i+2}x_0^{i-2}
    =
    \tau_{i,j}.
  \end{align*}
  It follows that
  \begin{align*}
    (\theta\circ\varphi)(x_i)
    =
    \theta(X_i)
    =
    x_i, \quad
    (\theta\circ\varphi)(\tau_{i,j})
    =
    \theta(T_{i,j})
    =
    \tau_{i,j}
  \end{align*}
  holds for all $i\geq0$ and $1\leq i<j$.
  Since the elements $x_i$ and $\tau_{i,j}$ generate $CV_T$, we obtain $\theta\circ\varphi=\mathrm{id}_{CV_T}$.

  Conversely, by the above observation about $\theta$, we have
  \begin{align*}
    (\varphi\circ\theta)(X_0)
    =
    \varphi(x_0)
    =
    X_0, \quad
    (\varphi\circ\theta)(X_1)
    =
    \varphi(x_1)
    =
    X_1.
  \end{align*}
  Moreover, for $1\leq i<j\leq7$, we have
  \begin{align*}
    (\varphi\circ\theta)(T_{i,j})
    =
    \varphi(\tau_{i,j})
    =
    T_{i,j}.
  \end{align*}
  Since $X_0$, $X_1$, and the elements $T_{i,j}$ with
  $1\leq i<j\leq7$ generate $CV_7$, we obtain $\varphi\circ\theta=\mathrm{id}_{CV_7}$.

  Thus, $\varphi$ and $\theta$ are mutually inverse, and $CV_T\cong CV_7$.
  Combining this with Proposition \ref{Proposition_CV_CVT}, we obtain $CV \cong CV_7$.
  Since $CV_7$ is finitely presented, so is $CV$.
\end{proof}

We finish this subsection by recording the finite relations that are actually used in the propagation arguments.
The following list (Table \ref{table_used_finite_relations}) is not intended to be minimal.
Its purpose is only to make explicit which finite relations are needed in the proof above.
Among the relations from $\cR$ listed below, those that are not explicitly cited as elements of $\cR$ in the propagation arguments are used to verify that, for indices at most $7$, the representative words $T_{i,j}$ coincide with the corresponding generators of $CV_7$.
We believe that this list can be further reduced, but we leave this question for future research.

\section*{Acknowledgements}
The first author is partially supported by JSPS KAKENHI Grant numbers 24K22836 and 26K16986.
The second author is partially supported by JSPS KAKENHI Grant number 24KJ0144.

\begin{table}[p]
  \centering
  \renewcommand{\arraystretch}{1.25}
  \begin{tabular}{
      p{0.18\linewidth}
      >{\centering\arraybackslash}p{0.07\linewidth}
      p{0.69\linewidth}
    }
    \hline
    Relation family & Count & Instances used \\
    \hline

    \relref{TA}
                    & 2
                    &
    \(\begin{aligned}[t]
         & [X_0X_1^{-1},X_0^{-1}X_1X_0]=1,\quad
                                          [X_0X_1^{-1},X_0^{-2}X_1X_0^2]=1.
      \end{aligned}\)
    \\

    \relref{TB1}
                    & 5
                    &
    \(\begin{aligned}[t]
         & T_{1,2}^2=1,\quad
        T_{1,3}^2=1,\quad
        T_{2,3}^2=1,\quad
        T_{2,4}^2=1,\quad
        T_{3,4}^2=1.
      \end{aligned}\)
    \\

    \relref{TB2}
                    & 4
                    &
    \(\begin{aligned}[t]
         & \Sigma_{1,2}\Sigma_{3,4}
        =\Sigma_{3,4}\Sigma_{1,2},\quad
        \Sigma_{1,2}\Sigma_{4,5}
        =\Sigma_{4,5}\Sigma_{1,2},  \\
         & \Sigma_{2,3}\Sigma_{4,5}
        =\Sigma_{4,5}\Sigma_{2,3},\quad
        \Sigma_{2,3}\Sigma_{5,6}
        =\Sigma_{5,6}\Sigma_{2,3}.
      \end{aligned}\)
    \\

    \relref{TB3}
                    & 2
                    &
    \(\begin{aligned}[t]
         & \Sigma_{1,3}\Sigma_{1,2}
        =\Sigma_{2,3}\Sigma_{1,3},\quad
        \Sigma_{2,4}\Sigma_{2,3}
        =\Sigma_{3,4}\Sigma_{2,4}.
      \end{aligned}\)
    \\

    \relref{TB5}
                    & 4
                    &
    \(\begin{aligned}[t]
         & \Sigma_{2,3}T_{4,6}
        =T_{4,6}\Sigma_{2,3},\quad
        \Sigma_{2,3}T_{5,7}
        =T_{5,7}\Sigma_{2,3},  \\
         & \Sigma_{1,2}T_{3,5}
        =T_{3,5}\Sigma_{1,2},\quad
        \Sigma_{1,2}T_{4,6}
        =T_{4,6}\Sigma_{1,2}.
      \end{aligned}\)
    \\

    \relref{TB6}
                    & 4
                    &
    \(\begin{aligned}[t]
         & T_{1,3}\Sigma_{1,2}
        =T_{2,3}T_{1,3},\quad
        T_{1,4}\Sigma_{1,2}
        =T_{3,4}T_{1,4},       \\
         & T_{2,4}\Sigma_{2,3}
        =T_{3,4}T_{2,4},\quad
        T_{2,5}\Sigma_{2,3}
        =T_{4,5}T_{2,5}.
      \end{aligned}\)
    \\

    \relref{TB7}
                    & 4
                    &
    \(\begin{aligned}[t]
         & T_{1,4}\Sigma_{2,3}
        =\Sigma_{2,3}T_{1,4},\quad
        T_{1,5}\Sigma_{2,3}
        =\Sigma_{3,4}T_{1,5},  \\
         & T_{2,5}\Sigma_{3,4}
        =\Sigma_{3,4}T_{2,5},\quad
        T_{2,6}\Sigma_{3,4}
        =\Sigma_{4,5}T_{2,6}.
      \end{aligned}\)
    \\

    \relref{TC1}
                    & 4
                    &
    \(\begin{aligned}[t]
         & \Sigma_{2,3}X_3
        =X_3\Sigma_{2,3},\quad
        \Sigma_{2,3}X_4
        =X_4\Sigma_{2,3},  \\
         & \Sigma_{1,2}X_2
        =X_2\Sigma_{1,2},\quad
        \Sigma_{1,2}X_3
        =X_3\Sigma_{1,2}.
      \end{aligned}\)
    \\

    \relref{TC3}
                    & 2
                    &
    \(\begin{aligned}[t]
         & \Sigma_{3,4}X_0
        =X_0\Sigma_{4,5},\quad
        \Sigma_{3,4}X_1
        =X_1\Sigma_{4,5}.
      \end{aligned}\)
    \\

    \relref{TD2}
                    & 2
                    &
    \(\begin{aligned}[t]
         & T_{2,3}X_2
        =X_1^2T_{2,5}T_{3,5},\quad
        T_{1,2}X_1
        =X_0^2T_{1,4}T_{2,4}.
      \end{aligned}\)
    \\

    \relref{TD3}
                    & 4
                    &
    \(\begin{aligned}[t]
         & T_{2,5}X_2
        =X_3T_{2,6}\Sigma_{3,4},\quad
        T_{2,4}X_2
        =X_2T_{2,5}\Sigma_{3,4}, \\
         & T_{1,3}X_1
        =X_1T_{1,4}\Sigma_{2,3},\quad
        T_{1,4}X_1
        =X_2T_{1,5}\Sigma_{2,3}.
      \end{aligned}\)
    \\

    \relref{TD4}
                    & 9
                    &
    \(\begin{aligned}[t]
         & T_{1,2}X_0=T_{1,3}\Sigma_{1,2},\quad
        T_{1,3}X_0=T_{1,4}\Sigma_{1,2},         \\
         & T_{1,4}X_0=T_{1,5}\Sigma_{1,2},\quad
        T_{1,5}X_0=T_{1,6}\Sigma_{1,2},         \\
         & T_{1,6}X_0=T_{1,7}\Sigma_{1,2},      \\
         & T_{2,3}X_1=T_{2,4}\Sigma_{2,3},\quad
        T_{2,4}X_1=T_{2,5}\Sigma_{2,3},         \\
         & T_{2,5}X_1=T_{2,6}\Sigma_{2,3},\quad
        T_{2,6}X_1=T_{2,7}\Sigma_{2,3}.
      \end{aligned}\)
    \\

    \relref{TD5}
                    & 14
                    &
    \(\begin{aligned}[t]
         & T_{2,3}X_0=X_0T_{3,4},\quad
        T_{2,4}X_0=X_0T_{3,5},         \\
         & T_{2,5}X_0=X_0T_{3,6},\quad
        T_{2,6}X_0=X_0T_{3,7},         \\
         & T_{3,4}X_0=X_0T_{4,5},\quad
        T_{3,5}X_0=X_0T_{4,6},         \\
         & T_{3,6}X_0=X_0T_{4,7},\quad
        T_{4,5}X_0=X_0T_{5,6},         \\
         & T_{4,6}X_0=X_0T_{5,7},\quad
        T_{5,6}X_0=X_0T_{6,7},         \\
         & T_{3,4}X_1=X_1T_{4,5},\quad
        T_{3, 5}X_1=X_1T_{4, 6},       \\
         & T_{4,5}X_1=X_1T_{5,6},\quad
        T_{4,6}X_1=X_1T_{5,7}.
      \end{aligned}\)
    \\

    \hline
    Total           & 60    &                \\
    \hline
  \end{tabular}
  \caption{Finite relations actually used in the proof.}
  \label{table_used_finite_relations}
\end{table}
\clearpage

\FloatBarrier
\bibliographystyle{plain}
\bibliography{references}
\bigskip

Yuya KODAMA

\address{
  Graduate School of Science and Engineering
  Kagoshima University
  1-21-35 Korimoto, Kagoshima
  Kagoshima 890-0065, Japan
}

\textit{E-mail address}: \href{mailto:yuya@sci.kagoshima-u.ac.jp}{\texttt{yuya@sci.kagoshima-u.ac.jp}}

\bigskip

Akihiro TAKANO

\address{Department of Mathematics, Graduate School of Science
  Osaka University
  1-1 Machikaneyama, Toyonaka
  Osaka 560-0043, Japan
}

\textit{E-mail address}: \href{mailto:takano.akihiro.sci@osaka-u.ac.jp}{\texttt{takano.akihiro.sci@osaka-u.ac.jp}}
\end{document}